\documentclass[10pt]{amsart}
\usepackage{array}
\usepackage{cancel}
\usepackage[british]{babel}
\usepackage{lmodern}
\usepackage{pifont,subfigure,graphicx}
\usepackage{caption}
\usepackage[colorlinks,citecolor=blue,urlcolor=blue]{hyperref}
\usepackage{url}
\usepackage{amsthm,amsmath,amssymb,amsfonts,mathrsfs,amsfonts,amstext,amsopn}
\usepackage{mathtools}
\usepackage{tikz}
\usepackage[version=3]{mhchem}
\usepackage{texdraw}
\usepackage{extarrows}
\usepackage{bbm}
\usepackage{enumitem}
\usepackage[all]{xy}

\usepackage{pgfplots}
\pgfplotsset{compat = newest}
\usetikzlibrary{decorations.fractals}
\usetikzlibrary{positioning}
\usetikzlibrary{arrows}
\numberwithin{equation}{section}
\theoremstyle{plain}
\newtheorem{theorem}{Theorem}[section]
\newtheorem{lemma}[theorem]{Lemma}

\newtheorem{proposition}[theorem]{Proposition}

\theoremstyle{definition}
\newtheorem{definition}[theorem]{Definition}
\newtheorem{example}[theorem]{Example}
\theoremstyle{remark}
\newtheorem{remark}[theorem]{Remark}

\newcounter{tmp}

\newcommand{\lt}{\left}
\newcommand{\rt}{\right}

\newcommand{\N}{\mathbb{N}}

\newcommand{\R}{\mathbb{R}}
\newcommand{\T}{\mathbb{T}}
\newcommand{\Z}{\mathbb{Z}}

\newcommand{\cA}{\mathcal{A}}
\newcommand{\cC}{\mathcal{C}}
\newcommand{\cL}{\mathcal{L}}
\newcommand{\cD}{\mathcal{D}}
\newcommand{\cM}{\mathcal{M}}
\newcommand{\cP}{\mathcal{P}}
\newcommand{\cS}{\mathcal{S}}
\newcommand{\cT}{\mathcal{T}}
\newcommand{\rd}{\mathrm{d}}

\newcommand{\supp}{{\rm supp\,}}
\newcommand{\Diam}{\textrm{Diam\,}}

\newcommand{\abs}{{\sf abs}}

\newcommand{\BL}{\mathcal{BL}}
\newcommand{\TV}{{\sf TV}}

\allowdisplaybreaks
\makeatletter
\newcommand{\thickhline}{%
    \noalign {\ifnum 0=`}\fi \hrule height 1pt
    \futurelet \reserved@a \@xhline
}
\newcolumntype{"}{@{\hskip\tabcolsep\vrule width 1pt\hskip\tabcolsep}}
\makeatother
\begin{document}
\title[VEs on DGMs]%[MFL of DS on Network Sequences]
  {Vlasov Equations on Digraph Measures}%{Mean Field Limit of Dynamical Systems on Network Sequences}
\author{Christian Kuehn}
\author{Chuang Xu}
\address{
Faculty of Mathematics\\
Technical University of Munich, Munich\\
Garching bei M\"{u}nchen\\
85748, Germany.}
\email{ckuehn@ma.tum.de}
\email{xuc@ma.tum.de}
%\subjclass[2020]{Primary ; secondary , }

\date{\today}

\noindent

\begin{abstract}
Many science phenomena are described as interacting particle systems (IPS). The mean field limit (MFL) of large all-to-all coupled deterministic IPS is given by the solution of a PDE, the \emph{Vlasov Equation (VE)}. Yet, many applications demand IPS coupled on networks/graphs. In this paper, we are interested in IPS on a sequence of directed graphs, or digraphs for short. It is interesting to know, how the limit of a sequence of digraphs associated with the IPS influences the macroscopic MFL. This paper studies VEs on a generalized digraph, regarded as limit of a sequence of digraphs, which we refer to as a \emph{digraph measure} (DGM) to emphasize that we work with its limit via measures. We provide (i) unique existence of solutions of the VE on continuous DGMs, and (ii) discretization of the solution of the VE by empirical distributions supported on solutions of an IPS via ODEs coupled on a sequence of digraphs converging to the given DGM. Our result extends existing results on one-dimensional Kuramoto-type networks coupled on dense graphs. Here we allow the underlying digraphs to be \emph{not necessarily dense} which include many interesting graphical structures such that stars, trees and rings, which have been frequently used in many sparse network models in finance, telecommunications, physics, genetics, neuroscience, and social sciences. A key contribution of this paper is a nontrivial generalization of Neunzert's in-cell-particle approach for all-to-all coupled \emph{indistinguishable} IPS with \emph{global} Lipschitz continuity in Euclidean spaces to distinguishable IPS on heterogeneous digraphs with local Lipschitz continuity, via a \emph{measure-theoretic viewpoint}. The approach together with the metrics is different from the known techniques in $L^p$-functions using graphons and their generalization via harmonic analysis of locally compact Abelian groups. Finally, to demonstrate the wide applicability, we apply our results to various models in higher-dimensional Euclidean spaces in epidemiology, ecology, and social sciences.
\end{abstract}


\maketitle
\tableofcontents

\section{Introduction}

Dynamical systems on networks are ubiquitous with wide applications in epidemiology, ecology, physics, social sciences, engineering, computer science, economics, neuroscience, etc. \cite{N18}. Networks can be dense or sparse. Almost all real world networks are \emph{sparse} \cite{L12}. For instance, social networks like friendship/collaboration network are sparse \cite{AN20}, since each person is connected to a finite number $m$ of other persons on the network of the entire population, and such $m$ is reasonably independent of the total population size on the earth. Generally, these social networks are locally dense while globally sparse \cite{P20}. Particle systems with short range interactions also form a sparse network \cite{ABPRS05}.

Some network models are coupled on graphs which are \emph{undirected}, since the bidirectional interaction of two nodes is symmetric, e.g., the oscillator network \cite{TSS20}, where these particles can be \emph{indistinguishable}. Other models admit directed network structure, e.g., the epidemic network, where the functional response of healthy individual and unhealthy individual may be different \cite{N18}. More importantly, these networks usually have a huge number of nodes, and it can be intractable to study these networks analytically or numerically. Both, the heterogeneity and density/sparsity of the network, may have a non-trivial impact on the dynamics on/of it. Let us illustrate this point by a simple Kuramoto oscillator network \cite{N18,SE19,TSS20}:
\begin{equation}\label{oscillator}
  \dot{\phi}_i(t)=\omega_i+\frac{1}{N}\sum_{j=1}^Na_{i,j}^N\sin(\phi_j(t)-\phi_i(t)),
\end{equation}
where $\omega_i$ is the natural frequency of the $i$-th oscillator, and $A^N=(a_{i,j}^N)$ is the adjacency matrix associated with a \emph{non-dense} network $G^N$ of the following four types: ring, star, tree, and ring of \emph{cliques}\footnote{A clique of a graph $G$ is an induced subgraph of $G$ that is complete.} (see Figure~\ref{fig-1}). These networks are widely used in modelling real-world phenomena (electrical network, interbank networks, television/computer networks, genetic network, social network) in finance, telecommunications, physics, genetics, neuroscience and social sciences \cite{FH16,SE19,P20,B20}.
  \tikzset{
    net node/.style = {circle, inner sep=0pt, outer sep=0pt, ball color=white},
    net root node/.style = {net node},
    net connect/.style = {draw=black},
  treenode/.style = {align=center, inner sep=0pt, text centered,
    font=\sffamily},
  arn_n/.style = {treenode, circle, font=\sffamily\bfseries, draw=black,
    fill=white, text width=1.5em},
   arn_r/.style = {treenode, circle, draw=black,
   fill=red, text width=1.5em},
  arn_x/.style = {treenode, circle, draw=black,
   fill=blue, opacity=0.7, text width=1.5em}
}
  \begin{figure}[h]
    \centering
    \begin{tikzpicture}
      \foreach \i in {1,...,10}
        \path (-126+\i*36:2) node (n\i) [arn_n,thick] {\i};
      \path [thick,net connect] (n1) -- (n2) -- (n3) -- (n4) -- (n5) -- (n6) -- (n7) -- (n8) -- (n9) -- (n10) -- (n1);
    \end{tikzpicture}\hspace{1cm}
    \begin{tikzpicture}
      \node (root) [arn_r,thick] {1};
      \foreach \i in {2,...,11}
        \draw [thick,-latex] (root)[arn_n] -- (-162+\i*36:2) node [arn_n] {\i};
    \end{tikzpicture}
    \vspace{.5cm}
\begin{tikzpicture}[-,thick,>=stealth',level/.style={sibling distance = 6cm/#1,
  level distance = 1.5cm},scale=.8]
\node [arn_r] {1}
    child{ node [arn_n] {2}
            {child{ node [arn_n] {4}
            	{
                child{ node [arn_x] {8}}
				child{ node [arn_x] {9}}
            }}
            child{ node [arn_n] {5}
				{
				child{ node [arn_x] {10}}
                child{ node [arn_x] {11}}
            }}
         }
    }
    child{ node [arn_n] {3}
            {child{ node [arn_n] {6}
            	{
				child{ node [arn_x] {12}}
                child{ node [arn_x] {13}}
            }}
            child{ node [arn_n] {7}
				{
				child{ node [arn_x] {14}}
                child{ node [arn_x] {15}}
            }}
         }
    };
\end{tikzpicture}
\hspace{.5cm}
%    \begin{tikzpicture}
%      \foreach \i in {3,6,...,15}
%        \path (-198+\i*24:3) node (\i) [arn_n,thick] {\i};
%        \foreach \i in {2,5,...,14}
%        \path (-206+\i*24:3) node (\i) [arn_n,thick] {\i};
%        \foreach \i in {1,4,...,13}
%        \path (-162+\i*24:2) node (\i) [arn_n,thick] {\i};
%      \path [thick,net connect] (1) -- (4) -- (7) -- (10) -- (13) -- (1);
%      \foreach \i in {1,4,...,13}
%            \path [thick,net connect] (1) -- (2) -- (3) -- (1);
%      \draw [thick,net connect] (4) -- (5) -- (6) -- (4);
%      \draw [thick,net connect] (7) -- (8) -- (9) -- (7);
%      \draw [thick,net connect] (10) -- (11) -- (12) -- (10);
%      \draw [thick,net connect] (13) -- (14) -- (15) -- (13);
%         \end{tikzpicture}
     \begin{tikzpicture}
      \foreach \i in {1,4,...,28}
     {
      \path (-102+\i*12:2) node (n\i) [arn_r,thick] {\i};
      }
      \foreach \j in {2,5,...,29}
      {
      \path (-120+\j*12:3) node (n\j) [arn_n,thick] {\j};
      }
      \foreach \k in {3,6,...,30}
      {
      \path (-114+\k*12:3) node (n\k) [arn_n,thick] {\k};
      }
      \path [thick,net connect] (n1) -- (n4) -- (n7) -- (n10) -- (n13) -- (n16) -- (n19) -- (n22) -- (n25) -- (n28) -- (n1);
             \foreach \i/\j/\k  in %{1/2,4/5,7/8,...,28/29}%
             {1/2/3,4/5/6,7/8/9,10/11/12,13/14/15,16/17/18,19/20/21,22/23/24,25/26/27,28/29/30}
       {
       \path [thick,net connect](n\i) -- (n\j) -- (n\k) -- (n\i);
       }
     \end{tikzpicture}
        \caption{Types of non-dense networks/graphs $G^N$ with $N$ labelled nodes. Top Left: Ring network. Top Right: Star network. Middle: Hierarchical tree network. Bottom: Ring network of cliques of size 3. Every node of a ring network has exactly two neighbors (degree 2). Every node except the red central node of a star network has exactly one neighbor while the central node has $N-1$ neighbors. For the binary tree network, every node at all levels except the first and the last has 3 neighbors while the red root node has 2 neighbors and the blue nodes at the last level has 1 neighbor. For the ring of cliques network, red nodes on the ring have 4 neighbors while nodes off the ring have 2 neighbors.}\label{fig-1}
  \end{figure}
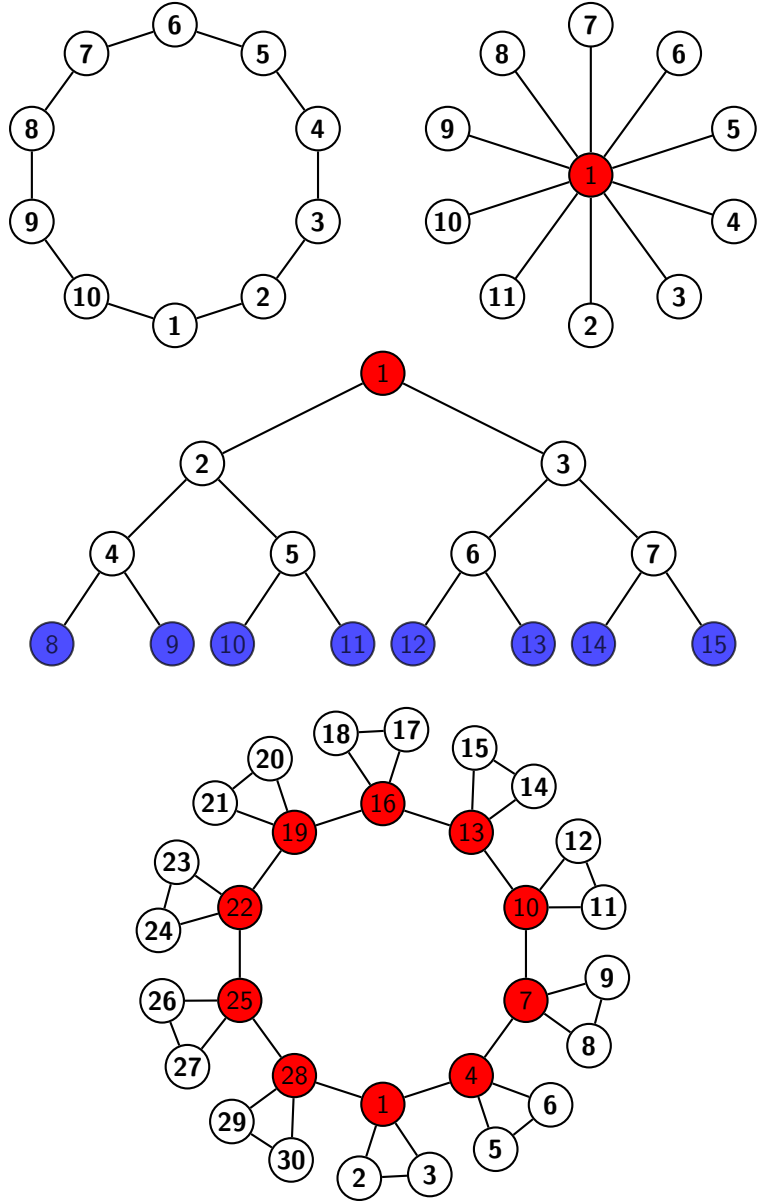It is noteworthy that the star network is neither sparse nor dense while a tree network, a ring network or a ring network of cliques is sparse (all nodes have uniformly bounded neighbors), c.f. \cite{L12} (see also Definition~\ref{graphing} and the comment that follows in Section~\ref{sect-preliminary}). An interesting yet natural question is: As the number $N$ of oscillators coupled on a graph $G^N$ tends to infinity, what is the dynamics of a \emph{typical} oscillator? It is known that the empirical distribution
\[\mu_N(t)=\frac{1}{N}\sum_{i=1}^N\delta_{\phi_i(t)},\]
is used to capture the dynamics of such a typical oscillator, where $\delta_{\phi_i(t)}$ is the Dirac measure at $\phi_i(t)$. The weak limit of $\mu_N(t)$ is the so-called \emph{mean field limit} (MFL). Furthermore, it is interesting to know how the typical dynamics of the network depends on the network structure. In order to have some understanding of this question, it seems fundamental to characterize the MFL of \eqref{oscillator}.

More generally, when the number of nodes of a sequence of  graphs increases to infinity, a dynamical system coupled on a sequence of graphs may also approach a \emph{limit} system. If the convergence is pointwise, then such a limit is the so-called \emph{continuum limit} of the particle system \cite{M19}. If the convergence is weak, then this limit is the MFL.

In the paper, we will try to address the arguably most fundamental topic regarding MFL of networks: Well-posedness and approximation of the MFL of finite dimensional dynamical systems on networks with heterogeneity ranging from dense to sparse. In subsequent works that follow, we will address the question how the heterogeneity and density/sparsity have an impact on the fine dynamics (e.g., synchronization, non-synchronization, bifurcation) of the network, in more specific contexts (e.g., for Kuramoto oscillators).

Next, we will briefly review some of the relevant results on MFL of dynamical systems on networks.
\subsection{Brief review of MFL of networks on graph limits}
Classical mathematical results on mean field theory of interacting particle systems (IPS) can be traced back at least to the 1970s~\cite{BH77,D79,N84}, where the MFL of dynamical systems coming from particle physics is studied. The relevant limit differential equation describing the MFL of the (all-to-all) coupled particle systems is the so-called \emph{Vlasov Equation} (VE) (on the complete graphs) \cite{D79,N84}. It is worth mentioning that despite \cite{D79} and \cite{N84} addressed the same problem independently with similar ideas, the metrics used in \cite{D79} and \cite{N84} are different: the Wasserstein metric is used in \cite{D79} while the bounded Lipschitz metric is used in \cite{N84}.
Systematic rigorous analysis of MFL of dynamical systems coupled on heterogeneous graphs seems to appear rather late, until the emergence of studies on \emph{limits of graph sequences} \cite{L12,KLS19,BS20}. In order to fully understand the impact of a sequence of discrete objects (graphs) on the dynamical behavior of the systems, one may need to find a way to represent these discrete objects in terms of analytical forms, which has only been developed recently, e.g., by Lov\'{a}sz \cite{L12}, Szegedy \cite{SL06}, and Backhausz \cite{BS20}. Most of these works focus on representing finite graphs and limit of a sequence of dense graphs as a function, so-called \emph{graphons} \cite{SL06,L12}. With this analytical representation of the discrete object, more works on MFL of dynamical systems (deterministic or stochastic) on graphons (deterministic or random) started to appear, see e.g.~\cite{KM18,BCW20}. With the most recent development of \emph{graphops} (for precise definitions, see \cite{BS20} or Section~\ref{sect-preliminary} below), operator representation of graph limits which may include sparse graphs, dense graphs, and those of intermediate density/sparsity, \cite{GK20} investigated MFLs of Kuramoto networks. We point out that thus far there have been several studies on MFLs of systems on graphons with wide applications, in control theory \cite{CH19}, neural networks and machine learning \cite{RCR20}, etc. It is worth pointing out that \cite{KM18} seems the first addressing the graphon MFL of dynamical systems introducing a \emph{fiber characteristic equation}, at least in the context of Kuramoto models. In this paper, we will try to address the MFL problem from the viewpoint of measure theory, i.e., by treating the digraph limits as a measure-valued function (\emph{digraph measures} as we name below), which naturally includes limits of classes of sparse graphs. Note that it is not new to regard limit of sparse graphs as measures \cite{KLS19,BS20}. Nevertheless, it seems our paper is the first to address MFL problem taking the measure analytic representation of graphs, \emph{by finding a proper complete metric space to work with}. To find a suitable complete metric space with a good metric seems the arguably most crucial part in successfully addressing the MFL problem with heterogeneity \cite{KM18,GK20}.

Next, let us review two relevant papers \cite{KM18,GK20} in a bit more detail. In \cite{KM18}, by extending Neunzert's idea \cite{N84} and introducing a fiber characteristic equation, the aforementioned two topics were investigated for a particular model (the Kuramoto oscillator model) coupled on graphons with a one-dimensional circular phase space; in \cite{GK20}, the results of \cite{KM18} were further extended to graphops, where harmonic analysis on locally compact Abelian groups was used to address the discretization of the graphops by graphons with kernels \cite{GK20}, under certain assumptions. This paper is motivated by \cite{N84,KM18,GK20}. We would like to stress that the metric for the analytical representation of graphs we use is different from \cite{KM18,GK20}.

\subsection{Contribution and main challenges}

In this paper, we aim to generalize Neunzert's approach to study well-posedness and discretization of particle systems on arbitrary finite-dimensional Euclidean space with a \emph{compact positively invariant subset} coupled on generalized digraphs. Despite the framework proposed in this paper is in the Euclidean space and the underlying digraph measures are assumed to be continuous in the vertex variable, it can be extended to IPS on \emph{Riemannian manifolds} such as sphere or torus coupled on digraph measures \emph{with finite discontinuity points} (see comments following the assumptions in Section~\ref{subsubsect-assumption} below), in a straightforward way with standard technicalities. Hence results in our paper apply potentially to e.g., Kuramoto models of higher-order interactions \cite{BBK20}, or the opinion dynamics model on the sphere \cite{CLP15}.

To provide some more intuitive understanding of digraph measures, let us revisit the graph sequences with the above four heterogeneity types in Figure~\ref{fig-1} and describe their weak limits as a measure-valued function. For graphs $G^N$ with a heterogeneous structure given in one of the four types in Figure~\ref{fig-1}, we assign uniform weight $N$ whenever two nodes are connected (i.e., $a_{ij}^N=0$ or $a_{ij}^N=N$). These graphs can be represented either as graphons, or measure-valued functions as proposed in this paper, both depending on the choice of the underlying vertex space $X$ as well as the reference measure $\mu_X$ assigned. To make a comparison, we list two representations of $G^N$ in Table~\ref{table-representations} to show the difference.
 \setlength{\tabcolsep}{9pt}
\renewcommand{\arraystretch}{1.8} \begin{table}[h]
  \begin{center}
  \resizebox{.95\columnwidth}{!}{
\begin{tabular}{p{2cm}|p{3.7cm}|p{3cm}|p{.8cm}|p{4.5cm}}
  \thickhline
%\rule{0pt}{4ex}   
 Network Type & $a^N_{ij}$ ($1\le i<j\le N$) for symmetric $G^N$ & $W^N$ & $X$&$\eta^N_x$ \\\hline
  Ring &  $N\mathbbm{1}_{\{1,N-1\}}(j-i)$& $\sum_{i,j=1}^Na^N_{i,j}\mathbbm{1}_{I^N_i\times I^N_j}$ & $\T$ &$\sum_{i=1}^N\mathbbm{1}_{A^N_i}(x)\frac{1}{N}\sum_{j=1}^Na^N_{i,j}\delta_{\frac{j-1}{N}}$ \\\hline
  Star & $N\mathbbm{1}_{\{1\}}(i)$ & $\sum_{i,j=1}^Na^N_{i,j}\mathbbm{1}_{I^N_i\times I^N_j}$ & $[0,1]$ & $\sum_{i=1}^N\mathbbm{1}_{I^N_i}(x)\frac{1}{N}\sum_{j=1}^Na^N_{i,j}\delta_{\frac{j-1}{N}}$\\\hline
  Binary tree & $N\mathbbm{1}_{\{i,i+1\}}(j-i)$ & $\sum_{i,j=1}^Na^N_{i,j}\mathbbm{1}_{I^N_i\times I^N_j}$ & [0,1] & $\sum_{i=1}^N\mathbbm{1}_{I^N_i}(x)\frac{1}{N}\sum_{j=1}^Na^N_{i,j}\delta_{\frac{j-1}{N}}$ \\\hline
 Ring of cliques & $\begin{cases}
    N\mathbbm{1}_{\{1,2,3,N-3\}}(j-i),\\ \hspace{1.8cm} \text{if}\ \frac{i+2}{3}\in\N,\\
    N\mathbbm{1}_{\{1\}}(j-i),\\ \hspace{1.8cm} \text{if}\ \frac{i+1}{3}\in\N,\\
    0,\hspace{2.5cm} \text{else}.
  \end{cases}$ \vspace{.3ex}& $\sum_{i,j=1}^Na^N_{i,j}\mathbbm{1}_{I^N_i\times I^N_j}$ & $\T$ & $\sum_{i=1}^N\mathbbm{1}_{A^N_i}(x)\frac{1}{N}\sum_{j=1}^Na^N_{i,j}\delta_{\frac{j-1}{N}}$ \\
  \thickhline
\end{tabular}
}
    \caption{Comparison between graphon and graph measure. The unit circle $\mathbb{T}$ is identified with $[0,1[$ \emph{via} the natural projection $\theta\to \text{e}^{2\pi{\rm i}}$. $I^N_i=\Bigl[\frac{i-1}{N},\frac{i}{N}\Bigr[$, for $1\le i<N$, $I^N_N=\Bigl[\frac{N-1}{N},1\Bigr]$ provides a partition for $[0,1]$ while $A^N_i=\Bigl[\frac{i-1}{N},\frac{i}{N}\Bigr[$, $1\le i\le N$, provides a partition for $\T$.}\label{table-representations}
    \end{center}
  \end{table}
\begin{table}[h]
  \begin{center}
\begin{tabular}{p{2.5cm}|p{3cm}|p{.8cm}|p{4.5cm}}
  \thickhline
%\rule{0pt}{4ex}
 Network Type &  $W^{\infty}$ & $X$&$\eta^{\infty}_x$ \\\hline
  Ring &  -- & $\T$ &$2\delta_x$ \\\hline
  Star & -- & $[0,1]$ & $\begin{cases}
    \textnormal{--},\quad\ \text{if}\  x=0,\\
    \delta_0,\quad \text{if}\ 0<x\le1,
  \end{cases}$\vspace{.3ex} \\\hline
  Binary tree & -- & [0,1] & $\begin{cases}
2\delta_0,\hspace{1.55cm} \text{if}\ x=0,\\
2\delta_{2x}+\delta_{x/2},\quad \text{if}\ 0<x\le1/2,\\
\delta_{x/2},\hspace{1.46cm} \text{if}\ 1/2<x\le1.
  \end{cases}$\vspace{.3ex} \\\hline
 Ring of cliques  & -- & $\T$ & -- \\
  \thickhline
\end{tabular}
\caption{Limit of representations of graph sequence $(G^N)_N$. Here `--' stands for the non-existence and $\lambda$ is the uniform measure over $[0,1]$.}\label{table-limit}
    \end{center}
  \end{table}

  %When taking the space of vertices $X=[0,1]$ or $X=, 
We list the limit of $(G^N)_{N\in\N}$ \emph{via} different representations in Table~\ref{table-limit}. From Table~\ref{table-limit}, the sequence of $(G^N)_{N\in\N}$ represented as measure-valued functions has a limit $\eta$ on $X$ and it has at most three discontinuity points. In contrast, the sequence represented as graphons $W^N\in{\sf L}^p([0,1]^2)$ does not converge in ${\sf L}^p([0,1]^2)$ for any $1\le p\le \infty$.

We illustrate the power of our results to sparse networks via the model \eqref{oscillator}. Assume $(G^N)_N$ is a sequence of rings specified by $(a^N_{i,j})$ given in Table~\ref{table-representations}. The well-posedness as well as the approximation of the MFL of the Kuramoto oscillator network \eqref{oscillator} can be addressed using the methods proposed in this paper with slight adjustment. To be more precise, we can show that the mean field limit of \eqref{oscillator} is given by the weak solution to the following VE:
\begin{align}
 \label{VE-oscillator} &\frac{\partial\rho(t,x,\phi)}{\partial t}+\frac{\partial}{\partial \phi}\left(\rho(t,x,\phi)\Bigl(
 \omega+2\int_{\mathbb{T}}\rho(t,x,\psi)\sin(\psi-\phi)\rd\psi\Bigr)\right)=0, t\in(0,T],\ x\in\T,\\
 \nonumber&\hspace{12cm}\mathfrak{m}\text{-a.e.}\ \phi\in \mathbb{T},\\
\nonumber  &\rho(0,\cdot)=\rho_0(\cdot),
  \end{align}
where $\mathfrak{m}$ refers to the Haar measure on the circle $\mathbb{T}$. Let $X=\T$ and $\mu_X=\mathfrak{m}$. Using the approach presented in this paper, one can characterize the weak limit of $\mu^N_t$ as follows. Let \begin{equation}\label{oscillator-empirical}
(\nu^N_t)_x=\sum_{i=1}^N\mathbbm{1}_{A^N_i}(x)\delta_{\phi^N_i(t)}.\end{equation}
 be a piecewise constant measure-valued function on $X$. Then formally, we can represent the empirical distribution via an integral of the measure-valued function $(\nu^N_t)_x$: $$\mu^N_t=\int_X(\nu^N_t)_x\rd\mu_X(x).$$ We can show that $\nu^N_t$ converges in a certain \emph{weak} sense to the measure-valued function $\nu_t$, where $(\nu_t)_{x}=\rho(t,x,\cdot)\mathfrak{m}$ is absolutely continuous w.r.t. $\mathfrak{m}$ with density $\rho(t,x,\cdot)$ being the \emph{uniform weak solution} to \eqref{VE-oscillator} (see Theorem~\ref{th-C} in the next Section for details), provided
 $(\nu^N_0)_x$ converges to the initial distribution $(\nu_0)_x$ with density $\rho_0(x,\cdot)$ uniformly in $x\in X$. Hence the MFL $\mu^N_t(\cdot)$ converges in a certain weak sense to $(\int_X\rho(t,x,\cdot)\mu_X(x))\mathfrak{m}(\cdot)$, where the measure is well-defined via the integral of the density $\rho(t,x,\cdot)$. In contrast, one fails to capture the MFL of the oscillator network \eqref{oscillator} on rings, by taking the underlying graphs as graphons \cite{KM18}. Similarly, if the sequence of the underlying networks are binary trees, the approach presented in our paper also applies, by Table~\ref{table-limit}. This illustrates well that the analytical way of interpreting graphs together with the metric/topology utilized does have an impact on characterizing the dynamics of the MFL of the same network models.

Apart from the above, we like to mention that not all networks can be represented appropriately as measure-valued functions which admit another measure-valued function as a limit, e.g, the sequence of stars or the sequence of rings of cliques (see Table~\ref{table-limit}). An appropriate representation of these graph limits in a possibly larger space than the space of measure-valued functions, is desirable. This question itself may be of independent interest in graph theory. Moreover, it is also interesting to know, how much within the class of sparse networks can be represented appropriately as measure-valued functions that converge to another measure-valued function so that we may have an idea of the sharp boundary within the space of graphs of our setup. We leave both questions for our future research.

%As mentioned earlier, the underlying graphs of some network models are directed, since the influence from one node dynamical system to its neighbour may be different the other way round, e.g., the epidemic network \cite{N18}.
We now comment on the condition on \emph{compact positive invariant regions}. Systems without compact positive invariant regions seem technically formidable to directly apply Neunzert's approach, since respective \emph{finite} upper estimates via Gronwall inequalities may not be possible without globally bounded Lipschitz functions in the vector field as assumed in \cite{N84,KM18,GK20}. Nonetheless, most applications, e.g. population models or chemical models of mass-action kinetics are only locally Lipschitz but not globally Lipschitz, yet most of these dynamical systems naturally admit a compact positively invariant subset (e.g., the density of population or concentration of chemical species is reasonably bounded for all times). More importantly, MFL of these models arising in diverse areas of science, e.g., population biology, molecular biology, chemistry, etc, are interesting. Hence for the sake of the practitioners in these areas, rather than introducing more crude techniques (e.g., by restricting the entire phase space to be a large ball in the Euclidean space while controlling density flux leaving certain area in a subtle way) we utilize the compactness of a positively invariant subset and confine the initial distributions of the mean field equations to be supported on this subset.

Other than the above technical challenge, the arguably biggest difficulty which reflects the novelty of this paper lies in the generalization from a dense graph (graphon) to a graph limit that is not necessarily dense. As pointed out in \cite{M19}, no matter for continuum limit or mean field limit, the absolute continuity of underlying graph measures (i.e., a graphon-type assumption or an approximation by graphons), was crucial for all the previous Vlasov equations derived on graphs. It is believed \cite{M19} that results in \cite{KM18} cannot be extended to cases, where absolute continuity fails. The reason for this is that convergence results as in \cite{KM18} are established based on approximation theory of ${\sf L}^p$-functions, and the existence of an ${\sf L}^p$-integrable kernel is precisely the ${\sf L}^p$-graphon. That means the approach in dealing with the approximation of VEs in \cite{KM18} cannot be extended in any direct way. To overcome, or rather get round this difficulty, tools from harmonic analysis for operators on locally compact Abelian groups were used in \cite{GK20}. They successfully reduced the problem for graphops, which may not admit a kernel to the situation dealt with in \cite{KM18}, since under certain assumptions, graphops can be approximated by graphons, and thus the approximation problem can be solved. However, in general situations, the assumptions in \cite{GK20} are not always easy to verify.

%Indeed, in \cite{KM18} due to the existence of a kernel as a graphon, ${\sf L}^p$-convergence theory applies. Furthermore, if suitable approximations via graphons can be used to obtain more general graph limits, such as certain graphops, one can again study mean-field limits~\cite{K20,GK20}. However, when the graph limits are no longer based on a kernel, graphon-based techniques  \cite{KM18} cannot be extended in a straightforward way~\cite{M19}.

As highlighted before, a natural analytical representation of the graph limit to work with in the context of dynamical systems is crucial and desirable. A smart choice of a complete metric space together with a good metric for the graph limits represented analytically to lie in is not trivial. To address this challenge, we consider in this paper graph limits purely from the perspective of measure theory: We regard a graph limit as a measure valued bounded (continuous) function. In doing so, the continuity on the vertex variable of the generalized graphs, the so-called \emph{digraph measures} (DGMs) (see Section~\ref{sect-preliminary} for the precise definition), is sufficient to guarantee well-posedness as well as discretization of Vlasov Equations on the DGMs. For the discretization result, we build upon the recently established results on deterministic empirical approximation of measures on Euclidean spaces \cite{XB19,C18}. We also point out that the IPS we study allow for distinguishable particles in terms of a directed generalized graph, whereas indistinguishable particles seem predominant in the literature where graphs are assumed to be symmetric.
%It is worth mentioning as well that the approach established in this paper readily adapts to dynamical systems with phase space being Riemannian manifolds such as sphere or torus, and hence apply potentially to e.g., Kuramoto models of high-order interactions \cite{BBK20}, or the opinion dynamics model on the sphere \cite{CLP15}.

Once we have successfully addressed the two main technical challenges, we carefully demonstrate how to apply our results to a wide variety of models ranging from epidemiology, ecology to social sciences (see Section~\ref{sect-applications}).

\subsection{Comparison with works in the literature}

We also compare our paper with the two aforementioned most relevant papers \cite{KM18,GK20}, from the technical perspective.

\begin{enumerate}
  \item[$\bullet$] In \cite{KM18}, the model is one dimensional and posed on a \emph{single} underlying generalized graph, which is symmetric and absolutely continuous (i.e., the graphon case is covered). In \cite{GK20}, the  \emph{single} underlying generalized graph is also symmetric (a certain class of graphops) with uniformly bounded fiber measures. In contrast, the main results in this paper hold for \emph{multiple} generalized digraphs which potentially are not symmetric.
  \item[$\bullet$] The topology utilized in \cite{GK20} which seems not easily metrizable, is also different from the uniform weak topology defined in this paper. Furthermore, the space of solutions of the VE is larger than those in \cite{KM18}, where all \emph{fiber measures} $(\nu_t)_x$  of the probability solutions of the VE are assumed to be probabilities (with normalized total variation norm) on the state space of the model. Indeed, $(\nu_t)_x$ stands for the distribution of particles at a location $x$, and in general particles may not be homogeneously distributed over all locations, e.g., one can even find no particles on certain locations in a sparse graph.
\item[$\bullet$] The reference probability measure on the vertex space $X$ is the Lebesgue measure on $[0,1]$ \cite{KM18} or the Haar measure on a locally compact Abelian group \cite{GK20}. In contrast, in our work the reference measure is not necessarily absolutely continuous w.r.t. the Lebesgue measure on the Euclidean space. Instead, the reference measure can be singular and discrete (see the examples in Section~\ref{sect-approximation-ode}). This demonstrates yet another advantage of the measure-theoretic viewpoint.
 \item[$\bullet$] The local Lipschitz continuity of functions assumed in this paper is weaker than the global ones generally assumed, e.g. in~\cite{KM18,GK20}. Such local Lipschitz assumptions are the only available ones in many applications (see Section~\ref{sect-applications}).
\end{enumerate}
In summary, we believe that a measure-theoretic approach can be extremely helpful to study a very large variety of IPS on graphs, as it exploits a natural analytical viewpoint of graph limits~\cite{BS20}, namely studying the graph limit purely \emph{via} fiber measures.

\section{Overview of the main results}
\begingroup
\setcounter{tmp}{\value{theorem}}% store current value of theorem counter
\setcounter{theorem}{0} %assign desired value to theorem counter
\renewcommand\thetheorem{\Alph{theorem}}% locally redefine the representation of the theorem counter

Here, we first provide an informal overview of the assumptions and the main results of this paper. We also outline the general strategy in a bit more detail. Precise results will be stated in Sections~\ref{sect-setup}-\ref{sect-approximation-ode}.

\subsection{Summary of main results}

\subsubsection*{Assumptions}\label{subsubsect-assumption}

Let $\mathcal{T}=[0,T]$ be the time domain of the dynamics for some $T>0$, and $r_1,r_2,r\in\N$. Let $\mathfrak{m}$ be the Lebesgue measure on $\R^{r_2}$. Let $\mathfrak{B}(X)$ be the Borel sigma algebra of a metric space $X$, and $\cM_+(X)$ the space of all finite signed Borel measures on $X$. Let $\mathcal{B}(X_1,X_2)$ ($\mathcal{C}(X_1,X_2)$, respectively) the space of bounded measurable (continuous, respectively) $X_2$-valued functions on the space $X_1$. To provide the basic setup, we need to specify assumptions regarding the vertex space $X$ of the DGM $\eta$, the vertex dynamics phase space $Y$, the vector field $h$ for the vertex dynamics, the interaction forces $g_i$ among different vertices, and the nonlocal mapping $V$ defining the VE. Our goal is to construct a measure-valued solution $\nu_t$ to the VE and prove an approximation theorem of the VE via finite-dimensional ODEs. To achieve this, we make the following assumptions.

\medskip
\noindent{$\mathbf{(A1)}$} $(X,\mathfrak{B}(X),\mu_X)$ is a compact Polish probability space equipped with metric induced by the $\ell_1$-norm of $\R^{r_1}\supseteq X$.

\medskip
\noindent{$\mathbf{(A2)}$} $(t,\psi,\phi)\mapsto g_i(t,\psi,\phi)\in\R^{r_2}$ is continuous in $t\in\cT$, and locally Lipschitz continuous in $(\psi,\phi)\in \R^{2r_2}$ uniformly in $t$, i.e., for every $(\psi,\phi)\in \R^{2r_2}$, there exists a neighbourhood $(\psi,\phi)\ni\mathcal{N}\subseteq\R^{2r_2}$ such that
$$\sup_{t\in\cT}\sup_{\tiny\begin{array}{l}
  (\psi_1,\phi_1)\neq (\psi_2,\phi_2),\\
(\psi_1,\phi_1),(\psi_2,\phi_2)\in\mathcal{N}
\end{array}}\frac{|g_i(t,\psi_1,\phi_1)-g_i(t,\psi_2,\phi_2)|}{|(\psi_1,\phi_1)-(\psi_2,\phi_2)|}<\infty.$$

\medskip
\noindent{$\mathbf{(A3)}$} $(t,x,\phi)\mapsto h(t,x,\phi)\in\R^{r_2}$ is continuous in $t\in\cT$, and locally Lipschitz continuous in $\phi\in \R^{r_2}$ uniformly in $(t,x)$, i.e., for every $\phi\in \R^{r_2}$ for some $r_2\in\N$, there exists a neighbourhood $\phi\ni\mathcal{N}\subseteq\R^{r_2}$ such that $$\sup_{t\in\cT}\sup_{x\in X}\sup_{\tiny\begin{array}{l}
  \phi_1\neq \phi_2,\\
\phi_1,\phi_2\in\mathcal{N}
\end{array}}\frac{|h(t,x,\phi_1)-h(t,x,\phi_2)|}{|\phi_1-\phi_2|}<\infty.$$

\medskip
\noindent{$\mathbf{(A4)}$} $\eta=(\eta^1,\ldots,\eta^r)\in (\mathcal{B}(X,\cM_+(X)))^r$.

\medskip
\noindent{$\mathbf{(A4)'}$} $\eta=(\eta^1,\ldots,\eta^r)\in (\mathcal{C}(X,\cM_+(X)))^r$.

\medskip
\noindent{$\mathbf{(A5)}$} $\nu_{\cdot}\in \mathcal{C}(\cT,\mathcal{B}_{\mu_X,1}(X,\cM_+(\R^{r_2})))$ is uniformly compactly supported in the sense that
there exists a compact set $E_{\nu_{\cdot}}\subseteq\R^{r_2}$ such that $\cup_{t\in\R}\cup_{x\in X}\supp(\nu_t)_x\subseteq E_{\nu_{\cdot}}$.

\medskip
\noindent ($\mathbf{A6}$) There exists a convex compact set $Y\subseteq\R^{r_2}$\footnote{Here $r_2$ is the smallest dimension $l$ such that $Y\subseteq\R^l$.}
 such that for all $\nu_{\cdot}$ satisfying ($\mathbf{A5}$) uniformly supported within $Y$, the following inequality holds:
\[V[\eta,\nu_{\cdot},h](t,x,\phi)\cdot\upsilon(\phi)\le0,\quad \text{for all}\ t\in\mathcal{T},\ x\in X,\quad \phi\in\partial Y,\]
where $\partial Y=Y\cap\overline{\R^{r_2}\setminus Y}$, $\upsilon(\phi)$ is the outer normal vector at $\phi$, and
\begin{multline}\label{Eq-VOperator}
V[\eta,\nu_{\cdot},h](t,x,\phi)
=\sum_{i=1}^r\int_X\int_{\R^{r_2}}g_i(t,\psi,\phi)\rd(\nu_t)_y(\psi)\rd\eta^i_x(y)+h(t,x,\phi),\\ \quad t\in\cT, x\in X,\ \phi\in \R^{r_2}
\end{multline}

\medskip

\noindent{$(\mathbf{A7})$}  $(t,x,\phi)\mapsto h(t,x,\phi)\in\R^{r_2}$ is is continuous in $x$ uniformly in $\phi$:
\[\lim_{|x-y|\to0}\sup_{\phi\in Y}|h(t,x,\phi)-h(t,y,\phi)|=0,\quad t\in\cT,\]
where $Y$ is the compact set given in ($\mathbf{A6}$). Moreover, $h$ is integrable uniformly in $x$:
\[\int_0^T\int_Y\sup_{x\in X}|h(t,x,\phi)|\rd\phi\rd t<\infty.\]

Let us provide some intuitive explanation for these assumptions. Assumption $(\mathbf{A1})$ means that the underlying generalized digraphs (DGMs) have the same compact vertex space $X$. Such compactness is used in establishing discretization of DGMs. Note that if different DGMs $\eta^i$ have different vertex spaces $X^i\subseteq\R^{r_{1,i}}$, then one can take $X=\cup_{i=1}^r X^i\subseteq\R^{r_1}$ with the metric induced by the $\ell_1$-norm of $\R^{r_1}$ with $r_1=\max_{1\le i\le r}r_{1,i}$.  Assumptions $(\mathbf{A2})$-$(\mathbf{A3})$ are the standard Lipschitz conditions for the well-posedness of (non-local) ODE models. Assumption $(\mathbf{A4})$ means that we interpret the graphs as measure-valued functions; note that we can think of $\eta_x$ as describing the local edge density or connectivity near vertex $x$. Next, we need the assumption for the approximation of the VE (i.e., the mean field equation for the IPS) that the family of graph measures $\eta_x$ are continuous in the vertex variable $x$, which is encoded in assumption $(\mathbf{A4})'$ (essentially used in Lemma~\ref{le-graph}).
We would like to remark that $\mathbf{(A4)'}$ is indeed not crucial for the approximation results. One can relax this assumption by allowing $x\mapsto\eta^i_x$ ($i=1,\ldots,r$) to have finitely many discontinuity points. Nevertheless, for the ease of exposition and to avoid arousing further difficulty in understanding the approximation, we only present the result under the continuity assumption $\mathbf{(A4)'}$.
In fact, such regularity condition does not exclude interesting situations, where the graph limits can be sparse, dense, or neither sparse nor dense (see the examples in Sections~\ref{sect-preliminary} and \ref{sect-approximation-ode}). Assumption $(\mathbf{A5})$ ensures the uniform boundedness of the time-dependent measures $\nu_t$ in total variation norm, which is used to establish the well-posedness of the non-local equation \eqref{Charac} of characteristics (see Theorem~\ref{theo-well-posedness-characteristic} below). Assumption $(\mathbf{A6})$ is \emph{Bony's condition} \cite{B69}  (also called \emph{Nagumo's condition} \cite{N42}) for the existence of a compact positively invariant subset $Y$ of the equation of characteristics; for vast research on this classical topic of independent interest, c.f.,  \cite{B69,B70,Y70,H72,R72,C72,M73,C75,C76} and \cite[Chap.10]{W98}. The compactness of $Y$ in $(\mathbf{A6})$ is required in $(\mathbf{A5})$ for bounds of $\nu_t$.  Assumption $(\mathbf{A7})$ is technical, used to establish the continuous dependence of solutions to the VE on $h$ (see Proposition~\ref{prop-sol-fixedpoint}).

\subsubsection*{Equation of characteristics}

Under {$\mathbf{(A1)}$}-{$\mathbf{(A5)}$}, the \emph{Vlasov operator} $V$ given in \eqref{Eq-VOperator} is well defined.

Let $t_0\in\cT$ and $\phi_0\in\mathcal{B}(X,\R^{r_2})$. For every $x\in X$, consider the following IVP of a measure-induced differential equation
\begin{equation}
  \label{Charac}
  \frac{\partial}{\partial t}\phi(t,x)=V[\eta,\nu_{\cdot},h](t,x,\phi),\quad t\in\cT,\quad \phi(t_0,x)=\phi_0(x).
\end{equation}
The IVP of \eqref{Charac} confined to a finite time interval $\cT$ is the so-called \emph{equation of characteristics} (or \emph{characteristic equation}). When the underlying space $X$ is finite, and the measures $(\nu_t)_x$ and $\eta^i_x$ for all $x\in X$ are finitely supported, \eqref{Charac} becomes a system of ODEs coupled on a finite set of directed graphs in terms of $\{\eta^i\}_{1\le i\le r}$. Hence, the characteristic equation forms an intermediate bridge between a finite-dimensional IPS and the Vlasov equation, effectively containing the information about both systems. The well-posedness of \eqref{Charac} is standard from ODE theory.

\begin{theorem}\label{theo-well-posedness-characteristic}
Assume ($\mathbf{A1}$)-($\mathbf{A5}$). Let $\phi_0\in\mathcal{B}(X;\R^{r_2})$. Then for every $x\in X$ and $t_0\in\cT$, there exists a solution $\phi(t,x)$ to the IVP of \eqref{Charac} with $\phi(t_0,x)=\phi_0(x)$ for all $t\in(T^{x,t_0}_{\min},T^{x,t_0}_{\max})\cap\cT$ with $(T^{x,t_0}_{\min},T^{x,t_0}_{\max})\subseteq\R$ being a neighbourhood of $t_0$ such that
\begin{enumerate}
\item[(i)] either (i-a) $T^{x,t_0}_{\max}>T$ or (i-b) $T^{x,t_0}_{\max}\le T$ and $\lim_{t\uparrow T^{x,t_0}_{\max}}|\phi(t,x)|=\infty$ holds, and
\item[(ii)] either (ii-a) $T^{x,t_0}_{\min}<0$ or (ii-b) $T^{x,t_0}_{\min}\ge0$ and $\lim_{t\downarrow T^{x,t_0}_{\min}}|\phi(t,x)|=\infty$ holds.
\end{enumerate}
In addition, assume ($\mathbf{A6}$) and $\nu_{\cdot}$ is uniformly supported within $Y$, then $(T^{x,t_0}_{\min},T^{x,0}_{\max})\cap\cT=\cT$ for all $x\in X$, and there exists a set $\left\{\mathcal{S}^x_{t,s}[\eta,\nu_{\cdot},h]\right\}_{t,s\in\cT}$ of transformations forming a group on $Y$ such that
$$\phi(t,x)=\mathcal{S}^x_{t,s}[\eta,\nu_{\cdot},h]\phi(s,x),\quad \text{for all}\ s,t\in\cT.$$
\end{theorem}

As a next step, we try to link the characteristic equation to the mean-field Vlasov equation. Since $\left\{\mathcal{S}^x_{t,s}\right\}_{t,s\in[0,+\infty)}$ is a group, from Theorem~\ref{theo-well-posedness-characteristic}, we have for all $x\in X$, $$(\mathcal{S}^x_{t,0}[\eta,\nu_{\cdot},h])^{-1}=\mathcal{S}^x_{0,t}[\eta,\nu_{\cdot},h],\quad t\in\cT.$$ The pushforward under the flow $\mathcal{S}^x_{t,0}[\eta,\nu_{\cdot},h]$ of an initial measure $(\nu_0)_x\in\mathcal{B}_{\mu_X,1}(X,\cM_+(Y))$ defines another time-dependent measure in $\mathcal{B}_{\mu_X,1}(X,\cM_+(Y))$ via the following \emph{fixed point equation} \begin{equation*}
  \mathcal{A}^x\nu_{\cdot}(t)=(\nu_t)_x=(\nu_0)_x\circ\mathcal{S}^x_{0,t}[\eta,\nu_{\cdot},h],\quad t\in\cT.
\end{equation*}
In particular, if $\nu_{\cdot}\in \mathcal{C}(\cT,\mathcal{B}_{\mu_X,1}(X,\cM_{+,\abs}(\R^{r_2})))$, then $\nu_{\cdot}\in \mathcal{C}(\cT,\mathcal{B}_{\mu_X,1}(X,\cM_{+,\abs}(Y)))$ by the positive invariance of $Y$. Hence the \emph{Vlasov operator} can be represented in terms of the density $\rho(t,y,\phi)\colon=\frac{\rd(\nu_t)_y(\phi)}{\rd\mu_X(y)\rd\phi}$ for every $t\in\cT$: \begin{equation}\label{Vrho-0}
\widehat{V}[\eta,\rho(\cdot),h](t,x,\phi)=\sum_{i=1}^r\int_X\int_{Y}g_i(t,\psi,\phi)\rho(t,y,\phi)\rd\psi
\rd\eta^i_x(y)+h(t,x,\phi).
\end{equation}

Let $${\sf L}^1_+(X\times Y,\mu_X\otimes\mathfrak{m})=\{f\in {\sf L}^1(X\times Y,\mu_X\otimes\mathfrak{m})\colon \int_{X\times Y}f\rd\mu_X\rd\mathfrak{m}=1\},$$ be the space of densities of probabilities on $X\times Y$. Conversely, for every function $\rho\colon\mathcal{T}\to{\sf L}^1_+(X\times Y,\mu_X\otimes\mathfrak{m})$, $$\rd(\nu_t)_y(\phi)=\rho(t,y,\phi)\rd\mu_X(y)\rd\phi$$ defines $\nu_{\cdot}\in\mathcal{B}(\mathcal{T},\mathcal{B}_{\mu_X,1}(X,\cM_+(Y)))$. Hence \eqref{Vrho-0} can be transformed to the Vlasov operator \eqref{Eq-VOperator} in terms of $\nu_{\cdot}$.

Let $\rho_0\colon X\times Y\to\R_+$ be continuous in $x$ for $\mathfrak{m}$-a.e. $\phi\in Y$, and integrable in $\phi$ for every $x\in X$ such that $$\int_X\int_Y\rho_0(x,\phi)\rd\phi\rd\mu_X=1.$$  Consider the VE
\begin{align}
 \label{Vlasov} &\frac{\partial\rho(t,x,\phi)}{\partial t}+\textrm{div}_{\phi}\left(\rho(t,x,\phi)\widehat{V}[\eta,\rho(\cdot),h](t,x,\phi)\right)=0, t\in(0,T],\ x\in X,\ \mathfrak{m}\text{-a.e.}\ \phi\in Y,\\
\nonumber  &\rho(0,\cdot)=\rho_0(\cdot).
  \end{align}
With the above assumptions and under appropriate metrics, one can show that the operator $\mathcal{A}=(\mathcal{A}^x)_{x\in X}\colon \mathcal{C}(\cT,\mathcal{B}_{\mu_X,1}(X,\mathcal{M}_+(Y)))\to  \mathcal{C}(\cT,\mathcal{B}_{\mu_X,1}(X,\mathcal{M}_+(Y)))$ is a contraction. Using the Banach fixed point theorem, it follows that the unique solution $\nu_{\cdot}$ to the fixed point equation exists. The fixed point equation was named by Neunzert \cite{N84} the \emph{generalized VE}, since in particular, $(\nu_t)_x$ is absolutely continuous for all $x\in X$ with its density solving the VE \eqref{Vlasov}, provided the initial measure $(\nu_0)_x$ is so for all $x\in X$. Hence we obtain the well-posedness of the VE \eqref{Vlasov}.

\begin{theorem}
Assume ($\mathbf{A1}$)-($\mathbf{A4}$) and ($\mathbf{A6}$). Assume $\rho_0(x,\phi)$ is continuous in $x\in X$ for $\mathfrak{m}$-a.e. $\phi\in Y$ such that $\rho_0\in {\sf L}^1_+(X\times Y,\mu_X\otimes\mathfrak{m})$, then there exists a unique uniformly weak solution\footnote{See Definition~\ref{def-weak-sol}.} to the IVP of \eqref{Vlasov} with initial condition $\rho(0,x,\phi)=\rho_0(x,\phi)$, $x\in X$, $\phi\in Y$.
\end{theorem}

Based on  $\mathbf{(A1)}$-$\mathbf{(A7)}$ with $\mathbf{(A4)}$ replaced by $\mathbf{(A4)'}$, we also establish continuous dependence of solutions to the fixed point equation on the underlying DGMs $\eta^i$ for $i=1,\ldots,m$, on the initial measure $\nu_0$,  as well as on function $h$ (see Proposition~\ref{prop-sol-fixedpoint} in Section~\ref{sect-setup}). Using this result combined with the recently established results on \emph{deterministic empirical approximation of positive measures} \cite{XB19,C18} (Lemma~\ref{le-partition}), we establish the discretization of solutions of VE over finite time interval $\cT$ by a sequence of discrete ODE systems coupled on finite graphs converging \emph{weakly} to the DGMs $\eta^i$.

Indeed, for any $\nu_0\in\mathcal{C}(X,\cM_+(Y))$ and $\eta$ satisfying $\mathbf{(A4)'}$,
by Lemmas~\ref{le-partition}-\ref{le-graph} in Section~\ref{sect-approximation-ode}, there exists
 \begin{enumerate}
 \item[$\bullet$] a partition $\{A^m_i\}_{1\le i\le m}$ of $X$ and points $x^m_i\in A^m_i$ for $i=1,\ldots,m$, for every $m\in\N$,
 \item[$\bullet$]  a sequence $\{\varphi^{m,n}_{(i-1)n+j}\colon i=1,\ldots,m,j=1,\ldots,n\}_{n,m\in\N}\subseteq Y$ and $\{a_{m,i}\colon i=1,\ldots,$ $m\}_{m\in\N}\subseteq\R_+$,
  \item[$\bullet$] a sequence $\{y^{\ell,m,n}_{(i-1)n+j}\colon i=1,\ldots,m,j=1,\ldots,n\}_{m,n\in\N}\subseteq Y$ and $\{b_{\ell,m,i}\colon i=1,\ldots,$ $m\}_{m\in\N}\subseteq\R_+$, for $\ell=1,\ldots,r$,
 \end{enumerate}
 such that
  $$\lim_{m\to\infty}\lim_{n\to\infty}d_{\infty}(\nu_0^{m,n},\nu_0)=0,$$
   $$\lim_{m\to\infty}\lim_{n\to\infty}d_{\infty}(\eta^{\ell,m,n},\eta)=0,\quad \ell=1,\ldots,r,$$
     $$\lim_{m\to\infty}\int_0^T\int_Y\sup_{x\in X}\lt|h^m(t,x,\phi)-h(t,x,\phi)\rt|\rd\phi\rd t=0,$$
 where
 \begin{subequations}
   \label{discretization-0}
 \begin{alignat}{2}
(\nu_0^{m,n})_x\colon=\sum_{i=1}^m\mathbbm{1}_{A^m_i}(x)\frac{a_{m,i}}{n}\sum_{j=1}^n
\delta_{\varphi^{m,n}_{(i-1)n+j}},\quad x\in X,\\ \eta^{\ell,m,n}_x\colon=\sum_{i=1}^m\mathbbm{1}_{A^m_i}(x)\frac{b_{\ell,m,i}}{n}
\sum_{j=1}^n\delta_{y^{\ell,m,n}_{(i-1)n+j}},\quad x\in X,\\
 h^m(t,z,\phi)=\sum_{i=1}^m\mathbbm{1}_{A^m_i}(z)h(t,x_i^m,\phi),\quad t\in\cT,\ z\in X,\ \phi\in Y.
 \end{alignat}
 \end{subequations}

Here we provide some heuristic intuition on how to understand the above approximations. Let us take $\eta$ for an example (the other approximation for $\nu_{\cdot}$ is analogous). Note that $\{A^m_i\}_{1\le i\le m}$ is a partition of $X$ with uniformly vanishing diameter $\sup_{1\le i\le m}\Diam A^m_i\to0$ as $m\to\infty$. Since $x\mapsto\eta_x$ is continuous, we have $\sup_{1\le i\le m}\sup_{x,x'\in A^m_i}d_{\sf BL}(\eta_x,\eta_{x'})$ is small for large $m$, and one can choose any point $x^m_i$ in $A^m_i$ so that $\eta_{x_i^m}$ is a representative for all fiber measures $\eta_x$ for $x\in A^m_i$. Then given $m\in\N$, for every $n\in\N$, one can look for uniform $n$-approximations (i.e., the deterministic empirical approximation with at most $n$ atoms) of the finite positive measure $\eta_{x^m_i}$ for each $i$, which is $\frac{b_{\ell,m,i}}{n}\sum_{j=1}^n\delta_{y^{\ell,m,n}_{(i-1)n+j}}$, where $b_{\ell,m,i}$ is the averaged total mass of $\eta^{\ell}_{x}$ for $x\in A^m_i$ provided $A^m_i$ is not a $\mu_X$-measure zero set, and is the total mass of $\eta^{\ell}_{x^m_i}$ otherwise. Equivalently, due to continuity of $\eta$, $\frac{1}{n}\sum_{j=1}^n\delta_{y^{\ell,m,n}_{(i-1)n+j}}$ is a deterministic empirical approximation of the probability measure $\frac{1}{b_{\ell,m,i}}\eta^{\ell}_{x^m_i}$, provided $\eta^{\ell}_{x^m_i}$ is not degenerate (i.e., $b_{\ell,m,i}\neq0$). The possibility that one can always construct such an approximation for a probability measure in the Euclidean space is guaranteed by recent works on deterministic empirical approximation of probabilities \cite{C18,XB19}. This is why the approximation seems different from those in e.g. \cite{KM18}, since all fiber measures $\eta_x$ are probabilities therein and the partition of $X=[0,1]$ is uniform (i.e., $\mu_X(A^m_i)=1/m$ for all $i=1,\ldots,m$, where $\mu_X$ is the Lebesgue measure on $X$).

Based on the above discretization of measures and functions, consider the following IVP of a coupled ODE system:
\begin{multline}
  \label{lattice-0}
  \dot{\phi}_{(i-1)n+j}=F^{m,n}_{i}(t,\phi_{(i-1)n+j},\Phi),\quad 0<t\le T,\quad \phi_{(i-1)n+j}(0)=\varphi_{(i-1)n+j},\\ i=1,\ldots,m,\ j=1,\ldots,n,
\end{multline}
where $\Phi=(\phi_{(i-1)n+j})_{1\le i\le m,1\le j\le n}$ and $$F^{m,n}_{i}(t,\psi,\Phi)=\sum_{\ell=1}^r\sum_{p=1}^{m}\frac{a_{m,i}b_{\ell,m,p}}{n^2}\sum_{j=1}^n
\mathbbm{1}_{A^m_p}(y^{\ell,m,n}_{(i-1)n+j})\sum_{q=1}^{n}g_{\ell}(t,\psi,\phi_{(p-1)n+q})+h^m(t,x^m_i,\psi).$$
For $t\in\cT$, let $\phi^{m,n}(t)=(\phi^{m,n}_{(i-1)n+j}(t))$ be the solution to \eqref{lattice-0}.  Define the time-dependent measures generated by the solutions to \eqref{lattice-0}: \begin{equation}\label{Eq-approx-0}
(\nu_t^{m,n})_x\colon=\sum_{i=1}^m\mathbbm{1}_{A^m_i}(x)\frac{a_{m,i}}{n}
\sum_{j=1}^{n}\delta_{\phi^{m,n}_{(i-1)n+j}(t)},\quad x\in X.
\end{equation}
\begin{theorem}\label{th-C}
Assume ($\mathbf{A1}$)-($\mathbf{A3}$), $\mathbf{(A4)'}$, ($\mathbf{A6}$)-($\mathbf{A7}$). Assume $\rho_0(x,\phi)$ is continuous in $x\in X$ for $\mathfrak{m}$-a.e. $\phi\in Y$ such that $\rho_0\in {\sf L}^1_+(X\times Y,\mu_X\otimes\mathfrak{m})$ and $$\sup_{x\in X}\|\rho_0(x,\cdot)\|_{{\sf L}^1(Y,\mathfrak{m})}<\infty.$$  Let  $\rho(t,x,\phi)$ be the uniformly weak solution to the VE \eqref{Vlasov} with initial condition $\rho_0$. Let $\nu_{\cdot}\in\mathcal{C}(\cT;\mathcal{B}_{\mu_X,1}(X,\cM_{\abs}(Y)))$ be the measure-valued function defined in terms of the uniformly weak solution to \eqref{Vlasov}:
$$\rd(\nu_t)_x=\rho(t,x,\phi)\rd\phi,\quad \text{for every}\quad t\in\cT\quad \text{and}\quad x\in X.$$ Then $\nu_t\in \mathcal{C}(\cT,\mathcal{C}_{\mu_X,1}(X,\cM_+(Y)))$. Moreover, let
$\nu^{m,n}_0\in\mathcal{B}_{\mu_X,1}(X,\cM_+(Y))$, $\eta^{\ell,m,n}\in \mathcal{B}(X,\cM_+(Y))$, and $h^m\in \mathcal{C}(\cT\times X\times Y,\R^{r_2})$ be defined in \eqref{discretization-0},
and $\nu^{m,n}_{\cdot}$ be defined in \eqref{Eq-approx-0}. Then $$\lim_{n\to\infty}d_{\infty}(\nu_t^{m,n},\nu_t)=0.$$
\end{theorem}

\endgroup

Finally we are going to apply the above main results to models in epidemiology, ecology, and social sciences (see Section~\ref{sect-applications}).

\subsection{Brief description of methods}
Here we provide a new perspective from measure theory rather than utilizing operator-theoretic methods from functional analysis. Instead of using the weak topology of the space of measures on the product space, we introduce the so-called \emph{uniform weak topology} in terms of the uniform metric which induces this slightly stronger topology (than the weak topology). We then define the uniform weak solution to the VE, and show that the solution of a fixed point equation (in the sense of Neunzert \cite{N84}), coincides with the solution of the VE, provided the initial distribution is absolutely continuous. Such an approach can be viewed as a generalization of Neunzert's in-cell-method \cite{N84}. With this new setup, the additional assumption of continuity of solutions to the VE required in \cite{GK20} is proved, via the Banach fixed point theorem by confining the contraction operator to the subclass $\mathcal{C}(\cT,\mathcal{C}_{\mu_X,1}(X,\cM_+(Y)))$ of continuous in time measure-valued functions which is continuous in the vertex variable (see Proposition~\ref{prop-continuousdependence}). We mention that to show this contraction operator from the space $\mathcal{C}(\cT,\mathcal{C}_{\mu_X,1}(X,\cM_+(Y)))$ to itself is technically very challenging, since the pushforward of a given initial measure under the flow of the equation of characteristics may not necessarily define an operator from a space of spatially \emph{continuous} measure-valued functions to the space itself.

A second difficulty comes from the compactness barrier. On the one hand, the compactness of the underlying phase space is technically crucial in \cite{KM18,GK20}, where this compactness condition is automatically fulfilled since the phase space is the unit circle. The technical reason for this assumption is that the arguments require a global bounded Lipschitz condition of the functions appearing in the vector field of the dynamical systems. On the other hand, Neunzert's approach requires that the measure under the pushforward solution map (flow) again lies in the space of measures supported on the same phase space, so that the operator is from one metric space to the same metric space in order for the Banach fixed point theorem to apply. That explains, why Neunzert's method cannot immediately apply to Euclidean spaces, which are \emph{not} compact. Most importantly, this might explain why Neunzert's approach has rarely been generalized to other models than the Kuramoto type models, since other models e.g., the epidemic models of mass-action kinetics for disease transmission and competition models have local but \emph{not} global Lipschitz functions in the vector field. We mention that as Neunzert pointed out \cite{N84}, as long as the functions in the vector field is not globally Lipschitz, the solution to the VE may only exist for a finite maximal time locally (which in general is numerically unknown). We deal with these problems by working on \emph{positively invariant compact subsets} of the phase space. We then carefully extend the approach in several different arguments, e.g., by showing the absolute continuity of solutions of the generalized VE (i.e., the fixed point equation) via Rademacher's change of variables' formula, which classically also holds only on the entire Euclidean space. Then we construct the fixed point equation by taking initial distributions supported on the positively invariant compact subset. In this way, we overcome the above two difficulties.

In addition, there is another difficulty in the approximation of solutions to the VE. In \cite{KM18}, the martingale convergence theorem is applied to the Hilbert space of ${\sf L}^2$-integrable functions (graphons). In \cite{GK20}, such approximation relies on certain technical assumptions in harmonic analysis from the viewpoint of operators (existence of summability kernels), which are crucial for approximation of ${\sf L}^1$-integrable functions on locally compact Abelian groups. However, for graphs which are not dense (with an ${\sf L}^2$ kernel) or graphops which are not limit of graphons (e.g., when $X$ is not simply the unit interval and $\mu_X$ the Lebesgue measure), the two approaches aforementioned fail to help. In this paper, we use the continuity of DGMs as well as the recently established results on uniform approximation of positive measures \cite{C18,XB19} (Proposition~\ref{prop-XBC}) combined with partitions of Euclidean space (Lemma~\ref{le-partition}) to derive an approximation of the initial distribution of the VE as well as the DGMs (Lemma~\ref{le-ini-2} amd Lemma~\ref{le-graph}).

\subsection*{Outline of the paper}

In the next section, we introduce notation, recall preliminaries on metric spaces, measure theory, and graph theory, and establish properties of several spaces of measure-valued functions, which play a crucial role in setting up the problem. The results are subsequently used to obtain continuity properties of flows of characteristic equations. In Section~\ref{sect-setup}, we establish continuity and then also Lipschitz continuity of the vector field as well as the flow of the equation of characteristics. In Section~\ref{sect-approximation-ode}, we provide specific approximation schemes for measure-valued continuous functions for several underlying vertex spaces, e.g., $[0,1]$, $\mathbb{T}^1$, $\mathbb{S}^2$, and $[0,1]^2$. Moreover, we provide discretization of VE on DGMs. In Section~\ref{sect-applications}, the main results are applied to models in epidemiology, ecology, and social sciences. A brief discussion including possible future research topics is presented in Section~\ref{sect-discussion}. The proofs of main results are contained in Section~\ref{sect-proof}. Finally, proofs of propositions and lemmas as  well as a quadratic Gronwall inequality are appended.

\section{Preliminaries}\label{sect-preliminary}
\subsection*{Notation}
Let $\R_{+}$ be the set of nonnegative real numbers. For every $x\in\R$, let $\lfloor x\rfloor$, $\lceil x\rceil$, and $\langle x \rangle\in\R/\Z$ be the largest integer not exceeding $x$, the smallest integer not smaller than $x$, and  the fractional part of $x$, respectively.
For $i=1,2$, let $X_i$ be a complete subspace of a finite dimensional Euclidean space endowed with the metric $d_i$ induced by the $\ell_1$-norm $|\cdot|$.
For instance $X_i$ can be a sphere or a torus in which case the metric induced by $|\cdot|$ is equivalent to the standard geodesic distance on $X_i$. For $i=1,2$, let $\pi_i$ denote the natural projection onto the $i$-th coordinate of the product space $X_1\times X_2$. For any $k\in\N$ and $x\in\R^k$, let $|x|$ denote the $1$-norm of $x$, $\delta_x$ denote the Dirac measure at $x$, and $\mathfrak{m}$ be the Lebesgue measure on $\R^k$; here we omit the dependence of $\mathfrak{m}$ on the dimension $k$. For any set $A\subseteq\R^{k}$, let $\overline{A}$ and $\overset{\circ}{A}$ denote its closure and interior, respectively. Let
$\textrm{Diam} A\colon=\sup_{x,y\in A}|x-y|$ be its diameter (for convention, $\Diam A=0$ if $\# A\le1$).  We use $\lambda|A$ to denote the uniform (probability) measure over $A$ whenever appropriate (e.g., when $A$ is either bounded but uncountable or finite and countable). Let $\mathbbm{1}_A$ be the indicator function on $A$. Let $B\subseteq\R^{k}$. We say $A$ is \emph{compactly embedded} in $B$ and denoted $A\subset\subset B$ if $\overline{A}\subseteq\overset{\circ}{B}$.

\subsection*{Spaces of functions on metric spaces}

A function $f\colon X_1\to X_2$ is \emph{bounded} if $f(X_1)\subseteq X_2$ is bounded.
Let $(\mathcal{B}(X_1,X_2),d_{\infty})$ be the space of bounded measurable functions $f\colon X_1\to X_2$ equipped with the uniform metric
\[d_{\infty}(f,g)=\sup_{x\in X_1}d_2(f(x),g(x)).\]
Let $\mathcal{C}(X_1,X_2)$ ($\mathcal{C}_{\sf b}(X_1,X_2)$, $\mathcal{C}_{0}(X_1,X_2)$, respectively) be the space of continuous functions (bounded continuous functions, continuous functions with compact support, respectively) from $X_1$ to $X_2$ equipped with the same uniform metric. Recall that both $(\mathcal{B}(X_1,X_2),d_{\infty})$ and $(\mathcal{C}_b(X_1,X_2),d_{\infty})$ are complete provided $(X_2,d_2)$ is a complete metric space. Hence $\mathcal{C}(X_1,X_2)=\mathcal{C}_b(X_1,X_2)$ is complete provided $X_1$ is compact.

Let $\mathcal{L}(X_1,X_2)\colon=\{g\in\mathcal{C}(X_1,X_2)\colon \mathcal{L}(g)=\sup_{x\neq y}\frac{d_2(g(x),g(y))}{d_1(x,y)}<\infty\}$ be the space of Lipschitz continuous functions from $X_1$ to $X_2$. Hence $\mathcal{BL}(X_1,X_2)=\mathcal{B}(X_1,X_2)\cap\mathcal{L}(X_1,X_2)$ denotes the space of bounded Lipschitz continuous functions. In particular, when $X_2=\R$, we suppress $X_2$ in $\mathcal{B}(X_1,X_2)$ and simply write $\mathcal{B}(X_1)$. Similarly, we write $\mathcal{C}(X_1)$ for $\mathcal{C}(X_1,\R)$, etc. Let $\mathcal{B}_1(X_1)=\{g\in \mathcal{B}(X_1)\colon\|g\|_{\infty}=\sup_{x\in X_1}|g(x)|\le1\}$, $\mathcal{L}_1(X_1)=\{g\in \mathcal{L}(X_1)\colon\mathcal{L}(g)\le1\}$, and $\mathcal{BL}_1(X_1)=\{g\in \mathcal{BL}(X_1)\colon \mathcal{BL}(g)=\|g\|_{\infty}+\mathcal{L}(g)\le1\}$.

\subsection*{Measure theory}
Let $i=1,2$. With a Borel (probability) measure $\mu_{X_i}$ on $X_i$, $(X_i,\mathfrak{B}(X_i),$ $\mu_{X_i})$ becomes a  Polish (probability) measure space. Let $\cM_{+}(X_i)$ be the set of all finite positive Borel measures on $X_i$ and $\cP(X_i)$ the set of all Borel probability measures on $X_i$. Let $\cM_{+,\abs}(X_i)\subseteq\cM_+(X_i)$ the set of finite positive absolutely continuous measures w.r.t. $\mu_{X_i}$. Let ${\sf L}^1(X_i,\mu_{X_i})$ denote the set of integrable functions w.r.t. $\mu_{X_i}$.
 For every $\mu\in\cM_+(X_i)$, let $\supp \mu$ be the support of $\mu$. For $f\in \mathcal{C}_b(X_i)$, denote
\[\mu(f)=\int_{X_i} f\mathrm{d}\mu.\]

Recall for $\mu_1,\ \mu_2\in\cM_+(X_i)$,  $\mu_1$ is \emph{absolutely continuous} with respect to $\mu_2$ and denoted $\mu_1\ll\mu_2$, if $\mu_2(A)=0$ implies that $\mu_1(A)=0$ for all $A\in\mathfrak{B}(X_i)$.

\begin{definition}
Given a set $A\subseteq X_1^2$. The set $A^*=\{(x,y)\in X_1^2\colon (y,x)\in A\}$ is called the \emph{dual of $A$}.
\end{definition}
\begin{definition}
Given a measure $\eta\in\cM_+(X_1^2)$. The measure $\eta^*$ defined by
$$\eta^*(A)=\eta(A^*),\quad \forall A\in\mathfrak{B}(X_1^2),$$ is called the \emph{dual of $\eta$}.
\end{definition}

\subsection*{Measure metrics}

For every $\eta\in\cM_+(X_1)$, let $$\|\eta\|_{\TV}=\sup_{A\in\mathfrak{B}(X_1)}\eta(A)=\eta(X_1)$$ be the total variation \emph{norm} of $\eta$. Recall that $\|\cdot\|_{\TV}$ is a norm for the Banach space of all finite signed Borel measures \cite{B07}.

The total variation norm induces the total variation metric:
\[d_{\sf TV}(\eta^1,\eta^2)=\sup_{A\in\mathfrak{B}(X_1)}|\eta^1(A)-\eta^2(A)|
=\sup_{f\in\mathcal{B}_1(X_1)}\int f\rd(\eta^1-\eta^2),\quad \eta^1,\ \eta^2\in\cM_+(X_1).\]

For every $\eta\in\cM_{+,\abs}(X_1)$, let $\rho_{\eta}\colon=\frac{\rd\eta}{\rd\mu_{X_1}}$ denote the Radon-Nikodym derivative w.r.t. the reference measure $\mu_{X_1}$.
\begin{proposition}\label{prop-total}
For every $\eta^1,\ \eta^2\in\cM_{+,\abs}(X_1)$,
  \[d_{\sf TV}(\eta^1,\eta^2)=\tfrac{1}{2}\sup_{f\in\mathcal{B}_1(X_1)}\lt|\int_{X_1}f(x)(\rho_{\eta^1}-\rho_{\eta^2})
  \rd\mu_{X_1}(x)\rt|
  =\tfrac{1}{2}\|\rho_{\eta^1}-\rho_{\eta^2}\|_{L^1(X_1,\mu_{X_1})}.\]
\end{proposition}
\begin{proof}
  The proof is standard assuming $\eta^1$ and $\eta^2$ are probabilities \cite{RR04}. Following a similar argument as in \cite{RR04}, we rigourously prove this conclusion without this assumption. Let $A=\{\rho_{\eta^1}>\rho_{\eta^2}\}$. Let $g=\mathbbm{1}_{A}-\mathbbm{1}_{X_1\setminus A}$. By definition, $\|g\|_{\infty}\le1$.
\begin{align}\label{Eq-5}
  d_{\sf TV}(\eta^1,\eta^2)\ge&|\eta^1(A)-\eta^2(A)|\\
  \nonumber=&\tfrac{1}{2}\int_{X_1}(\rho_{\eta^1}-\rho_{\eta^2})g\rd\mu_{X_1}\\
  \nonumber=&\tfrac{1}{2}\sup_{f\in\mathcal{B}_1(X_1)}\lt|\int_{X_1}f(x)(\rho_{\eta^1}-\rho_{\eta^2})\rd\mu_{X_1}(x)\rt|\\
  \label{Eq-7}\ge&\sup_{B\in\mathfrak{B}(X_1)}|\eta^1(B)-\eta^2(B)|=d_{\sf TV}(\eta^1,\eta^2),
\end{align}
where \eqref{Eq-7} holds since $$|\eta^1(B)-\eta^2(B)|=\tfrac{1}{2}\left|\int_{X_1}(\mathbbm{1}_B-\mathbbm{1}_{X_1\setminus B})(\rho_{\eta^1}-\rho_{\eta^2})\right|$$ and $|\mathbbm{1}_B-\mathbbm{1}_{X_1\setminus B}|\le1$. In summary, we have shown that the two inequalities \eqref{Eq-5} and \eqref{Eq-7} are also equalities.
\end{proof}
Define the bounded Lipschitz norm (on the space of all finite signed Borel measures):
\[\|\nu\|_{\sf BL}\colon=\sup_{f\in\mathcal{BL}_1(X_1)}\int_{X_1}f\rd\nu,\quad \nu\in\cM_+(X_1),\]
which induces the bounded Lipschitz distance: For $\nu^1,\ \nu^2\in\cM_+(X_1)$,
\[d_{\sf BL}(\nu^1,\nu^2)=\sup_{f\in\mathcal{BL}_1(X_1)}\int f(x)\mathrm{d}(\nu^1(x)-\nu^2(x)).\]
In particular, if $\nu^1(X_1)=\nu^2(X_1)$, then $d_{\sf BL}$ is equivalent to the Kantorovich-Rubinstein  metric \cite{B07}:
\[d_{\sf KR}(\nu^1,\nu^2)=\sup_{f\in\mathcal{L}_1(X_1)}\int f(x)\mathrm{d}(\nu^1(x)-\nu^2(x))\]
such that $$d_{\sf BL}(\nu^1,\nu^2)\le \nu^1(X_1)^{-1}d_{\sf KR}(\nu^1,\nu^2)\le2\min\{d_{\sf BL}(\nu^1,\nu^2),d_{\sf TV}(\nu^1,\nu^2)\}.$$
Moreover, $d_{\sf BL}$ also metrizes the weak-$*$ topology on $\cM_+(X_1)$ \cite[Theorem~8.3.2]{B07} and $(\cM_+(X_1),$ $d_{\sf BL})$ is a Polish space \cite[Theorem~8.9.4]{B07}.

\subsection*{Relation between measures on product spaces and measure-valued functions}

The reference measure of the product space $X_1\times X_2$ is $\mu_{X_1}\otimes\mu_{X_2}$ via  Carath\'{e}odory's extension. For every $\eta\in\cM_+(X_1\times X_2)$ such that  its first marginal $\eta\circ \pi_1^{-1}\ll\mu_{X_1}$,  $$\eta=\mu_{X_1}\otimes\eta_x$$ is understood in the integral sense \cite[Chap.1]{B14}:
\[\int_{X_1\times X_2}f\rd\eta=\int_{X_1}\int_{X_2}f(x,y)\rd\eta_x(y)\rd\mu_{X_1}(x),\quad \forall f\in \mathcal{C}_b(X_1\times X_2),\] where $\eta_x$ is called the \emph{fiber measure}.

When $X_1$ is compact, we have $\mathcal{C}_b(X_1,\cM_+(X_2))=\mathcal{C}(X_1,\cM_+(X_2))$. Let $$\mathcal{B}_{\mu_{X_1},1}(X_1,\cM_+(X_2))=\{\eta\in \mathcal{B}(X_1,\cM_+(X_2))\colon \|\eta_x(X_2)\|_{{\sf L}^1(X_1,\mu_{X_1})}=1\}.$$
Analogously, let  $\mathcal{C}_{\mu_{X_1},1}(X_1,\cM_+(X_2))=\mathcal{C}(X_1,\cM_+(X_2))\cap \mathcal{B}_{\mu_{X_1},1}(X_1,\cM_+(X_2))$.

By Proposition~\ref{prop-fibercomplete} below, one can identify every $\eta\in \mathcal{B}(X_1,\cM_+(X_2))$ with a finite measure $\mu_{X_1}\otimes\eta_x\in\cM_+(X_1\times X_2)$, and $\eta\in \mathcal{B}_{\mu_{X_1},1}(X_1,\cM_+(X_2))$ with a finite measure $\mu_{X_1}\otimes\eta_x\in\cP(X_1\times X_2)$. Nevertheless, since the metric $d_{\infty}$ defined in \eqref{fibermetricBL} below is stronger than $d_{\sf BL}$ inducing the weak topology on $\cM_+(X_1\times X_2)$, two measure-valued functions $\eta^1,\ \eta^2\in \mathcal{B}_{\mu_{X_1}}(X_1,\cM_+(X_2))$ identify with a same finite measure in $\cM_+(X_1\times X_2)$ provided $\mu_{X_1}(\{\eta^1\neq\eta^2\})=0$. Hence, we will slightly abuse any measure-valued function $\eta\in \mathcal{B}(X_1,\cM_+(X_2))$ for a measure $\mu_{X_1}\otimes\eta_x$ in $\cM_+(X_1\times X_2)$.

Therefore, every function in $\mathcal{B}_{\mu_{X_1},1}(X_1,\cM_{+}(X_2))$ can be identified with a probability measure in $\cP(X_1\times X_2)$, i.e., $\mathcal{B}_{\mu_{X_1},1}(X_1,\cM_{+}(X_2))\hookrightarrow\cP(X_1\times X_2)$. Analogously, we have $\mathcal{B}(X_1,\cM_+(X_2))\hookrightarrow\cM_+(X_1\times X_2)$, etc.

For $\eta\in\mathcal{B}(X_1,\cM_+(X_2))$, let $$\|\eta\|=\sup_{x\in X_1}\|\eta_x\|_{\TV}.$$
Hence given $\eta^1,\eta^2\in \mathcal{B}(X_1,\cM_+(X_2))$, define the \emph{uniform bounded Lipschitz metric}:
\begin{equation}\label{fibermetricBL}
  d_{\infty}(\eta^1,\eta^2)=\sup_{x\in X}d_{\sf BL}(\eta^1_x,\eta^2_x).
\end{equation}
Since $(\cM_+(X_2),d_{\sf BL})$ is complete, both $\mathcal{B}(X_1,\cM_+(X_2))$ and $\mathcal{C}(X_1,\cM_+(X_2))$ equipped with the uniform bounded Lipschitz metric are complete.

We simply denote $\eta_x^i$ for $(\eta^i)_x$, and write $d_{\sf BL}$ for the Lipschitz bounded metric for $\cM_+(X_1\times X_2)$.

In the following proposition, we compare the uniform bounded Lipschitz distance between two measure-valued functions $\eta^1,\ \eta^2\in \mathcal{B}(X_1,\cM_+(X_2))$ as well as the bounded Lipschitz distance of the two measures in $\cM_+(X_1\times X_2)$ identified with $\eta^1,\ \eta^2$.
\begin{proposition}\label{prop-1}
Let $\eta^1,\ \eta^2\in \mathcal{B}(X_1,\cM_+(X_2))$. Then
  $$d_{\infty}(\eta^1,\eta^2)\ge d_{\sf BL}(\eta^1,\eta^2).$$ In other words, the convergence induced by the uniform bounded Lipschitz metric is no weaker than the weak convergence.
\end{proposition}
\begin{proof}
For any $f\in\mathcal{BL}(X_1\times X_2)$, $x\in X_1$, we have $f(x,\cdot)\in\mathcal{BL}(X_2)$. Note that
  \begin{align*}
    d_{\sf BL}(\eta^1,\eta^2)=&\sup_{f\in\mathcal{BL}(X_1\times X_2)}\int_{X_1\times X_2}f(x,y)\mathrm{d}(\eta^1(x,y)-\eta^2(x,y))\\
    =&\sup_{f\in\mathcal{BL}(X_1\times X_2)}\int_{X_1\times X_2}f(x,y)\mathrm{d}(\eta^1_x(y)-\eta^2_x(y))\mathrm{d}\mu_{X_1}(x)\\
    \le&\sup_{f\in\mathcal{BL}(X_1\times X_2)}\int_{X}d_{\sf BL}(\eta^1_x,\eta^2_x)\mathrm{d}\mu_X(x)\le \sup_{x\in X_1}d_{\sf BL}(\eta^1_x,\eta^2_x)=d_{\infty}(\eta^1,\eta^2).
  \end{align*}
\end{proof}

Indeed, $d_{\infty}$ can induce a stronger topology than the weak topology in $\cM_+(X_1\times X_2)$.
\begin{example}\label{ex-stronger-topology}
Let $X_1=X_2=[0,1]$ with $\mu_{X_1}=\mu_{X_2}=\lambda|_{[0,1]}$. For $n\in\N$, let $$f(x)=1-\sqrt{1-x^2},\quad f_n(x)=\begin{cases}
  x,\quad \textnormal{if}\quad 0\le x\le 1-1/n,\\
  -(n-1)(x-1),\quad \textnormal{if}\quad 1-1/n<x\le 1,
\end{cases}\ \ x\in[0,1].$$ Then $f(x),f_n(x)\in[0,1]$ and $\{x\in X\colon f_n(x)\neq f(x)\}=]1-1/n,1]$. Let $\eta^n=\mu_X\otimes\delta_{f_n(x)}$ and $\eta=\mu_X\otimes\delta_{f(x)}$. It is easy to see that $\eta^n,\eta\in \mathcal{C}(X,\cM_+(X))$. Moreover, $$d_{\infty}(\eta,\eta^n)=d_{\sf BL}(\eta^n_1,\eta_1)=1,\quad \textnormal{for}\ n\in\N.$$ Hence $\lim_{n\in\N}d_{\infty}(\eta^n,\eta)=1$. On the other hand,
\begin{align*}
  d_{\sf BL}(\eta^n,\eta)=&\sup_{g\in\BL_1([0,1]^2)}\int_0^1(g(x,f(x))-g(x,f_n(x)))\rd x\\
  \le&\int_0^1|f(x)-f_n(x)|\rd x\\
  =&\int_{]1-1/n,1]}|f(x)-f_n(x)|\rd x\\
  \le&2/n\to0,\quad \textnormal{as}\quad n\to\infty,
\end{align*}
which implies that $\lim_{n\to\infty}d_{\sf BL}(\eta^n,\eta)=0$. This shows that $d_{\infty}$ does induce a stronger topology than $d_{\sf BL}$.
\end{example}

Next, we provide some properties of the above function spaces which play an important role in the proof of the main results in subsequent sections.
\begin{proposition}\label{prop-fibercomplete}
For $i=1,2$, let $X_i$ be a complete subspace of a finite dimensional Euclidean space. Assume $(X_1,\mathfrak{B}(X_1),\mu_{X_1})$ is a compact probability space.
\begin{enumerate}
\item[\textnormal{(i)}] For every $\eta_{\cdot}\in \mathcal{B}(X_1,\cM_+(X_2))$, $\|\eta\|<\infty$.
\item[\textnormal{(ii)}] $(\mathcal{B}(X_1,\cM_+(X_2)), d_{\infty})$ and $(\mathcal{B}_{\mu_{X_1},1}(X_1,\cM_+(X_2)), d_{\infty})$ are complete metric spaces. In particular, $(\mathcal{C}_b(X_1,\cM_+(X_2)), d_{\infty})$ and $(\mathcal{C}_{\mu_{X_1},1}(X,\cM_+(X_2)), d_{\infty})$ are so.
\end{enumerate}
\end{proposition}

\begin{proof}
\noindent(i). Let $x\in X_1$. Since $\eta_{\cdot}\in \mathcal{B}(X_1,\cM_+(X_2))$,  $$\sup_{y\in X}\sup_{f\in\mathcal{BL}_1(X_2)}\lt|\int_{X_2}f\rd(\eta_x-\eta_y)\rt|=\sup_{y\in X_1}d_{\sf BL}(\eta_x,\eta_y)<\infty.$$ Taking $f=\mathbbm{1}_Z$ yields $$\|\eta\|=\sup_{y\in X_1}\eta_{y}(X_2)\le\eta_x(X_2)+\sup_{y\in X_1}d_{\sf BL}(\eta_x,\eta_y)<\infty,$$ since $\eta_x\in\cM_+(X_2)$.

\noindent(ii). It suffices to show $(\mathcal{B}_{\mu_X,1}(X_1,\cM_+(X_2)), d_{\infty})$ is closed. For any Cauchy sequence $(\eta^n)\subseteq (\mathcal{B}_{\mu_X,1}(X_1,\cM_+(X_2)), d_{\infty})$, since $\mathcal{B}(X_1,\cM_+(X_2))$ is complete, there exists $\eta\in \mathcal{B}(X_1,\cM_+(X_2))$ such that $d_{\infty}(\eta,\eta^n)\to0$ as $n\to\infty$.  Since $(\eta^n)$ is Cauchy, in the light of the proof of case (i), there exists $N'\in\N$ such that for all $m,n\ge N'$,
\begin{equation*}\sup_{x\in X_1}|\eta_x^m(X_2)-\eta^n_x(X_2)|\le\sup_{x\in X_1}d_{\sf BL}(\eta^m_x,\eta^n_x)<1,\end{equation*} which implies that
$$0\le\eta^n_x(X_2)\le\max\left\{\eta^{N'}_x(X_2)+1,\max_{1\le j\le N'-1}\eta^j_x(X_2)\right\},\quad \forall x\in X_1.$$ Since $\eta^j\in {\sf L}^1(X_1,\mu_{X_1})$ for all $1\le j\le N'$, and $\int_{X_1}\eta^n_x(X_2)\mu_{X_1}(x)=1$ for all $n\in\N$, by Dominated Convergence Theorem, \[\int_{X_1}\eta_x(X_2)\mu_{X_1}(x)=\lim_{n\to\infty}\int_{X_1}\eta^n_x(X_2)\mu_{X_1}(x)=1,\]
i.e., $\eta\in \mathcal{B}_{\mu_{X_1},1}(X_1,\cM_+(X_2))$.

Since the intersection of closed sets are closed, $(\mathcal{C}_{\mu_{X_1},1}(X_1,\cM_+(X_2)), d_{\infty})$ is also complete.
\end{proof}

\begin{definition}
For $i=1,2$, let $X_i$ be a complete subspace of a finite dimensional Euclidean space. Assume $(X_1,\mathfrak{B}(X_1),\mu_{X_1})$ is a compact probability space.
For $$\mathcal{B}_{\mu_X,1}(X_1,\cM_+(X_2))\ni\eta\colon\begin{cases}
  X_1\to \cM_+(X_2),\\ x\mapsto\eta_x,
\end{cases}$$ $\eta$ is  \emph{weakly continuous} if for every $f\in\mathcal{C}_{\sf b}(X_2)$,
\[\mathcal{C}(X_1)\ni\eta(f)\colon\begin{cases}
  X_1\to\R,\\ x\mapsto\eta_x(f)\colon=\int_{X_2}f\rd\eta_x.
\end{cases}\]
\end{definition}

\begin{definition} For $i=1,2$, let $X_i$ be a complete subspace of a finite dimensional Euclidean space. Assume $(X_1,\mathfrak{B}(X_1),\mu_{X_1})$ is a compact probability space.
Let $\mathcal{I}\subseteq\R$ be a compact interval. For $$\eta_{\cdot}\colon\begin{cases}
  \mathcal{I}\to \mathcal{B}(X_1,\cM_+(X_2)),\\ t\mapsto\eta_t,
\end{cases}$$ $\eta_{\cdot}$ is  \emph{uniformly weakly continuous} if for every $f\in \mathcal{C}_{\sf b}(X_2)$, $t\mapsto(\eta_t)_x(f)$ is continuous in $t$ uniformly in $x\in X_1$.
\end{definition}
By slightly abusing the notation, for $\eta_{\cdot}\in\mathcal{C}(\mathcal{I},\mathcal{B}(X_1,\cM_+(X_2)))$, let $$\|\eta_{\cdot}\|=\sup_{t\in\mathcal{I}}\sup_{x\in X_1}\|(\eta_t)_x\|_{\TV}.$$

The following proposition unveils the relation between continuity and (uniform) weak continuity.
\begin{proposition}\label{prop-nu}
For $i=1,2$, let $X_i$ be a complete subspace of a finite dimensional Euclidean space. Assume $(X_1,\mathfrak{B}(X_1),\mu_{X_1})$ is a compact probability space. Let $\mathcal{I}\subseteq\R$ be a compact interval.
\begin{enumerate}
\item[\textnormal{(i)}] Let $\eta_{\cdot}\colon \mathcal{I}\to \mathcal{B}_{\mu_{X_1},1}(X_1,\cM_+(X_2))$. Then $\eta_{\cdot}$ is uniformly weakly continuous if and only if $\eta_{\cdot}\in \mathcal{C}(\mathcal{I},\mathcal{B}_{\mu_{X_1},1}(X_1,\cM_+(X_2)))$.
\item[\textnormal{(ii)}] Assume $\eta_{\cdot},\ \xi_{\cdot}\in \mathcal{C}(\mathcal{I},\mathcal{B}_{\mu_{X_1},1}(X_1,\cM_+(X_2)))$, then
    $\|\eta_{\cdot}\|<\infty$ and $t\mapsto d_{\infty}(\eta_t,\xi_t)$ is continuous.
\item[\textnormal{(iii)}] Assume $\eta\in \mathcal{C}(X_1,\cM_+(X_2))$. Then $\eta$ is weakly continuous.
\end{enumerate}

\end{proposition}
\begin{proof}
\noindent{(i)}
\begin{enumerate}
\item[Step I.] Uniform weak continuity implies continuity. Assume $\eta_{\cdot}$ is uniformly weakly continuous. Fix $t\in\mathcal{I}$ and $\mathcal{I}\supseteq t_j\to t$. Since $\eta_{\cdot}$ is uniformly weakly continuous, for every $f\in \mathcal{C}_{\sf b}(X_2)$, $$(\eta_{t_j})_{x}(f)\to(\eta_t)_{x}(f)\quad \text{uniformly in}\ x.$$
Since $X_2$ is Polish,  $d_{\sf BL}$ metrizes the weak-$*$ topology of $\cM_+(X_2)$ \cite[Thm.~8.3.2]{B07}, and we have
\[\lim_{j\to\infty}d_{\sf BL}((\eta_{t_j})_x,(\eta_t)_x)=0,\]
and the convergence is uniform in $x\in X_1$. This means that
\[\lim_{j\to\infty}\sup_{x\in X_1}d_{\sf BL}((\eta_{t_j})_x,(\eta_t)_x)=0,\]
i.e., $$\lim_{j\to\infty}d_{\infty}(\eta_{t_j},\eta_t)=0.$$
This shows $\eta_{\cdot}\in\mathcal{C}(\mathcal{I},\mathcal{B}_{\mu_{X_1},1}(X_1,\cM_+(X_2)))$.

\item[Step II.] Continuity implies uniform weak continuity. Assume $\eta_{\cdot}\in\mathcal{C}(\mathcal{I},\mathcal{B}_{\mu_{X_1},1}(X_1,\cM_+(X_2)))$. For every fixed $t\in\mathcal{I}$ and $\mathcal{I}\supseteq t_j\to t$, we have
\[\lim_{j\to\infty}d_{\infty}(\eta_{t_j},\eta_t)=0.\] Hence  $$d_{\sf BL}((\eta_{t_j})_x,(\eta_t)_x)\to0,\quad \text{uniformly in}\ x\in X_1.$$ Since $d_{\sf BL}$ metrizes the weak-$*$ topology of $\cM_+(X_2)$,
$$(\eta_{t_j})_{x}(f)\to(\eta_t)_{x}(f)\quad \forall\ f\in \mathcal{C}_{\sf b}(X_2),$$ also uniformly in $x\in X_1$. This shows that $\eta_{\cdot}$ is uniformly weakly continuous.
\end{enumerate}
\noindent{(ii)}
Since $\mathcal{I}$ is compact, and $(\mathcal{B}_{\mu_{X_1},1}(X_1,\cM_+(X_2)),d_{\infty})$ is complete by Proposition~\ref{prop-fibercomplete}(ii),  we have $\eta_{\cdot}$ is uniformly bounded: $$\sup_{t\in\mathcal{I}}\sup_{x\in X_1}\sup_{f\in\mathcal{BL}_1(X_2)}\lt|\int_{X_2}f\rd((\eta_t)_x-(\eta_0)_x)\rt|
=\sup_{t\in\mathcal{I}}d_{\infty}(\eta_0,\eta_t)<\infty.$$ Letting $f=\mathbbm{1}_{X_2}$ yields
\[\|\eta_{\cdot}\|\le\|\eta_0\|+\sup_{t\in\mathcal{I}}d_{\infty}(\eta_0,\eta_t)<\infty,\] since $\eta_0\in \mathcal{B}_{\mu_{X_1},1}(X_1,\cM_+(X_2))$.

Finally we show $t\mapsto d_{\infty}(\eta_t,\xi_t)$ is continuous, which directly follows from the following triangle inequality: For $s,t\in\mathcal{I}$,
\begin{align*}
  |d_{\infty}(\xi_t,\eta_t)-d_{\infty}(\xi_s,\eta_s)|
  \le&|d_{\infty}(\xi_t,\eta_t)-d_{\infty}(\xi_t,\eta_s)|+|d_{\infty}(\xi_t,\eta_s)-d_{\infty}(\xi_s,\eta_s)|\\
  \le& d_{\infty}(\eta_t,\eta_s)+d_{\infty}(\xi_t,\xi_s)\to0,\quad \text{as}\ |t-s|\to0,
\end{align*}since $\eta_{\cdot},\ \xi_{\cdot}\in \mathcal{C}(\mathcal{I},\mathcal{B}_{\mu_{X_1},1}(X_1,\cM_+(X_2)))$.

\noindent{(iii)} The argument in Step II in (i) applies by replacing $t$ by $x$ as well as $|t-s|$ by $|x-y|$.
\end{proof}

\begin{definition}\label{def-alphametric}
For $i=1,2$, let $X_i$ be a complete subspace of a finite dimensional Euclidean space. Assume that $(X_1,\mathfrak{B}(X_1),\mu_{X_1})$ is a compact probability space.  Let $\mathcal{I}\subseteq\R$ be a compact interval and $\alpha>0$.  For $\nu^1_{\cdot},\nu^2_{\cdot}\in \mathcal{C}(\mathcal{I},\mathcal{B}_{\mu_{X_1},1}(X_1,\cM_+(X_2)))$, let
  \[d_{\alpha}(\nu^1_{\cdot},\nu^2_{\cdot})=\sup_{t\in\mathcal{I}}\textnormal{e}^{-\alpha t}d_{\infty}(\nu^1_t,\nu^2_t)\]be a \emph{weighted} uniform metric.
\end{definition}
These metrics are going to be used below to establish the contraction of a mapping used in the unique existence of a fixed point equation.
\begin{proposition}\label{alpha-complete}
For $i=1,2$, let $X_i$ be a complete subspace of a finite dimensional Euclidean space. Assume that $(X_1,\mathfrak{B}(X_1),\mu_{X_1})$ is a compact probability space.  Let $\mathcal{I}\subseteq\R$ be a compact interval and $\alpha>0$. Then $(\mathcal{C}(\mathcal{I},\mathcal{B}_{\mu_{X_1},1}(X_1,\cM_+(X_2))),d_{\alpha})$ and $(\mathcal{C}(\mathcal{I},\mathcal{C}_{\mu_{X_1},1}(X_1,$ $\cM_+(X_2))),d_{\alpha})$ are both complete.
\end{proposition}
\begin{proof}
Since the weighted uniform metric $d_{\alpha}$
 is equivalent to the uniform metric $\sup_{t\in\mathcal{I}}\textnormal{e}^{-\alpha t}$ $\cdot d_{\infty}(\nu^1_t,\nu^2_t)$
 between $\nu^1_{\cdot}$ and $\nu^2_{\cdot}$, the conclusions yield immediately from Proposition~\ref{prop-fibercomplete}, since all continuous functions on $\mathcal{I}$ are bounded.
\end{proof}

Let $X_1=X$, $X_2=X$ or $X_2=\R^{r_2}$, and $\mathcal{I}=\cT$. The spaces $\mathcal{B}(\cT,\mathcal{B}_{\mu_X,1}(X,\cM_+(\R^{r_2})))$, $\mathcal{C}(\cT,\mathcal{B}_{\mu_X,1}(X,\cM_+(\R^{r_2})))$ and $\mathcal{C}(\cT$, $C_{\mu_X,1}(X,\cM_{\abs}(\R^{r_2})))$ will serve as the underlying spaces for initial probabilities of the generalized VEs (with the last in the sense of the classical VE), and $\mathcal{B}(X,\cM_+(X))$ and $\mathcal{C}_{\sf b}(X,\cM_+(X))$ will correspond to the space of \emph{generalized digraphs} (DGMs), as illuminated below.

\subsection*{Digraph measures}

Let $X$ be the vertex space. For any $\eta\in \mathcal{B}(X,\cM_+(X))$, the measure $\eta_x$ represents the ``edge'' from $x$ to other vertices in $X$. For instance, $\sup_{x\in X}\supp\eta_x<\infty$ is interpreted as  every vertex $x$ has uniformly finitely many outward directed edges while $\inf_{x\in X}\supp\eta_x=\infty$ means that every vertex connects infinitely many other vertices. Hence $\eta$ can be viewed as a \emph{digraph}.

We now classify digraph measures into sub-categories according to their \emph{denseness}. Similar notions have appeared in the literature, in particular we mention the recent theory of graphops (graph operators), where families of fiber measures associated to graphops, plays a key role in this regard \cite{BS20}. Our work is motivated directly by this theory, but as we have explained above, staying purely on the level of operator theory as in~\cite{GK20} leads to relatively strong requirements on the graphop, so a measure-theoretic viewpoint is a natural generalization/alternative.

\begin{definition}\label{DGM}
Any measure-valued function in $\mathcal{B}(X,\cM_+(X))$ is a \emph{digraph measure} (abbreviated as ``\emph{DGM}'').
\end{definition}

\begin{definition}\label{graphop}
  Let $\eta\in \mathcal{B}(X,\cM_+(X))$. We say $\eta$ is \emph{symmetric} w.r.t.~a reference measure $\mu_X\in\cP(X)$ if $\mu_X\otimes\eta_x\in\cM_+(X^2)$ is symmetric. A symmetric DGM is called a \emph{GM} for short, which is also called a \emph{graphop}.
\end{definition}

That a GM can be also viewed as a graphop \cite{BS20}, is due to \emph{Riesz representation theorem} \cite{BS20}. Yet, it turns out to be crucial from a technical perspective in the context of Vlasov equations, whether one works directly with measures or via operator-theoretic representations. Indeed, it is a known theme in PDEs that the choice of solution space is critical, so our setting can be viewed as another manifestation of this problem. Furthermore, we remark that the DGM slightly generalizes the notion of graphops to the asymmetric setting. Other notions in the literature \cite{BS20} can be also analogously extended to digraph measures.

\begin{definition}\label{graphon}
A DGM $\eta\in \mathcal{B}(X,\cM_+(X))$ is called a \emph{digraphon} w.r.t. a reference measure $\mu_X\in\cP(X)$ if $\mu_X\otimes\eta_x\in\cM_{+,\abs}(X^2)$. A symmetric digraphon is a \emph{graphon} \cite[Sec.8]{BS20}.
\end{definition}

\begin{remark}\label{re-generalization}
Let $W\in {\sf L}^1(X^2,\mu_X\otimes\mu_X)$ be graphon.
In \cite{KM18}, $X=[0,1]$, $\mu_X$ is the Lebesgue measure on $X$, and it is assumed that $W$ satisfies
\begin{equation}
\label{KM-condition}\lim_{z\to0}\int_I|W(x+z,y)-W(x,y)|\rd y=0,\quad \forall\ x\in X.
\end{equation} %By symmetry of graphons, this implies that $W\in \mathcal{C}(X^2)$. Let $\rd\eta_x(y)=W(x,y)\rd\mu_(y)$.
Then it is easy to show that $x\mapsto\eta_x$ is continuous since \begin{align*}
  d_{\sf BL}(\eta_x,\eta_{x'})=&\sup_{f\in\BL_1(X)}\int_Xf\rd(\eta_x-\eta_{x'})\\
  =&\sup_{f\in\BL_1(X)}\int_Xf(y)(W(x,y)-W(x',y))\rd\mu_X(y)\\
  \le&\|W(x,\cdot)-W(x',\cdot)\|_{{\sf L}^1(X,\mu_X)},
  \end{align*}
  as well as \eqref{KM-condition}. %the fact that the uniform convergence due to continuity implies the ${\sf L}^1$-convergence by Dominated Convergence Theorem.
  This demonstrates that our main results that follow assuming only continuity of DGMs does generalize the result in \cite{KM18} (except that technically the underlying phase space is different), %where $X=[0,1]$ and the reference measure $\mu_X$ is the Lebesgue measure on $X$ (see condition \cite[(1.17)]{KM18}),
  where the network has a graphon limit, or continuous graphops which can be approximated by graphons as in \cite{GK20}.
\end{remark}

Let $\mathcal{D}(X)\colon=\{\eta\in \mathcal{B}(X,\cM_+(X))\colon \sup_{x\in X}\supp\eta_x<\infty\}$.
\begin{definition}\label{graphing}
A DGM $\eta\in\cD(X)$ is called a \emph{digraphing}. A symmetric digraphing is a \emph{graphing} \cite[Sec.9]{L12,BS20}.
\end{definition}

In terms of denseness of a graph, a digraphon is \emph{dense} while a digraphing is \emph{sparse} \cite{BS20}.

In the following, we provide several examples to show the diversity of DGMs in $\mathcal{C}(X,\cM_+(X))$, in terms of the denseness as well as heterogeneity of the graphs.

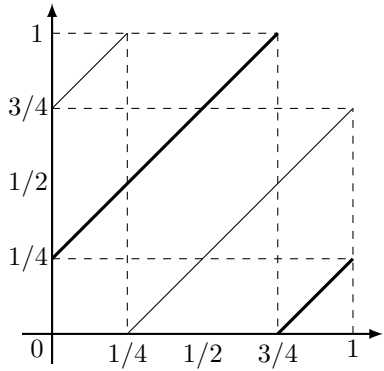
\begin{figure}[h]
\centering
{
\begin{tikzpicture}
    [%%%%%%%%%%%%%%%%%%%%%%%%%%%%%%
        dot/.style={circle,draw=black,fill,inner sep=1pt},scale=4
    ]%%%%%%%%%%%%%%%%%%%%%%%%%%%%%%
  \path[-] (0,3/4) edge (1/4,1);
  \path[-,very thick] (0,1/4) edge (3/4,1);
  \path[-] (1/4,0) edge (1,3/4);
  \path[-,very thick] (3/4,0) edge (1,1/4);
 \path[-,dashed] (1,0) edge (1,3/4);
  \path[-,dashed] (0,1) edge (3/4,1);
  \path[-,dashed] (1/4,1) edge (1/4,0);
  \path[-,dashed] (3/4,1) edge (3/4,0);
  \path[-,dashed] (0,3/4) edge (1,3/4);
  \path[-,dashed] (0,1/4) edge (1,1/4);
  %axes
\draw[->,thick,-latex] (0,-.1) -- (0,1.1);
\draw[->,thick,-latex] (-.1,0) -- (1.1,0);
\node[] at (-0.05,-0.05){0};
\node[] at (1/4,-0.07){1/4};
\node[] at (1/2,-0.07){1/2};
\node[] at (3/4,-0.07){3/4};
\node[] at (1,-0.05){1};
\node[] at (-0.08,1/4){1/4};
\node[] at (-0.08,1/2){1/2};
\node[] at (-0.08,3/4){3/4};
\node[] at (-0.05,1){1};
\end{tikzpicture}
\caption{Example~\ref{ex-circle}. The circle graphop $\eta$, when viewed as a measure on the product space $\mathbb{S}^1\times\mathbb{S}^1=\mathbb{T}^2\equiv[0,1[^2$, is a uniform measure over two disjoint circles (one thin and one thick) on $\mathbb{T}^2$.}\label{fig-circlegraphop}
}
\end{figure}

\begin{example}[Circle graphop]\label{ex-circle}
Let $X=\mathbb{S}^1$ be the unit circle, identified with $[0,1[$ by the mapping $x\mapsto \textnormal{e}^{2\pi \textnormal{i} x}$.
Define the \emph{circle graphop} $\eta$ by
$$\eta_x=\delta_{\langle x+1/4\rangle}+\delta_{\langle x-1/4\rangle},\quad x\in X.$$ By definition, the circle graphop is a graphing, and  we can view it as an abstract graph $G=(X,E)$ with $(x,y)\in E$ if and only if $x\perp y$. See Figure~\ref{fig-circlegraphop} for its illustration. One can take $\eta$ as a uniform measure over two disjoint circles on $\mathbb{T}^2$.
\end{example}

\begin{example}\label{ex-mix}
\begin{enumerate}
\item[(i)]    Let $X=\mathbb{S}^2$. For every $x\in\mathbb{S}^2$, let $x^{\perp}\colon=\{y\in\mathbb{S}^2\colon y\perp x\}$ be the circle on $\mathbb{S}^2$ perpendicular to $x$ and $\eta_x=\lambda|_{x^{\perp}}$ be the uniform measure over $x^{\perp}$. This definition yields a GM $\eta\colon x\mapsto\eta_x$ that can also be viewed as a graphop, called the \emph{spherical graphop}~\cite{BS20}. This graphop is neither a graphing nor a graphon. See Figure~\ref{fig-i-ii}(a).
\item[(ii)] Let $X=[0,1]$. For every $x\in X$, let $$A_x=\left[\frac{3}{2}x,1-\frac{3}{2}x\right]\mathbbm{1}_{[0,1/3[}(x)+\{1/2\}\mathbbm{1}_{[1/3,2/3]}(x)
    +\left[\frac{3}{2}(1-x),1-\frac{3}{2}
    (1-x)\right]\mathbbm{1}_{]2/3,1]}(x),$$ and $\eta_x=\lambda|_{A_x}$. Hence \[\eta_x\in\begin{cases}
      \cM_{+,\abs}(X),\quad \textnormal{if}\quad x\in[0,1/3[\,\cup\,]2/3,1],\\
      \cD(X),\quad\hspace{.9cm} \textnormal{if}\quad   x\in[1/3,2/3].
    \end{cases}\] Moreover, it is straightforward to verify that $$d_{\sf BL}(\eta_x,\eta_y)\le d_{\sf TV}(\eta_x,\eta_y)\to0,\quad \text{as}\ |x-y|\to0,$$ which implies $\eta\in \mathcal{C}(X;\cM_+(X))$.  This DGM $\eta$ is again neither dense nor sparse, but can be viewed as a measure\---a linear combination of an absolutely continuous measure supported on $([0,1/3]\cup[2/3,1])\times[0,1]$ and a singular measure supported on $[1/3,2/3]\times\{1/2\}$. See Figure~\ref{fig-i-ii}(b).
\item[(iii)] Let $X=\mathbb{S}^1$. Let ${\sf C}$ be the Cantor set on $[0,1]$. Let $F_{\zeta_{\sf C}}$ be the distribution function of the uniform measure $\zeta_{\sf C}$ over ${\sf C}$. For every $x\in X$, define a uniform measure $\eta_x$ over a Cantor-like set (see Figure~\ref{fig-iii}) within $[x,x+\tfrac{3}{4}[\!\!\mod 1\subsetneq X$ by its distribution function $$F_{\eta_x}(z)=F_{\zeta_{\sf C}}\left(\langle\frac{4}{3}(z-x)\rangle\right),\quad z\in [x,x+\tfrac{3}{4}[\!\!\!\!\mod 1.$$ For every $f\in \BL_1(X)$, extend it to be a periodic function on $\R$, we have\begin{align*}
      \lt|\int f\rd(\eta_x-\eta_y)\rt|=&\lt|\int_{[x,x+3/4[\!\!\!\!\mod 1}f(z)\rd F_{\zeta_{\sf C}}\left(\langle\frac{4}{3}(z-x)\rangle\right)\rt.\\&\lt.-\int_{[y,y+3/4[\!\!\!\!\mod 1}f(z)\rd F_{\zeta_{\sf C}}\left(\langle\frac{4}{3}(z-y)\rangle\right)\rt|\\
      =&\int_0^1\lt|f\left(x+\frac{3}{4}z\right)-f\left(y+\frac{3}{4}z\right)\rt|\rd F_{\zeta_{\sf C}}(z)\\
      \le&\int_{0}^1d^{X}(x,y)\rd F_{\zeta_{\sf C}}(z)=d^X(x,y),
    \end{align*} where $d^X(x,y)=\min\{|x-y|,1-|x-y|\}$ is the arc length between $x$ and $y$ on $X$. This shows that $x\mapsto\eta_x$ is continuous, by the supremum representation of the bounded Lipschitz metric.
    Hence $\eta\in\mathcal{C}(X;\cM_+(X))$ is a DGM which can be regarded as a uniform measure over a curve of Hausdorff dimension $(1+\frac{\log 2}{\log 3})$ on $\mathbb{T}^2$. Moreover, $\eta$ is a continuous (since $\zeta_{\sf C}$ is so) but \emph{not} absolutely continuous measure on $\mathbb{T}^2$.
    \end{enumerate}
\end{example}

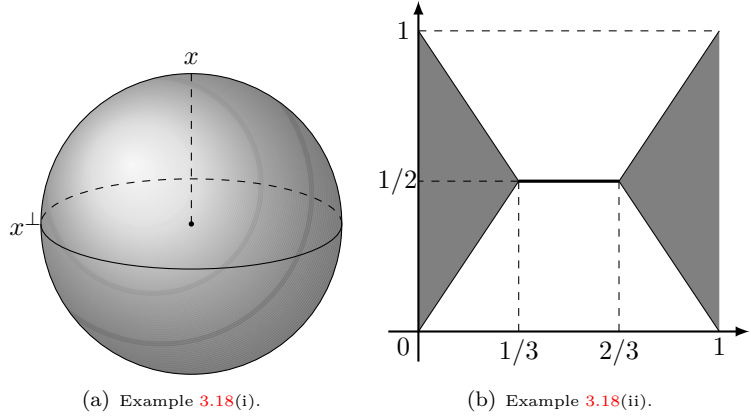
\begin{figure}[h]
\centering
{\captionsetup[subfigure]{justification=centering, width=6cm, belowskip=-3cm}

\subfigure[\tiny{Example~\ref{ex-mix}(i).}]{
\begin{tikzpicture}
  \shade[ball color = gray!40, opacity = 0.4] (0,0) circle (2cm);
  \draw (0,0) circle (2cm);
  \draw (-2,0) arc (180:360:2 and 0.6);
  \draw[dashed] (2,0) arc (0:180:2 and 0.6);
  \fill[fill=black] (0,0) circle (1pt);
  \draw[dashed] (0,0 ) -- node[above]{} (0,2);
  \node[] at (0,2.2) {$x$};
  \node[] at (-2.2,0) {$x^{\perp}$};
\end{tikzpicture}
}
\subfigure[\tiny{Example~\ref{ex-mix}(ii).}]{
\begin{tikzpicture}
    [%%%%%%%%%%%%%%%%%%%%%%%%%%%%%%
        dot/.style={circle,draw=black,fill,inner sep=1pt},scale=4
    ]%%%%%%%%%%%%%%%%%%%%%%%%%%%%%%
  \path[-] (0,1) edge (1/3,1/2);
  \path[-] (0,0) edge (1/3,1/2);
  \path[-,very thick] (1/3,1/2) edge (2/3,1/2);
  \path[-] (2/3,1/2) edge (1,0);
  \path[-] (2/3,1/2) edge (1,1);
  \path[-,dashed] (0,1) edge (1,1);
  \path[-,dashed] (2/3,1/2) edge (2/3,0);
  \path[-,dashed] (1/3,1/2) edge (1/3,0);
  \path[-,dashed] (1/3,1/2) edge (0,1/2);
  %axes
\draw[->,thick,-latex] (0,-.1) -- (0,1.1);
\draw[->,thick,-latex] (-.1,0) -- (1.1,0);
\fill[black, opacity=.5] (0,0) -- (1/3,1/2) -- (0,1) -- cycle;
\fill[black, opacity=.5] (1,0) -- (2/3,1/2) -- (1,1) -- cycle;
\node[] at (-0.05,-0.05){0};
\node[] at (1/3,-0.07){1/3};
\node[] at (2/3,-0.07){2/3};
\node[] at (1,-0.05){1};
\node[] at (-0.07,1/2){1/2};
\node[] at (-0.05,1){1};
\end{tikzpicture}
}
}
\caption{Illustration for Example~\ref{ex-mix}(i)-(ii). (a): For every vertex $x$ on the sphere, $\eta_x$ is the uniform measure over the circle $x^{\perp}$. (b): the shaded region is the support of the absolutely continuous part of the measure $\eta$ in Example~\ref{ex-mix}(ii) while the thick line between $(1/3,1/2)$ and $(2/3,1/2)$ is the support of its singular part.}\label{fig-i-ii}
\end{figure}

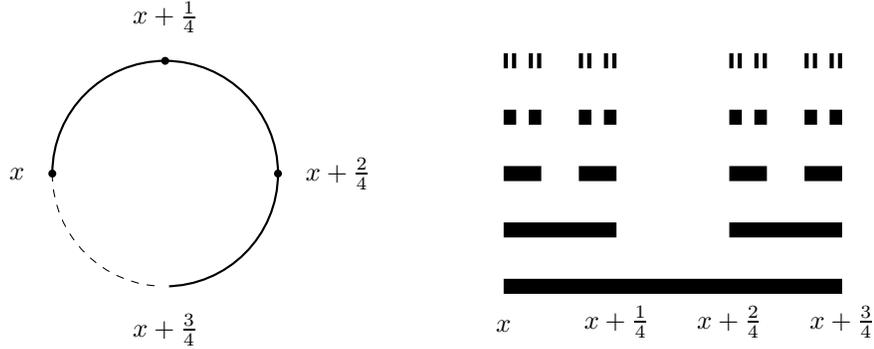
\begin{figure}
\begin{tikzpicture}[decoration=Cantor set,line width=2mm,scale=1.5,color=black]
  \draw[thick] (-3,0) arc (-90:180:1);
  \draw[dashed,thin] (-4,1) arc (180:270:1);
  \node[circle,fill=black,inner sep=0pt,minimum size=3pt,label=left:{$x$}] (a) at (-4,1) {};
  \node[circle,fill=black,inner sep=0pt,minimum size=3pt,label=above:{$x+\frac{1}{4}$}] (b) at (-3,2) {};
  \node[circle,fill=black,inner sep=0pt,minimum size=3pt,label=right:{$x+\frac{2}{4}$}] (c) at (-2,1) {};
  \node[circle,fill=white,inner sep=0pt,minimum size=3pt,label=below:{$x+\frac{3}{4}$}] (d) at (-3,0) {};
  \draw [black] (0,-.55)  node[anchor=south] {$x$};
  \draw [black] (1,-.6) node[anchor=south] {$x+\frac{1}{4}$};
  \draw [black] (2,-.6)  node[anchor=south] {$x+\frac{2}{4}$};
  \draw [black] (3,-.6)  node[anchor=south] {$x+\frac{3}{4}$};
  \draw (0,0) -- (3,0);
  \draw decorate{ (0,.5) -- (3,.5) };
  \draw decorate{ decorate{ (0,1) -- (3,1) }};
  \draw decorate{ decorate{ decorate{ (0,1.5) -- (3,1.5) }}};
  \draw decorate{ decorate{ decorate{ decorate{ (0,2) -- (3,2) }}}};
\end{tikzpicture}
\caption{Illustration of Example~\ref{ex-mix}(iii). Construction of a Cantor-like set on $[x+1/4,x+1[\mod 1$.
}\label{fig-iii}
\end{figure}

\begin{example}
Let $\zeta_{\sf C}^{-1}$ be the inverse measure of $\zeta_{\sf C}$, i.e., the \emph{quantile function} of $\zeta_{\sf C}^{-1}$ is $F_{\zeta_{\sf C}}$ \cite{XB19}. Note that $\zeta_{\sf C}^{-1}$ is discrete \cite{XB19}. In an analogous way as demonstrated in Example~\ref{ex-mix}(i), one can construct the following measure-valued function:
$$F_{\eta_x}(z)=F_{\zeta_{\sf C}^{-1}}\left(\langle\frac{4}{3}(z-x)\rangle\right),\quad x,z\in X.$$
Hence it is easy to verify that $\eta\in\mathcal{C}(X;\cM_+(X))$ is a DGM which can be regarded as a singular measure supported on countably smooth curves on $\mathbb{T}^2$. Note that in contrast to Example~\ref{ex-mix}(iii), $\eta$ is \emph{not} a continuous measure (in the sense of its joint distribution function) on $\mathbb{T}^2$.
\end{example}

\begin{remark}\label{re-not-dense}
In general, the space of continuous functions $\mathcal{C}(X,Y)$ is \emph{not dense} in the space of bounded functions $\mathcal{B}(X,Y)$ in the uniform metric. For instance, take $X=[0,1]$ and $Y=\R$. Let $f=\mathbbm{1}_{X\setminus\mathbb{Q}}$. Then $f$ cannot be approximated by any continuous function in the supremum norm. Hence given any $\xi\in\cP_{\abs}(Y)$, let $x\mapsto\eta_x=f(x)\xi$. Then $\eta\in \mathcal{B}_{\mu_X,1}(X,\cM_{+,\abs}(Y))$. It is obvious that $\eta$ cannot be approximated by any sequence in $\mathcal{C}_{\mu_X,1}(X,\cM_{+}(Y))$. This gives us a clue that using the uniform bounded Lipschitz metric, one may not expect approximation of VE on a DGM of arbitrary weak regularity (see Section~\ref{sect-approximation-ode}).
\end{remark}
\begin{remark}
  From these examples one can see that given a sequence of graphs, there will be different ways to represent each finite digraph as a digraph measure (a digraphon or a digraphing), even when the vertex space $X$ is prescribed. We here provide a specific situation to illustrate this point. Let $W^N=(W^N_{i,j})$ be the adjacency matrix of the digraph $G^N$ for every $N\in\N$. Assume $\sup_{N\in\N}\max_{1\le i\le N}\#\{j\colon W_{ij}\neq0\}$, the supremum of the maximum degree of the graphs, is finite. Let $X=[0,1]$. Then one can take the following DMG (which is a digraphing) as a representation of each $W^N$:
  $$\eta^{W^N}_x=\sum_{j=1}^NW^N_{i,j}\delta_{i/N},\quad x\in I_i^N,\quad i=1,\ldots,N,$$ where $I_i^N=\bigl[\frac{i-1}{N},\frac{i}{N}\bigr[$ for $i<N$ and $I_N^N=\bigl[1-\frac{1}{N},1\bigr]$. With this representation, one can obtain the limit of $\eta^{W^N}$ as a finitely supported measure valued function. In contrast, recall that there is another graphon representation for each $W^N$ \cite{L12}. However, if represented as a graphon, then the sequence of graphons converges to the zero digraphon defined on $[0,1]^2$. This also shows digraphons are limits of dense digraphs, which may not be suitable to characterize limits of non-dense digraphs.
\end{remark}

\section{Vlasov equation on the DGMs}\label{sect-setup}

In this section, we establish well-posedness of weak solutions to the VE \eqref{Vlasov}. To do this, we first study the equation of characteristics, namely \eqref{Charac-recall} below. Then we construct a fixed point equation via the solution map of the equation of characteristics. Using Lipschitz properties of the flow of the equation of characteristics, we prove the existence of a unique solution to the fixed point equation by the Banach contraction mapping theorem. Then, by establishing the connection between solutions to the fixed point equation and weak solutions of the VE, we prove the well-posedness of the VE in an indirect way. This idea originally is due to Nuenzert \cite{N84} and we provide a generalization of his method incorporating DGMs and carefuly associated choices of function spaces.

\subsection{Characteristic equation}

In this subsection, we will establish the Lipschitz continuity and continuity of the Vlasov operator $V[\eta,\nu_{\cdot},h]$ for the characteristic equation. Recall the characteristic equation: For every $x\in X$,
\begin{equation}
  \label{Charac-recall}
  \frac{\partial}{\partial t}\phi(t,x)=V[\eta,\nu_{\cdot},h](t,x,\phi),\quad t\in[s,T],\ \quad \phi(s,x)=\phi_s(x),\quad \text{for}\ 0\le s<T,
\end{equation}
where $V[\eta,\nu_{\cdot},h]$ is the Vlasov operator defined by
\begin{equation*}
V[\eta,\nu_{\cdot},h](t,x,\phi)=\sum_{i=1}^r\int_X\int_{\R^{r_2}}g_i(t,\psi,\phi)\rd(\nu_t)_y(\psi)\rd\eta^i_x(y)+h(t,x,\phi),
\end{equation*}%and
for $t\in\cT$, $x\in X$, and $\phi\in \R^{r_2}$. We first establish properties of the Vlasov operator.

\begin{proposition}
  \label{Vlasovf}
  Assume $\mathbf{(A1)}$-$\mathbf{(A5)}$. Then $V[\eta,\nu_{\cdot},h](t,x,\phi)$ is
  \begin{enumerate}
  \item[\textnormal{(i)}] continuous in $t$,
  \item[\textnormal{(ii)}] locally Lipschitz continuous in $\phi\in\mathcal{N}$ for some bounded open set $\mathcal{N}\subseteq\R^{r_2}$ uniformly in $(t,x)$ with Lipschitz constant $L_1$:
  \[\sup_{t\in\cT}\sup_{x\in X}\lt|V[\eta,\nu_{\cdot},h](t,x,\phi_1)-V[\eta,\nu_{\cdot},h](t,x,\phi_2)\rt|\le L_1|\phi_1-\phi_2|,\]
  \end{enumerate}where
$$L_1=L_1(\eta,\nu_{\cdot})\colon=\|\nu_{\cdot}\|\sum_{i=1}^r\mathcal{L}_{\mathcal{N}}(g_i)\|\eta^i\|
+\mathcal{L}_{\mathcal{N}}(h)<\infty,$$where $\mathcal{L}_{\mathcal{N}}(g_i)$ ($\mathcal{L}_{\mathcal{N}}(h)$, respectively) is the Lipschitz constant of $g_i$ ($h$, respectively) restricted to $\mathcal{N}$.
  Additionally assume $\mathbf{(A6)}$ with the convex compact set $Y$, then $V[\eta,\nu_{\cdot},h](t,x,\phi)$ is
  \begin{enumerate}
  \item[\textnormal{(iii)}] Lipschitz continuous in $h$ with Lipschitz constant $1$:
  \[\sup_{t\in\cT}\sup_{x\in X}\sup_{\phi\in Y}|V[\eta,\nu_{\cdot},h_1](t,x,\phi)-V[\eta,\nu_{\cdot},h_2](t,x,\phi)|\le\|h_1-h_2\|_{\infty},\]
  where $$\|h_1-h_2\|_{\infty}=\sup_{t\in\cT}\sup_{x\in X}\sup_{\phi\in Y}|h_1(t,x,\phi)-h_2(t,x,\phi)|.$$
   \item[\textnormal{(iv)}] Lipschitz continuous in $\nu_{\cdot}$ with Lipschitz constant $L_2$:
   \begin{equation*}
  \sup_{x\in X}\sup_{\phi\in Y}|V[\nu^1_{\cdot}](t,x,\phi)-V[\nu^2_{\cdot}](t,x,\phi)|
  \le L_2d_{\infty}(\nu^1_t,\nu^2_t),
\end{equation*}
where $L_2=L_2(\eta)\colon=r_2\sum_{i=1}^r(\BL(g_i)+1)\|\eta^i\|$.
   \item[\textnormal{(v)}] continuous in $\eta$: Let $\bigl\{\eta^K\bigr\}_{K\in\N}=\bigl\{(\eta^{K,i})_{1\le i\le r}\bigr\}_{K\in\N}\subseteq \mathcal{B}(X,\cM_+(Y))$, for $i=1,\ldots,r$. If $\lim_{K\to\infty}\sum_{i=1}^rd_{\infty}(\eta^i,\eta^{K,i})=0$, then
for every $\xi_{\cdot}\in \mathcal{B}(\cT,\cM_+(Y))$,
\[\lim_{K\to\infty}\int_0^t\int_Y\sup_{x\in X}|V[\eta,\nu_{\cdot},h](\tau,x,\phi)-V[\eta^K,\nu_{\cdot},h](\tau,x,\phi)|\rd\xi_{\tau}(\phi)\rd\tau=0,
\quad \forall t\in\cT,\] provided $\nu_{\cdot}\in \mathcal{C}(\cT,\mathcal{C}_{\mu_X,1}(X,\cM_+(Y)))$. holds.
   \end{enumerate}
   \end{proposition}
The proof of Proposition~\ref{Vlasovf} is provided in Appendix~\ref{appendix-Vlasovf}.

\begin{remark}
 The technical property in Proposition~\ref{Vlasovf}(v) will be used in the proof of Proposition~\ref{prop-continuousdependence}(iv) below.
\end{remark}
\begin{theorem}\label{theo-well-posedness-characteristic}
Assume ($\mathbf{A1}$)-($\mathbf{A5}$). Let $\phi_0\in\mathcal{B}(X;\R^{r_2})$. Then for every $x\in X$ and $t_0\in\cT$, there exists a solution $\phi(t,x)$ to the IVP of \eqref{Charac} with $\phi(t_0,x)=\phi_0(x)$ for all $t\in(\tau_{x,t_0}^-,\tau_{x,t_0}^+)\subseteq\cT$ with $(\tau_{x,t_0}^-,\tau_{x,t_0}^+)\ni t_0$ such that
\begin{enumerate}
\item[\textnormal{(i)}] either \textnormal{(i-a)} $\tau_{x,t_0}^+=T$ or \textnormal{(i-b)} $\tau_{x,t_0}^+<T$ and $\lim_{t\uparrow \tau_{x,t_0}^+}|\phi(t,x)|=\infty$ holds;
\item[\textnormal{(ii)}] either \textnormal{(ii-a)} $\tau_{x,t_0}^-=0$ or \textnormal{(ii-b)} $\tau_{x,t_0}^->0$ and $\lim_{t\downarrow \tau_{x,t_0}^-}|\phi(t,x)|=\infty$ holds.
\end{enumerate}
In addition, assume ($\mathbf{A6}$) and $\nu_{\cdot}$ is uniformly supported within $Y$, then $\tau_{x,t_0}^+=T$ for all $x\in X$, and there exists a family of transformations $\left\{\mathcal{S}^x_{t,s}[\eta,\nu_{\cdot},h]\right\}_{t,s\in\cT}$ on $Y$ such that
$$\phi(t,x)=\mathcal{S}^x_{t,s}[\eta,\nu_{\cdot},h]\phi(s,x),\quad \text{for all}\ s,t\in\cT.$$
\end{theorem}
The proof of Theorem~\ref{theo-well-posedness-characteristic} is provided in Section~\ref{sect-proof}.

For every $\alpha>0$, let $d_{\alpha}$ be the metric in Definition~\ref{def-alphametric}. Define the following operator: For $\nu_{\cdot}\in(\mathcal{C}(\cT,\mathcal{B}_{\mu_X,1}(X,\cM_+(Y))),d_{\alpha})$, $$ \nu_{\cdot}\mapsto\mathcal{A}[\eta,h](\nu_{\cdot}),$$
 via$$((\mathcal{A}[\eta,h](\nu_{\cdot}))_t)_x=(\nu_0)_x\circ \mathcal{S}_{t,0}^x[\eta,\nu_{\cdot},h],\quad x\in X.$$
We will show there exists $\alpha>0$ such that $\cA[\eta,h]$ is a contraction mapping and hence by Banach fixed point theorem, the fixed point equation
\begin{equation}\label{Fixed}
\nu_{\cdot}=\cA[\eta,h]\nu_{\cdot},\quad t\in\cT.
\end{equation}
admits a unique solution.
Beforehand, let us investigate properties of $\cA$.
\begin{proposition}\label{prop-continuousdependence}
Assume $\mathbf{(A1)}$-$\mathbf{(A6)}$. Let $L_1$ and $L_2$ be given in Proposition~\ref{Vlasovf}. Then $\cA[\eta,h]$ is
\begin{enumerate}
\item[\textnormal{(i)}] is continuous in $t$: For every $\nu_{\cdot}\in \mathcal{C}(\cT,\mathcal{B}_{\mu_X,1}(X,\cM_+(Y)))$, we have $t\mapsto(\cA[\eta,h]\nu_{\cdot})_{t}\in \mathcal{C}(\cT,\mathcal{B}_{\mu_X,1}(X,\cM_+(Y)))$. In particular,
    if $\nu_{\cdot}\in\mathcal{C}(\cT,$ $\mathcal{C}_{\mu_X,1}(X,\cM_+(Y)))$, then $\cA[\eta,h]\nu_{\cdot}$ $\in \mathcal{C}(\cT,\mathcal{C}_{\mu_X,1}(X,$ $\cM_+(Y)))$.
    Moreover, the mass conservation law holds: $$((\cA[\eta,\nu_{\cdot},h]\nu_{\cdot})_{t})_x(Y)=(\nu_0)_x(Y),\quad \forall x\in X.$$
    \item[\textnormal{(ii)}] Lipschitz continuous in $\nu_{\cdot}$: For all $t\in\cT$, and $\nu^1_{\cdot},\nu^2_{\cdot}\in \mathcal{C}(\cT,\mathcal{B}_{\mu_X,1}(X,\cM_+(Y)))$,
  \begin{equation*}
  \begin{split}
  &d_{\infty}((\cA[\eta,h]\nu^1_{\cdot})_t,(\cA[\eta,h]\nu^2_{\cdot})_t)\\
  \le& \textnormal{e}^{L_1(\nu^2_{\cdot})t}d_{\infty}(\nu_0^1,\nu_0^2)+ L_2\|\nu_{\cdot}^1\|\textnormal{e}^{L_1(\nu^2_{\cdot})t}\int_0^td_{\infty}(\nu^1_{\tau},\nu^2_{\tau})
  \textnormal{e}^{-L_1(\nu^2_{\cdot})\tau}\rd\tau.
  \end{split}
  \end{equation*}
  \item[\textnormal{(iii)}] Lipschitz continuous in $h$: For all $t\in\cT$, and $h_1,h_2$ both fulfilling $\mathbf{(A3)}$ with $h$ replaced by $h_1,h_2$, respectively,\begin{equation*}
        d_{\infty}(\mathcal{A}[\eta,h_1](\nu_{t}),\mathcal{A}[\eta,h_2](\nu_{t}))\le T\|\nu_{\cdot}\|\textnormal{e}^{L_3t}\|h_1-h_2\|_{\infty},
\end{equation*}
where $L_3=L_3(\nu_{\cdot},h_2)\colon=L_1(\nu_{\cdot})+\BL(h_2)\textnormal{e}^{L_1(\nu_{\cdot})T}$.
  \item[\textnormal{(iv)}] Absolute continuity. If $\nu_0\in \mathcal{B}_{\mu_X,1}(X,\cM_{+,\abs}(Y))$, then $$(\cA[\eta,\nu_{\cdot},h]\nu_{\cdot})_{t}\in \mathcal{B}_{\mu_X,1}(X,\cM_{+,\abs}(Y)),\quad \forall t\in\cT.$$
\end{enumerate}
\end{proposition}
The proof of Proposition~\ref{prop-continuousdependence} is provided in Appendix~\ref{appendix-prop-continuousdependence}.

\begin{proposition}\label{prop-sol-fixedpoint}
Assume $\mathbf{(A1)}$-$\mathbf{(A4)}$ and $\mathbf{(A6)}$-$\mathbf{(A7)}$. Let $\nu_0\in \mathcal{B}_{\mu_X,1}(X,\cM_+(Y))$, and $L_1$, $L_2$, and $L_3$ be given in Proposition~\ref{prop-continuousdependence}. Then there exists a unique solution $\nu_{\cdot}\in\mathcal{C}(\cT,\mathcal{B}_{\mu_X,1}(X,\cM_+(Y)))$ to the fixed point equation
\eqref{Fixed}. In particular, if $\nu_0\in \mathcal{C}_{\mu_X,1}(X,\cM_+(Y))$, then $\nu_{\cdot}\in\mathcal{C}(\cT,\mathcal{C}_{\mu_X,1}(X,$ $\cM_+(Y)))$. Moreover, the solutions have continuous dependence on
\begin{enumerate}
\item[\textnormal{(i)}] the initial conditions:
\[d_{\infty}(\nu_t^1,\nu_t^2)\le \textnormal{e}^{(L_1(\nu^2_{\cdot})+L_2\|\nu_{\cdot}^1\|)t}d_{\infty}(\nu_0^1,\nu_0^2),\quad t\in\cT,\]
where $\nu^i_{\cdot}$ is the solution to \eqref{Fixed} with initial condition $\nu^i_0$ for $i=1,2$.
\item[\textnormal{(ii)}] $h$: $$d_{\infty}(\nu_t^1,\nu_t^2)\le \frac{1}{L_3(\nu^2_{\cdot})}\|\nu^1_{\cdot}\|
    \textnormal{e}^{\BL(h_2)\textnormal{e}^{L_1(\nu^2_{\cdot})T}T}\textnormal{e}^{(L_1(\nu^2_{\cdot})+L_2\|\nu^1_{\cdot}\|)t}\|h_1-h_2\|_{\infty}.$$
where $\nu^i_{\cdot}$ is the solution to \eqref{Fixed} with functions $h_i$ for $i=1,2$.
\item[\textnormal{(iii)}] $\eta$: Let $\{\eta^K\}_{K\in\N}=\{(\eta^{K,i})_{1\le i\le r}\}_{K\in\N}\subseteq \mathcal{B}(X,\cM_+(Y))$, for $i=1,\ldots,r$, such that $\lim_{K\to\infty}\sum_{i=1}^rd_{\infty}(\eta^i,\eta^{K,i})=0$. Assume $\nu_0\in \mathcal{C}_{\mu_X,1}(X,\cM_+(Y))$. Then
    $$\lim_{K\to\infty}\sup_{t\in\cT}d_{\infty}(\nu_t,\nu^K_t)=0,$$
where $\nu^K_{\cdot}$ is the solution to \eqref{Fixed} with DGMs $\eta^K$ for $K\in\N$.
\end{enumerate}
\end{proposition}
The proof of Proposition~\ref{prop-sol-fixedpoint} is provided in Appendix~\ref{appendix-prop-sol-fixedpoint}.

\subsection{VE}

In this subsection, we will use properties in the previous subsection to show well-posedness of VE \eqref{Vlasov}.

First, let us define the weak solution to  \eqref{Vlasov}.

\begin{definition}\label{def-weak-sol}
Let $Y$ be a compact positively invariant subset of \eqref{Charac} given in Theorem~\ref{theo-well-posedness-characteristic}. We say $\rho\colon\cT\times X\times Y\to\R_{\ge0} $ is a \emph{uniformly weak solution} to the IVP \eqref{Vlasov} if for every $x\in X$, the following two conditions are satisfied:
\medskip

\noindent(i) Normalization. $\int_X\int_Y\rho(t,x,\phi)\rd\phi\rd x=1$, for all $t\in\cT$.
\medskip

  \noindent(ii) Uniform weak continuity. $t\mapsto\int_Yf(\phi)\rho(t,x,\phi)\rd\phi$ is continuous uniformly in $x\in X$, for every $f\in \mathcal{C}(Y)$.
\medskip

  \noindent(iii) Integral identity: For all test functions $w\in \mathcal{C}^1(\cT\times Y)$ with $\supp w\subseteq [0,T[\times U$ and $U\subset\subset Y$, the equation below holds: \begin{multline}\label{eq-test}
    \int_0^T\int_Y\rho(t,x,\phi)\lt(\frac{\partial w(t,\phi)}{\partial t}+\widehat{V}[\eta,\rho(\cdot),h](t,x,\phi)\cdot\nabla_{\phi} w(t,\phi)\rt)\rd\phi\rd t\\
    +\int_Yw(0,\phi)\rho_0(x,\phi)\rd\phi=0,
  \end{multline}
  where $\supp w=\overline{\{(t,u)\in \cT\times Y\colon w(t,u)\neq0\}}$ is the support of $w$, and we recall
\begin{equation*}
\widehat{V}[\eta,\rho(\cdot),h](t,x,\phi)=\sum_{i=1}^r\int_X\int_{Y}g_i(t,\psi,\phi)\rho(t,y,\phi)
\rd\psi\rd\eta^i_x(y)+h(t,x,\phi).
\end{equation*}
\end{definition}

\begin{remark}
Definition~\ref{def-weak-sol} is well-posed, since from \eqref{eq-test}, by choosing suitable test functions one can show that $$\rho(t,\cdot)\in {\sf L}^1_+(X\times Y,\mu_X\otimes\mathfrak{m}),\quad \text{for all}\ t\in\cT,$$ provided $\rho_0\in {\sf L}^1_+(X\times Y,\mu_X\otimes\mathfrak{m})$ and $\rho(t,\cdot)$ solves \eqref{eq-test}. Hence this definition of a uniformly weak solution can be
slightly \emph{stronger} than the weak solution defined in \cite{N84,KM18}, since \eqref{eq-test} is required to hold for every $x\in X$ but not just $\mu_X$-a.e. $x\in X$. This is because for every $t\in\cT$, we regard $\rho(t,x,\phi)$ as densities of every point $(\nu_{t})_x$ on the continuous curve $\{(\nu_{t})_x\}_{x\in X}$, instead of the density of a probability in the space $\cP(X\times Y)$.\end{remark}
Now we present the unique existence of solutions to the VE \eqref{Vlasov}.
\begin{theorem}\label{th-equi}
Assume ($\mathbf{A1}$)-($\mathbf{A4}$) and ($\mathbf{A6}$). Assume $\rho_0(x,\phi)$ is continuous in $x\in X$ for $\mathfrak{m}$-a.e. $\phi\in Y$ such that $\rho_0\in {\sf L}^1_+(X\times Y,\mu_X\otimes\mathfrak{m})$, then there exists exists a unique uniformly weak solution $\rho(t,x,\phi)$ to the IVP of \eqref{Vlasov} with initial condition $\rho(0,x,\phi)=\rho_0(x,\phi)$, $x\in X$, $\phi\in Y$. \end{theorem}

The proof of Theorem~\ref{th-equi} is provided in Section~\ref{sect-proof}.

\section{Approximation of time-dependent solutions of VE}\label{sect-approximation-ode}

In this section, we seek approximations of solutions of the VE in bounded Lipschitz distance by probability measures with finitely supported piecewise defined measures generated from solutions to discretized ODEs. More specifically, we construct approximation of the absolutely continuous solutions of \eqref{Fixed} by finitely supported probabilities on $X\times Y$ on atoms $(x_i^n,\varphi_i^n(t))$ with $(\varphi_1^n(t),\ldots,\varphi_n^n(t))$ being the solution of the ODEs whose initial data are distributed asymptotically converging to the initial distribution of \eqref{Fixed}. In terms of measures, we seek for approximations by $\cD(Y)$-valued step functions on $X$.

To state the approximation result, let us first recall several recent results on approximation of probability measures by deterministic empirical measures.
\begin{proposition}\label{prop-XBC}
  \cite{C18,XB19}
Let $q\in\N$  and $\cP(\R^q)$ be the space of all Borel probability measures on $\R^q$. Let $\eta\in\cP(\R^q)$. Assume $\eta$ is compactly supported. Then there exists a uniformly bounded sequence $\{z^n_j\}_{j=1,\ldots,n;n\in\N}\subseteq \R^q$ such that
$$\lim_{n\to\infty}d_{\sf BL}(\eta^n,\eta)=0,$$
where $\eta^n=\frac{1}{n}\sum_{\ell=1}^n\delta_{z^n_j}$ is a deterministic empirical approximation of $\eta$.
\end{proposition}

We first provide two one-dimensional examples to illustrate how these empirical approximations are constructed, given a probability measure.

\begin{example}
  Let $X=[0,1]$. Let $\eta\in\mathcal{C}(X;\cM_+(X))$ be defined in Example~\ref{ex-mix}(ii). Hence $\frac{1}{\eta_x(X)}\eta_x\in\cP(X)$, for every $x\in X$.
  Note that $$\eta_x(X)=\begin{cases}
    1-3x,\quad \textnormal{if}\quad x\in[0,1/3[,\\
    1,\quad\hspace{2.2em} \textnormal{if}\quad x\in[1/3,2/3],\\
    3x-2,\quad \textnormal{if}\quad x\in]2/3,1].
  \end{cases}$$
  For every $m\in\N$, let
  $$x_i^m=\frac{i}{m+1},\quad \textnormal{for}\quad i=1,\ldots,m,$$  \begin{align*}
      \eta^{m,n}_{x^m_i}=&\begin{cases}
        \frac{1}{n}\sum_{j=1}^n\delta_{\frac{3}{2}x^m_i+\frac{j}{n+1}(1-\frac{3}{2}x^m_i)},\hspace{4.1em} \text{for}\quad i=1,\ldots,\lfloor\frac{m+1}{3}\rfloor-1,\\
         \delta_{1/2},\hspace{13.3em} \text{for}\quad i=\lfloor\frac{m+1}{3}\rfloor+1,\ldots,\lceil\frac{2(m+1)}{3}\rceil-1,\\
         \frac{1}{n}\sum_{j=1}^n\delta_{\frac{3}{2}(1-x^m_i)+\frac{j}{n+1}(1-\frac{3}{2}(1-x^m_i))},\quad \text{for}\quad i=\lceil\frac{2(m+1)}{3}\rceil,\ldots,m.
      \end{cases}\end{align*}
      It is also straightforward to show that \[\lim_{n\to\infty}d_{\sf BL}\left(\tfrac{1}{\eta_{x_i^m}(X)}\eta_{x_i^m},\eta_{x_i^m}^{m,n}\right)=0,\quad \text{for}\ i=1,\ldots,m.\] Indeed, $\eta_{x_i^m}^{m,n}$ are the best uniform $n$-approximation of $\tfrac{1}{\eta_{x_i^m}(X)}\eta_{x_i^m}$ with at most $n$ atoms \cite{XB19}.
\end{example}
\begin{example}
Let $X=\mathbb{S}^1$ and $\eta$ be defined in Example~\ref{ex-mix}(iii).
  Moreover, for every $n\in\N$, let $\zeta^n\colon=\frac{1}{n}\sum_{i=1}^n\delta_{y_i^n}$ be the best uniform approximation of the Cantor measure $\zeta_{\sf C}$ \cite{XB19}. Since $\supp\zeta^n\subset]0,1[$ and the distance on the circle is no greater than that on the real line, for every $x\in X$, one can show that $$\eta_{x}^{n}\colon=\frac{1}{n}\sum_{i=1}^n\delta_{\langle x+\frac{3}{4}y_i^n\rangle},$$ is a uniform approximation of $\eta_x$ such that $\lim_{n\to\infty}d_{\sf BL}(\eta_x,\eta_{x}^{n})=0$.
\end{example}

\begin{lemma}
  [Partition of $X$]\label{le-partition}
Assume  $(\mathbf{A1})$. Then there exists a sequence of pairwise disjoint partitions $\{A^m_i\colon i=1,\ldots,m\}_{m\in\N}$ of $X$ such that
  $X=\cup_{i=1}^m A_i^m$ for every $m\in\N$ and
  \[\lim_{m\to\infty}\max_{1\le i\le m}\Diam A^m_i=0.\]
\end{lemma}

The proof of Lemma~\ref{le-partition} is provided in Appendix~\ref{appendix-le-partition}.

\begin{lemma}[Approximation of the initial distribution]\label{le-ini-2}
Assume $(\mathbf{A1})$ and $(\mathbf{A7})$. Let $\{A^m_i\}_{1\le i\le m}$ be a partition of $X$ for $m\in \N$ satisfying \[\lim_{m\to\infty}\max_{1\le i\le m}\Diam A^m_i=0.\] Let $x_i^m\in A_i^m$, for $i=1,\ldots,m$, $m\in\N$.
Then there exists a sequence $\{\varphi^{m,n}_{(i-1)n+j}\colon i=1,\ldots,m,j=1,\ldots,n\}_{n,m\in\N}\subseteq Y$ such that $$\lim_{m\to\infty}\lim_{n\to\infty}d_{\infty}(\nu_0^{m,n},\nu_0)=0,$$
where $\nu_0^{m,n}\in \mathcal{B}_{\mu_X,1}(X,\cM_+(Y))$ with $$(\nu_0^{m,n})_x\colon=\sum_{i=1}^m\mathbbm{1}_{A^m_i}(x)\frac{a_{m,i}}{n}\sum_{j=1}^n
\delta_{\varphi^{m,n}_{(i-1)n+j}},\quad x\in X,$$
$$a_{m,i}=\begin{cases}
  \frac{\int_{A_i^m}(\nu_0)_x(Y)\rd\mu_X(x)}{\mu_X(A_i^m)},\quad \textnormal{if}\quad \mu_X(A_i^m)>0,\\
  (\nu_0)_{x^m_i}(Y),\hspace{4.2em} \textnormal{if}\quad \mu_X(A_i^m)=0.
\end{cases}$$
\end{lemma}

The proof of Lemma~\ref{le-ini-2} is provided in Appendix~\ref{appendix-le-ini-2}.

\begin{lemma}[Approximation of the DGM]\label{le-graph}
Assume $(\mathbf{A1})$ and $(\mathbf{A4})'$. For every $m\in\N$, let $A^m_i$ and $x_i^m$ be defined in Lemma~\ref{le-ini-2} for $i=1,\ldots,m$, $m\in\N$. Then for every $\ell=1,\ldots,r$, there exists a sequence $\{y^{\ell,m,n}_{(i-1)n+j}\colon i=1,\ldots,m,j=1,\ldots,n\}_{m,n\in\N}\subseteq Y$ such that $$\lim_{m\to\infty}\lim_{n\to\infty}d_{\infty}(\eta^{\ell,m,n},\eta^{\ell})=0,$$
where $\eta^{\ell,m,n}\in \mathcal{B}(X,\cM_+(Y))$ with $$\eta^{\ell,m,n}_x\colon=\sum_{i=1}^m\mathbbm{1}_{A^m_i}(x)\frac{b_{\ell,m,i}}{n}
\sum_{j=1}^n\delta_{y^{\ell,m,n}_{(i-1)n+j}},\quad x\in X,$$ $$b_{\ell,m,i}=\begin{cases}
  \frac{\int_{A_i^m}(\eta^{\ell})_x(X)\rd\mu_X(x)}{\mu_X(A_i^m)},\quad \textnormal{if}\quad \mu_X(A_i^m)>0,\\
  \eta^{\ell}_{x_i^m}(X),\hspace{4.2em}\quad \textnormal{if}\quad \mu_X(A_i^m)=0.
\end{cases}$$
\end{lemma}
The proof of Lemma~\ref{le-graph} is provided in Appendix~\ref{appendix-le-graph}.

In what follows, we provide several concrete examples of the deterministic empirical approximation of DGMs with different vertex spaces together with their partitions. All these examples can be used in the applications in Section~\ref{sect-applications}.

The example below gives a high dimensional $X$ with a reference measure $\mu_X$ supported on a lower dimensional curve.
\begin{example}\label{ex-triangle}
  Let $X=\{(z_1,z_2)\in\R^2_+\colon |z|\le1\}$ be the triangle, and $\mu_X=\lambda|_{\{(x_1,x_2)\in X\colon x_1=x_2\}}$ be the uniform measure over the line segment from $(0,0)$ to $(1/2,1/2)$. Hence $\mu_X$ is singular to the Lebesgue measure $\mathfrak{m}$ on $X$. Let $\eta_x=|x|\lambda|_{E_x}$, where $E_x=\{(z_1,z_2)\in X\colon z_1x_2=z_2x_1\}$. Hence $E_x$ is a line through $x$ provided $x\neq0$, and $E_0=X$. It is easy to verify that $x\mapsto\eta_x$ is continuous. Since $\supp\eta_x=\infty$ for all $x\in X\setminus\{0\}$ and $\eta_0=0$, and $\eta_x\perp\mu_X$ for all $x\in X\setminus E_{(1/2,1/2)}$ and $\eta_x\ll \mu_X$ for all $x\in E_{(1/2,1/2)}\setminus\{0\}$, we have that $\eta$ is neither a graphing nor a graphon. For $x\neq0$, the uniform measure $$\eta^n_x=\frac{1}{n}\sum_{j=1}^n\delta_{\tfrac{j}{n+1}\tfrac{x}{|x|}}$$ over the equipartition of the line segment $E_x$ is the best uniform approximation of $\frac{1}{|x|}\eta_x$. For every $m\in\N$, we can take the uniform partition such that $A^m$ is no coarser than $A^{\lfloor\sqrt{m}\rfloor^2}$ and no finer than $A^{\lceil\sqrt{m}\rceil^2}$ with $\{A^{m^2}_{i}\}_{1\le i\le m^2}=\{\Delta((p/m,q/m),((p+1)/k,q/k),(p/k,(q+1)/k))\}_{0\le p,q\le m-1}$ of $X$, consisting of $m^2$ congruent triangles, where $\Delta(a,b,c)$ stands for the triangle with vertices $a,b,c$. See Figure~\ref{fig-ex-triangle}. Hence we can take $x^m_i$ to be any point in $A^m_{i}\setminus\{0\}$ and $$y^{m,n}_{(i-1)n+j}=\tfrac{j}{n+1}\tfrac{x^m_i}{|x^m_i|},\quad j=1,\ldots,n;\quad
b_{m,i}=\begin{cases}
  \tfrac{\int_{A_i^m}|x|\rd\mu_X(x)}{\mu_X(A_i^m)},\quad \text{if}\quad \mu_X(A_i^m)>0,\\
  |x^m_i|,\hspace{3.7em}\quad \text{if}\quad \mu_X(A_i^m)=0.
\end{cases}$$ In other words, $b_{m,i}$ is the $\ell_1$-norm of the \emph{barycenter} of $A_i^m$ w.r.t. $\mu_X$, provided $A_i^m$ has a non-empty intersection with the line segment $E_{(1/2,1/2)}$, and $b_{m,i}$ is the $\ell_1$-norm of $x^m_i$ otherwise. Moreover, there are at most $(1+2\lfloor2^{-1}\lceil\sqrt{m}\,\rceil\rfloor)$ triangles $A^m_i$ with a positive $\mu_X$-measure (a non-empty intersection with $E_{(1/2,1/2)}$).  \end{example}

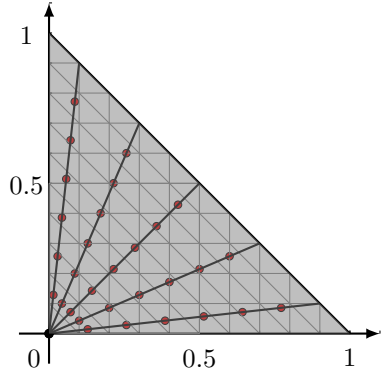
\begin{figure}
\begin{center}
\begin{tikzpicture}
   [%%%%%%%%%%%%%%%%%%%%%%%%%%%%%%
        dot/.style={circle,draw=black, fill,inner sep=1pt},scale=4
    ]%%%%%%%%%%%%%%%%%%%%%%%%%%%%%%
%label coordinates
\foreach \x in {.1,.2,...,1.1}
    \draw (\x,.01) -- node[below,yshift=-1mm] {} (\x,0);
    \draw (1,.01) -- node[below,yshift=-1mm] {1} (1,0);
\node[below,xshift=-2mm,yshift=-1mm] at (0,0) {0};
\foreach \y in {.1,.2,...,1.1}
    \draw (.01,\y) -- node[below,xshift=-3mm,yshift=2mm] {} (0,\y);
     \node[below,xshift=-3mm,yshift=2mm] at (0,1){1};
     \node[below,yshift=-1mm] at (.5,0){0.5};
     \node[below,xshift=-3mm,yshift=2mm] at (0,.5){0.5};
    \foreach \y in {0,.1,...,.7}
     {\node[red,dot] at (\y,\y*3/7){};
     \node[red,dot] at (\y*3/7,\y){};
     \node[red,dot] at (\y*5/7,\y*5/7){};
     \node[red,dot] at (\y*9/7,\y/7){};
     \node[red,dot] at (\y/7,\y*9/7){};
     }
%triangles
     \foreach \y in {1,...,10}
    {\draw[thin,gray] (0,\y/10) -- (1-\y/10,\y/10);
    \draw[thin,gray] (\y/10,0) -- (0,\y/10);
   \draw[thin,gray] (\y/10,0) -- (\y/10,1-\y/10);
}

\draw[->,thick,-latex] (0,-.1) -- (0,1.1);
\draw[->,thick,-latex] (-.1,0) -- (1.1,0);
\draw[step=.5cm,gray,very thin] (0,0) grid (.5,.5);
%supports of eta
\draw[thick,black] (0,0) -- (.5,.5);
\draw[thick,black] (0,0) -- (.3,.7);
\draw[thick,black] (0,0) -- (.7,.3);
\draw[thick,black] (0,0) -- (.1,.9);
\draw[thick,black] (0,0) -- (.9,.1);
\draw[thick,black] (0,1) -- (1,0);
   \node[dot,thick] at (0,0){};
  \fill[gray, opacity=.5] (0,0) -- (1,0) -- (0,1) -- cycle;
  \end{tikzpicture}
  \caption{Example~\ref{ex-triangle}. $m=10$ and $n=7$. Red dots: atoms of uniform $n$-approximations of $\eta_x$ for different $x$.}\label{fig-ex-triangle}
\end{center}
\end{figure}

Indeed, one can also take $X$ to be a discrete subset of the Euclidean space and the DGM is neither sparse not dense.
\begin{example}\label{ex-discrete}
  Let $X=\{1/n\}_{n\in\N}\cup\{0\}\subseteq\R$. Let $\mu_X=\sum_{i=2}^{\infty}2^{-i+1}\delta_{1/i}$, $\eta\in\mathcal{B}(X,\cM_+(X))$ be such that $\eta_x=x\lambda|_{\{y\in X\colon y\ge x\}}$ for $x\in X$, where for convention $0\lambda_X=0$ denotes the trivial zero measure. Since $\sup_{x\in X}\#\supp\eta=\infty$ and $\supp\mu_X\subsetneq\sup\eta_1$, we have $\eta\in\cC(X,\cM_+(X))$ is neither a digraphing nor a digraphon. Hence for every $m\in\N$, let $A^m_i=\{1/i\}$ for $1\le i<m$ and $A^m_m=\{y\in X\colon y\le1/m\}$. Then it is easily seen that $\max_{1\le i\le m}\Diam A^m=1/m\to0$ as $m\to\infty$. Let $x^m_i=1/i$ for $i=1,\ldots,m$. We have $\eta_{x^m_i}=\tfrac{x^m_i}{i}\sum_{k=1}^i\delta_{1/k}$, for $i=1,\ldots,m$. Therefore, for all $n\ge m$,
  \[\eta_x^{m,n}=\begin{cases}\delta_{1},\quad\hspace{12.7em} \text{if}\ x\in A^m_1,\\
  \frac{1}{i}\sum_{j=1}^i\delta_{1/j},\quad\hspace{8.7em} \text{if}\ x\in A^m_i,\quad 1<i<m,\\ \left(\sum_{k=0}^{\infty}2^{-k-1}\tfrac{1}{m+k}\right)\frac{1}{m}\sum_{j=1}^m\delta_{1/j},\quad \text{if}\ x\in A^m_m.\end{cases}\]
\end{example}

\begin{lemma}[Approximation of $h$]\label{le-h}
Assume $(\mathbf{A3})$ and $(\mathbf{A7})$.

For every $m\in\N$, let $$h^m(t,z,\phi)=\sum_{i=1}^m\mathbbm{1}_{A^m_i}(z)h(t,x_i^m,\phi),\quad t\in\cT,\ z\in X,\ \phi\in Y.$$ Then
  $$\lim_{m\to\infty}\int_0^T\int_Y\sup_{x\in X}\lt|h^m(t,x,\phi)-h(t,x,\phi)\rt|\rd\phi\rd t=0.$$
\end{lemma}
The proof of Lemma~\ref{le-h} is provided in Appendix~\ref{appendix-le-h}.

Now we are ready to provide a discretization of the the VE on the DGM by a sequence of ODEs.
 From Lemmas~\ref{le-partition}-\ref{le-graph}, there exists
 \begin{enumerate}
 \item[$\bullet$] a partition $\{A^m_i\}_{1\le i\le m}$ of $X$ and points $x^m_i\in A^m_i$ for $i=1,\ldots,m$, for every $m\in\N$,
 \item[$\bullet$]  a sequence $\{\varphi^{m,n}_{(i-1)n+j}\colon i=1,\ldots,m,j=1,\ldots,n\}_{n,m\in\N}\subseteq Y$ and $\{a_{m,i}\colon i=1,\ldots,$ $m\}_{m\in\N}\subseteq\R_+$,
  \item[$\bullet$] a sequence $\{y^{\ell,m,n}_{(i-1)n+j}\colon i=1,\ldots,m,j=1,\ldots,n\}_{m,n\in\N}\subseteq Y$ and $\{b_{\ell,m,i}\colon i=1,\ldots,$ $m\}_{m\in\N}\subseteq\R_+$, for $\ell=1,\ldots,r$,
 \end{enumerate}
 such that
  $$\lim_{m\to\infty}\lim_{n\to\infty}d_{\infty}(\nu_0^{m,n},\nu_0)=0,$$
   $$\lim_{m\to\infty}\lim_{n\to\infty}d_{\infty}(\eta^{\ell,m,n},\eta)=0,\quad \ell=1,\ldots,r,$$
     $$\lim_{m\to\infty}\int_0^T\int_Y\sup_{x\in X}\lt|h^m(t,x,\phi)-h(t,x,\phi)\rt|\rd\phi\rd t=0,$$
 where  \begin{subequations}
   \label{discretization}
 \begin{alignat}{2}
(\nu_0^{m,n})_x\colon=\sum_{i=1}^m\mathbbm{1}_{A^m_i}(x)\frac{a_{m,i}}{n}\sum_{j=1}^n
\delta_{\varphi^{m,n}_{(i-1)n+j}},\quad x\in X,\\ \eta^{\ell,m,n}_x\colon=\sum_{i=1}^m\mathbbm{1}_{A^m_i}(x)\frac{b_{\ell,m,i}}{n}
\sum_{j=1}^n\delta_{y^{\ell,m,n}_{(i-1)n+j}},\quad x\in X,\\
 h^m(t,z,\phi)=\sum_{i=1}^m\mathbbm{1}_{A^m_i}(z)h(t,x_i^m,\phi),\quad t\in\cT,\ z\in X,\ \phi\in Y.
 \end{alignat}
 \end{subequations}
 Consider the following IVP of a coupled ODE system:
\begin{multline}
  \label{lattice}
  \dot{\phi}_{(i-1)n+j}=F^{m,n}_{i}(t,\phi_{(i-1)n+j},\Phi),\quad 0<t\le T,\quad \phi_{(i-1)n+j}(0)=\varphi_{(i-1)n+j},\\ i=1,\ldots,m,\ j=1,\ldots,n,
\end{multline}
where $\Phi=(\phi_{(i-1)n+j})_{1\le i\le m,1\le j\le n}$ and $$F^{m,n}_{i}(t,\psi,\Phi)=\sum_{\ell=1}^r\sum_{p=1}^{m}\frac{a_{m,i}b_{\ell,m,p}}{n^2}\sum_{j=1}^n
\mathbbm{1}_{A^m_p}(y^{\ell,m,n}_{(i-1)n+j})\sum_{q=1}^{n}g_{\ell}(t,\psi,\phi_{(p-1)n+q})+h^m(t,x^m_i,\psi).$$
\begin{proposition}\label{prop-well-posed-lattice}
Then there exists a unique solution $\phi^{m,n}(t)=(\phi^{m,n}_{(i-1)n+j}(t))$ to \eqref{lattice},
for $m,n\in\N$.
\end{proposition}
The proof of Proposition~\ref{prop-well-posed-lattice} is provided in Appendix~\ref{appendix-prop-well-posed-lattice}.

For $t\in\cT$, let $\phi^{m,n}(t)=(\phi^{m,n}_{(i-1)n+j}(t))$ be the solution to \eqref{lattice}. Define \begin{equation}\label{Eq-approx}
(\nu_t^{m,n})_x\colon=\sum_{i=1}^m\mathbbm{1}_{A^m_i}(x)\frac{a_{m,i}}{n}
\sum_{j=1}^{n}\delta_{\phi^{m,n}_{(i-1)n+j}(t)},\quad x\in X.
\end{equation}

\begin{theorem}\label{th-approx}
Assume ($\mathbf{A1}$)-($\mathbf{A3}$), $\mathbf{(A4)'}$, ($\mathbf{A6}$)-($\mathbf{A7}$). Assume $\rho_0(x,\phi)$ is continuous in $x\in X$ for $\mathfrak{m}$-a.e. $\phi\in Y$ such that $\rho_0\in {\sf L}^1_+(X\times Y,\mu_X\otimes\mathfrak{m})$ and $$\sup_{x\in X}\|\rho_0(x,\cdot)\|_{{\sf L}^1(Y,\mathfrak{m})}<\infty.$$  Let  $\rho(t,x,\phi)$ be the uniformly weak solution to the VE \eqref{Vlasov} with initial condition $\rho_0$. Let $\nu_{\cdot}\in\mathcal{C}(\cT;\mathcal{B}_{\mu_X,1}(X,\cM_{\abs}(Y)))$ be the measure-valued function defined in terms of the uniformly weak solution to \eqref{Vlasov}:
$$\rd(\nu_t)_x=\rho(t,x,\phi)\rd\phi,\quad \text{for every}\quad t\in\cT\quad \text{and}\quad x\in X.$$ Then $\nu_t\in \mathcal{C}_{\mu_X,1}(X,\cM_+(Y))$, for all $t\in\cT$. Moreover, let
$\nu^{m,n}_0\in\mathcal{B}_{\mu_X,1}(X,\cM_+(Y))$, $\eta^{\ell,m,n}\in \mathcal{B}(X,\cM_+(Y))$, and $h^m\in \mathcal{C}(\cT\times X\times Y,\R^{r_2})$ be defined in \eqref{discretization},
and $\nu^{m,n}_{\cdot}$ be defined in \eqref{Eq-approx}. Then $$\lim_{n\to\infty}d_{\infty}(\nu_t^{m,n},\nu_t)=0.$$
\end{theorem}
The proof of Theorem~\ref{th-approx} is provided in Section~\ref{sect-proof}.
\begin{remark}
  The uniform integrability condition $\sup_{x\in X}\|\rho_0(x,\cdot)\|_{{\sf L}^1(Y,\mathfrak{m})}<\infty$ means that the uniform measure $\nu_0$ lies in $\mathcal{B}_{\mu_X,1}(X,\cM_+(Y))$. Although the continuity as well as the uniform integrability of the initial density is technical, it does generalize the results in \cite{KM18}, see Remark~\ref{re-generalization}.
\end{remark}

\begin{remark}
The continuity condition for DGMs $\mathbf{(A4)'}$ can be further relaxed to

  \medskip
\noindent{$\mathbf{(A4)''}$} $\eta=(\eta^1,\ldots,\eta^r)\in (\mathcal{B}(X,\cM_+(X)))^r$ such that $\cup_{i=1}^r\{z\in X\colon x\mapsto\eta^i_x\ \text{is}\ \text{disconti-}$ $\text{nuous at}\ z\}$ is finite.

Then $X=(\cup_{j=1}^KX_j)\cup(\cup_{k=1}^{K'}z_k)$ for some $K'\le K$, where the function $x\mapsto\eta^i_x$ is discontinuous at $x=z_k$, while is continuous confined to each subset $X_j$. Hence one can further take partitions of every $X_j$, and then choose an arbitrary point in each subset of the partition together with these discontinuity points $z_k$ to construct the approximation of DGMs and hence the ODE approximations.
\end{remark}

\section{Applications}\label{sect-applications}
In this section, we apply our main results to several models in biology. To save the Arabic numbers in the labels of assumptions, we point out that the labels of assumptions (together with $X,Y$) differ from subsection to subsection.

\subsection{A multi-group epidemic model without demography}
In this subsection, we apply our main results to an SIS epidemic model with heterogeneous group structure.

Assume

\medskip
\noindent ($\mathbf{H1}$) For $(u_1,u_2)\in\R^2_+$, let $\beta(t,u_1,u_2)\ge0$ be in general the disease transmission function which may not respect mass-action kinetics, and $\beta(t,u_1,u_2)=0$ provided $u_1u_2=0$. Moreover, $\beta$ is continuous in $t$, and locally Lipschitz continuous in $u_1,\ u_2$  uniformly in $t$.

\medskip

\noindent ($\mathbf{H2}$)  For $u\in\R_+$, let $\gamma(t,x,u)\ge0$ be the recovery rate function, and for every $x\in X$, $\gamma(t,x,u)=0$ provided $u=0$. Moreover, $\gamma$ is continuous in $t$, and Lipschitz continuous in $u\in \R_+$ uniformly in $t$, and continuous in $x$ uniformly in $u$.

For any fixed $N\in\N$, let \begin{equation}
  \label{Eq-Y-SIS}Y=\{u\in\R^2_+\colon u_1+u_2=N\}.
\end{equation}

\medskip
\noindent{$\mathbf{(H3)}$} $\nu_{\cdot}\in \mathcal{C}(\cT,\mathcal{B}_{\mu_X,1}(X,\cM_+(\R^2)))$  is uniformly compactly supported within $Y\subseteq\R^{2}_+$.
\medskip

Under ($\mathbf{H1}$)-($\mathbf{H3}$), consider a general non-local multi-group SIS model on a DGM $\eta$:
\begin{equation*}
\begin{split}
  \frac{\partial S_x}{\partial t}=&-\int_X\int_{\R^2_+}\beta(t,\psi_2,S_x)\rd(\nu_t)_y(\psi)\rd\eta_x(y)
  +\gamma(t,x,I_x),\\
  \frac{\partial I_x}{\partial t}=&\int_X\int_{\R^2_+}\beta(t,\psi_2,S_x)\rd(\nu_t)_y(\psi)\rd\eta_x(y)
  -\gamma(t,x,I_x),\end{split}
\end{equation*}
where $S_x$ and $I_x$ stand for the number of susceptible and infected individuals at location $x\in X$ (or interpreted as in the group with label $x$).

By $\mathbf{(H3)}$, let
$$g(t,\psi,\phi)=\beta(t,\psi_2,\phi_1)\begin{pmatrix}
  -1\\
  1
\end{pmatrix},\quad h(t,x,\phi)=\gamma(t,x,\phi_1)\begin{pmatrix}
1\\
-1
\end{pmatrix},$$ $$V[\eta,\nu_{\cdot},h](t,x,\psi)=\int_X\int_Yg(t,\psi,\phi)\rd(\nu_t)_y(\psi)\rd\eta_x(y)+h(t,x,\phi),$$and

\begin{equation*}
\widehat{V}[\eta,\rho_{\cdot},h](t,x,\phi)=\int_X\int_Yg(t,\psi,\phi)\rho(t,y,\psi)\rd\psi\rd\eta_x(y)
+h(t,x,\phi).
\end{equation*}
Consider the VE
\begin{equation}\label{Vlasov-SIS}
\begin{split}
&\frac{\partial\rho(t,x,\phi)}{\partial t}+\textrm{div}_{\phi}\left(\rho(t,x,\phi)\widehat{V}[\eta,\rho(\cdot),h](t,x,\phi)\right)=0,\quad t\in(0,T],\ x\in X,\ \mathfrak{m}\text{-a.e.}\ \phi\in Y,\\
 &\rho(0,\cdot)=\rho_0(\cdot).
\end{split}
  \end{equation}

 From Lemmas~\ref{le-partition}-\ref{le-graph}, there exists
 \begin{enumerate}
 \item[$\bullet$] a partition $\{A^m_i\}_{1\le i\le m}$ of $X$ and points $x^m_i\in A^m_i$ for $i=1,\ldots,m$, for every $m\in\N$,
 \item[$\bullet$]  a sequence $\{\varphi^{m,n}_{(i-1)n+j}\colon i=1,\ldots,m,j=1,\ldots,n\}_{n,m\in\N}\subseteq Y$ and $\{a_{m,i}\colon i=1,\ldots,$ $m\}_{m\in\N}\subseteq\R_+$,
  \item[$\bullet$] a sequence $\{y^{m,n}_{(i-1)n+j}\colon i=1,\ldots,m,j=1,\ldots,n\}_{m,n\in\N}\subseteq Y$ and $\{b_{m,i}\colon i=1,\ldots,$ $m\}_{m\in\N}\subseteq\R_+$,
 \end{enumerate}
 such that
  $$\lim_{m\to\infty}\lim_{n\to\infty}d_{\infty}(\nu_0^{m,n},\nu_0)=0,$$
   $$\lim_{m\to\infty}\lim_{n\to\infty}d_{\infty}(\eta^{m,n},\eta)=0,$$
     $$\lim_{m\to\infty}\int_0^T\int_Y\sup_{x\in X}\lt|h^m(t,x,\phi)-h(t,x,\phi)\rt|\rd\phi\rd t=0,$$
 where
 \begin{subequations}
   \label{discretization-SIS}
 \begin{alignat}{2}
(\nu_0^{m,n})_x\colon=\sum_{i=1}^m\mathbbm{1}_{A^m_i}(x)\frac{a_{m,i}}{n}\sum_{j=1}^n
\delta_{\varphi^{m,n}_{(i-1)n+j}},\quad x\in X,\\ \eta^{m,n}_x\colon=\sum_{i=1}^m\mathbbm{1}_{A^m_i}(x)\frac{b_{m,i}}{n}
\sum_{j=1}^n\delta_{y^{m,n}_{(i-1)n+j}},\quad x\in X,\\
 h^m(t,z,\phi)=\sum_{i=1}^m\mathbbm{1}_{A^m_i}(z)h(t,x_i^m,\phi),\quad t\in\cT,\ z\in X,\ \phi\in Y.
 \end{alignat}
 \end{subequations}
 Consider the following IVP of a coupled ODE system:
\begin{multline}
  \label{lattice-SIS}
  \dot{\phi}_{(i-1)n+j}=F^{m,n}_{i}(t,\phi_{(i-1)n+j},\Phi),\quad 0<t\le T,\quad \phi_{(i-1)n+j}(0)=\varphi_{(i-1)n+j},\\ i=1,\ldots,m,\ j=1,\ldots,n,
\end{multline}
where $\Phi=(\phi_{(i-1)n+j})_{1\le i\le m,1\le j\le n}$ and $$F^{m,n}_{i}(t,\psi,\Phi)=\sum_{p=1}^{m}\frac{a_{m,i}b_{m,p}}{n^2}\sum_{j=1}^n
\mathbbm{1}_{A^m_p}(y^{m,n}_{(i-1)n+j})\sum_{q=1}^{n}g(t,\psi,\phi_{(p-1)n+q})+h^m(t,x^m_i,\psi).$$

Then by Proposition~\ref{prop-well-posed-lattice}, there exists a unique solution $\phi^{m,n}(t)=(\phi^{m,n}_{(i-1)n+j}(t))$ to \eqref{lattice-SIS},
for $m,n\in\N$.

For $t\in\cT$, define \begin{equation}\label{Eq-approx-SIS}
(\nu_t^{m,n})_x\colon=\sum_{i=1}^m\mathbbm{1}_{A^m_i}(x)\frac{a_{m,i}}{n}\sum_{j=1}^{n}\delta_{\phi^{m,n}_{(i-1)n+j}(t)},\quad x\in X.
\end{equation}
\begin{theorem}
  Assume ($\mathbf{A1}$), and ($\mathbf{H1}$)-($\mathbf{H2}$). Then there exists a unique uniformly weak solution $\rho(t,x,\phi)$ to \eqref{Vlasov-SIS}. Assume additionally $\rho_0(x,\phi)$ is continuous in $x\in X$ for $\mathfrak{m}$-a.e. $\phi\in Y$ such that $\rho_0\in {\sf L}^1_+(X\times Y,\mu_X\otimes\mathfrak{m})$ and $$\sup_{x\in X}\|\rho_0(x,\cdot)\|_{{\sf L}^1(Y,\mathfrak{m})}<\infty.$$ Let $\nu_{\cdot}\in\mathcal{C}(\cT;\mathcal{B}_{\mu_X,1}(X,\cM_{\abs}(Y)))$ be the measure-valued function defined in terms of the uniformly weak solution to \eqref{Vlasov-SIS}:
$$\rd(\nu_t)_x=\rho(t,x,\phi)\rd\phi,\quad \text{for every}\quad t\in\cT,\quad \text{and}\quad x\in X.$$ Then $\nu_t\in \mathcal{C}(\cT,\mathcal{C}_{\mu_X,1}(X,\cM_+(Y)))$. Moreover, let
$\nu^{m,n}_0\in\mathcal{B}_{\mu_X,1}(X,\cM_+(Y))$, $\eta^{m,n}\in \mathcal{B}(X,$ $\cM_+(Y))$, and $h^m\in \mathcal{C}(\cT\times X\times Y,\R^2)$ be defined in \eqref{discretization-SIS},
and $\nu^{m,n}_{\cdot}$ be defined in \eqref{Eq-approx-SIS}. Then $$\lim_{n\to\infty}d_{\infty}(\nu_t^{m,n},\nu_t)=0.$$
\end{theorem}
\begin{proof}
It is straightforward to verify that ($\mathbf{H1}$) implies ($\mathbf{A2}$), and ($\mathbf{H2}$) implies ($\mathbf{A3}$) and ($\mathbf{A7}$).
It remains to show ($\mathbf{A6}$) is fulfilled with $Y$ defined in \eqref{Eq-Y-SIS}. This is a simple consequence of the fact that  this SIS model is conservative:
$$\frac{\partial }{\partial t}(S_x(t)+I_x(t))=0.$$
\end{proof}

\subsection{A multi-group epidemic model with demography}
In this subsection we will apply our main results to an SEIRS epidemic model with demography of heterogeneous group structure. Before proposing the heterogeneous model, let us first revisit the single-group SEIRS model \cite{BSKA20} with the following flow chart:
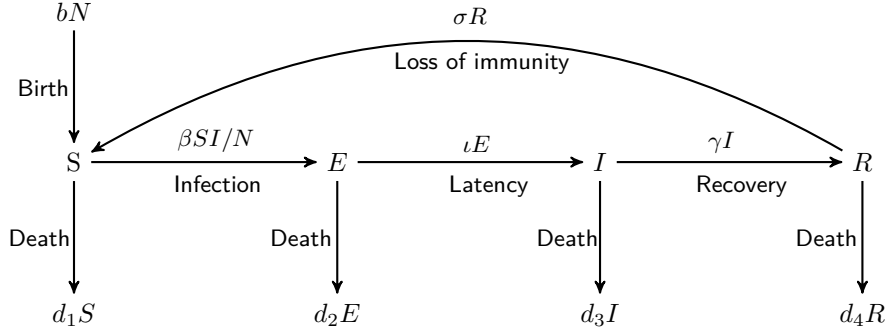
\begin{center}
\begin{tikzpicture}
    [%%%%%%%%%%%%%%%%%%%%%%%%%%%%%%
        ->,>=stealth',auto,node distance=1.5cm,
  thick,main node/.style={circle,draw,font=\sffamily\Large\bfseries}
    ]%%%%%%%%%%%%%%%%%%%%%%%%%%%%%%
  \node[] (A1){S};
  \node[above = of A1] (A0) {$bN$};
  \node[below = of A1] (A2) {$d_1S$};
  \node[above left=.5cm and -.3cm of A1] (A01) {\sffamily\small{Birth}};
  \node[above left=.5cm and -.5cm of A2] (A12) {\sffamily\small{Death}};
  \node[right =3cm of A1] (B1) {$E$};
  \node[below = of B1] (B2) {$d_2E$};
  \node[above left=.5cm and -.5cm of B2] (B12) {\sffamily\small{Death}};
  \node[right = 3cm of B1] (C1) {$I$};
  \node[below = of C1] (C2) {$d_3I$};
  \node[above left=.5cm and -.5cm of C2]  (C12) {\sffamily\small{Death}};
 \node[right = 3cm of C1] (D1) {$R$};
  \node[below = of D1] (D2) {$d_4R$};
 \node[above left=.5cm and -.5cm of D2]  (D12) {\sffamily\small{Death}};

  \draw[->] (A0)--(A1);
  \draw[->] (A1)--(A2);
  \draw[->] (B1)--(B2);
  \draw[->] (C1)--(C2);
  \draw[->] (D1)--(D2);
  \draw[->,every node/.style={font=\sffamily\small}] node[below left=-.2cm and .6cm of B1]{Infection}  (A1)--node[above left=-.0cm and -.8cm of B1]{$\beta SI/N$} (B1);
  \draw[->,every node/.style={font=\sffamily\small}] node[below left=-.2cm and .6cm of C1]{Latency}  (B1)--node[above left=-.0cm and -.4cm of C1]{$\iota E$} (C1);
  \draw[->,every node/.style={font=\sffamily\small}] node[below left=-.2cm and .6cm of D1]{Recovery} (C1)-- node[above left=-.0cm and -.2cm of D1] {$\gamma I$} (D1);
 \path[every node/.style={font=\sffamily\small}]node[above left=.8cm and 3.5cm of D1]{Loss of immunity} (D1) edge[bend right] node [above left=.1cm and -.4cm of D1] {$\sigma R$} (A1);
  \end{tikzpicture}
\end{center}
The deterministic model is given by an ODE with mass-action disease transmission:
\begin{align*}
  \dot{S}=&bN-\beta SI/N+\sigma R-d_1S\\
  \dot{E}=&\beta SI/N-\iota E-d_2E\\
  \dot{I}=&\iota E-\gamma I-d_3I\\
  \dot{R}=&\gamma I-\sigma R-d_4R
\end{align*}
where $b$ is the birth rate, $d_i$ the death rates for different compartments, $N$ the total population size, $\beta$ the transmission rate per capita, $\iota^{-1}$ the latency period, $\gamma$ the recovery rate, and $\sigma$ the rate of losing immunity.

In the following, we will generalize the above model to a multi-group model with heterogeneous group structure. Let $(X,\mathfrak{B}(x),\mu_X)$ be a compact probability space satisfying ($\mathbf{A1}$).
For $i=1,\ldots,4$, define the time-and-location-dependent death rate functions $d_i\colon\cT\times X\to\R_+$, and $$d(x)=\inf_{t\in\cT}\min_{1\le i\le4}d_i(t,x),\quad x\in X.$$ The other constants become time-and-location-dependent as well. Let $\Lambda\colon X\to\R_+$ be the influx of newly born healthy susceptibles. Assume

\medskip
\noindent ($\mathbf{H1}$) $\Lambda\in \mathcal{C}(X)$; for $i=1,2,3,4$, $d_i(t,x)$ is continuous in $x$, and $d(x)>0$ for all $x\in X$.

\medskip
\noindent ($\mathbf{H2}$) $M\colon=\sup_{x\in X}\frac{\Lambda(x)}{d(x)}<\infty$.

\medskip
\noindent ($\mathbf{H3}$) For $(u_1,u_2)\in\R^2_+$, let $\beta(t,u_1,u_2)\ge0$ be in general the disease transmission function, and $\beta(t,u_1,u_2)=0$ provided $u_1u_2=0$. Moreover, $\beta$ is continuous in $t$, and locally Lipschitz continuous in $u_1,\ u_2$  uniformly in $t$.

\medskip
\noindent ($\mathbf{H4}$)  For $u\in\R_+$, let $\iota(t,x,u)\ge0$ be the reciprocal of the latency period function, and for every $x\in X$, $\iota(t,x,u)=0$ provided $u=0$. Moreover, $\iota$ is continuous in $t$, and Lipschitz continuous in $u\in \R_+$  uniformly in $t$, and continuous in $x$ uniformly in $u$.

\medskip

\noindent ($\mathbf{H5}$)  For $u\in\R_+$, let $\gamma(t,x,u)\ge0$ be the recovery rate function, and for every $x\in X$, $\gamma(t,x,u)=0$ provided $\phi_3=0$. Moreover, $\gamma$ is continuous in $t$, and locally Lipschitz continuous in $u\in \R_+$ uniformly in $t$, and continuous in $x$ uniformly in $u$.

\medskip

\noindent ($\mathbf{H6}$)  For $u\in\R_+$, let $\sigma(t,x,u)\ge0$ be the rate function of losing immunity, and for every $x\in X$, $\sigma(t,x,u)=0$ provided $u=0$. Moreover, $\sigma$ is continuous in $t$, and Lipschitz continuous in $u\in \R_+$ uniformly in $t$, and continuous in $x$ uniformly in $u$.

Let \begin{equation}
  \label{Eq-Y-SEIRS}
  Y=\{\phi\in\R^4_+\colon\|\phi\|_1\le M\}\subseteq\R^4_+.
\end{equation}

\medskip
\noindent{$\mathbf{(H7)}$} $\nu_{\cdot}\in \mathcal{C}(\cT,\mathcal{B}_{\mu_X,1}(X,\cM_+(\R^4)))$  is uniformly compactly supported within $Y\subseteq\R^{4}_+$.

Under ($\mathbf{H1}$)-($\mathbf{H7}$), consider a general non-local multi-group SEIRS model on a DGM $\eta$:
\begin{equation*}
\begin{split}
  \frac{\partial S_x}{\partial t}=&\Lambda(x)-\int_X\int_{\R^4_+}\beta(t,\psi_3,S_x)\rd(\nu_t)_y(\psi)\rd\eta_x(y)
  +\sigma(t,x,R_x)-d_1(t,x)S_x\\
  \frac{\partial E_x}{\partial t}=&\int_X\int_{\R^4_+}\beta(t,\psi_3,S_x)\rd(\nu_t)_y(\psi)\rd\eta_x(y)-\iota(t,x,E_x)-d_2(t,x)E_x\\
  \frac{\partial I_x}{\partial t}=&\iota(t,x,E_x)-\gamma(t,x,I_x)-d_3(t,x)I_x\\
  \frac{\partial R_x}{\partial t}=&\gamma(t,x,I_x)-\sigma(t,x,R_x)-d_4(t,x)R_x.
\end{split}
\end{equation*}

By $\mathbf{(H7)}$, let
$$g(t,\psi,\phi)=\begin{pmatrix}
  -\beta(t,\psi_3,\phi_1)\\
  \beta(t,\psi_3,\phi_1)\\
  0\\
  0
\end{pmatrix},\quad h(t,x,\phi)=\begin{pmatrix}
  \Lambda(x)+\gamma(t,x,\phi_4)-d_1(t,x)\phi_1\\
  -\iota(t,x,\phi_2)-d_2(t,x)\phi_2\\
  \iota(t,x,\phi_2)-\gamma(t,x,\phi_3)-d_3(t,x)\phi_3\\
  \gamma(t,x,\phi_3)-\sigma(t,x,\phi_4)-d_4(t,x)\phi_4
\end{pmatrix},$$ $$V[\eta,\nu_{\cdot},h](t,x,\psi)=\int_X\int_Yg(t,\psi,\phi)\rd(\nu_t)_y(\psi)\rd\eta_x(y)+h(t,x,\phi),$$and

\begin{equation*}
\widehat{V}[\eta,\rho_{\cdot},h](t,x,\phi)=\int_X\int_Yg(t,\psi,\phi)\rho(t,y,\psi)\rd\psi\rd\eta_x(y)
+h(t,x,\phi).
\end{equation*}
Consider the VE
\begin{equation}\label{Vlasov-SEIRS}
\begin{split}
&\frac{\partial\rho(t,x,\phi)}{\partial t}+\textrm{div}_{\phi}\left(\rho(t,x,\phi)\widehat{V}[\eta,\rho(\cdot),h](t,x,\phi)\right)=0,\quad t\in(0,T],\ x\in X,\ \mathfrak{m}\text{-a.e.}\ \phi\in Y,\\
 &\rho(0,\cdot)=\rho_0(\cdot).
\end{split}
  \end{equation}

Again from Lemmas~\ref{le-partition}-\ref{le-graph}, there exists
 \begin{enumerate}
 \item[$\bullet$] a partition $\{A^m_i\}_{1\le i\le m}$ of $X$ and points $x^m_i\in A^m_i$ for $i=1,\ldots,m$, for every $m\in\N$,
 \item[$\bullet$]  a sequence $\{\varphi^{m,n}_{(i-1)n+j}\colon i=1,\ldots,$ $m,j=1,\ldots,n\}_{n,m\in\N}\subseteq Y$ and $\{a_{m,i}\colon i=1,\ldots,$ $m\}_{m\in\N}\subseteq\R_+$,
  \item[$\bullet$] a sequence $\{y^{m,n}_{(i-1)n+j}\colon i=1,\ldots,m,j=1,\ldots,n\}_{m,n\in\N}\subseteq Y$ and $\{b_{m,i}\colon i=1,\ldots,m\}_{m\in\N}\subseteq\R_+$,
 \end{enumerate}
 such that
  $$\lim_{m\to\infty}\lim_{n\to\infty}d_{\infty}(\nu_0^{m,n},\nu_0)=0,$$
   $$\lim_{m\to\infty}\lim_{n\to\infty}d_{\infty}(\eta^{m,n},\eta)=0,$$
     $$\lim_{m\to\infty}\int_0^T\int_Y\sup_{x\in X}\lt|h^m(t,x,\phi)-h(t,x,\phi)\rt|\rd\phi\rd t=0,$$
 where  \begin{subequations}
   \label{discretization-SEIRS}
 \begin{alignat}{2}
(\nu_0^{m,n})_x\colon=\sum_{i=1}^m\mathbbm{1}_{A^m_i}(x)\frac{a_{m,i}}{n}\sum_{j=1}^n
\delta_{\varphi^{m,n}_{(i-1)n+j}},\quad x\in X,\\ \eta^{m,n}_x\colon=\sum_{i=1}^m\mathbbm{1}_{A^m_i}(x)\frac{b_{m,i}}{n}
\sum_{j=1}^n\delta_{y^{m,n}_{(i-1)n+j}},\quad x\in X,\\
 h^m(t,z,\phi)=\sum_{i=1}^m\mathbbm{1}_{A^m_i}(z)h(t,x_i^m,\phi),\quad t\in\cT,\ z\in X,\ \phi\in Y.
 \end{alignat}
 \end{subequations}
 Consider the following IVP of a coupled ODE system:
\begin{multline}
  \label{lattice-SEIRS}
  \dot{\phi}_{(i-1)n+j}=F^{m,n}_{i}(t,\phi_{(i-1)n+j},\Phi),\quad 0<t\le T,\quad \phi_{(i-1)n+j}(0)=\varphi_{(i-1)n+j},\\ i=1,\ldots,m,\ j=1,\ldots,n,
\end{multline}
where $\Phi=(\phi_{(i-1)n+j})_{1\le i\le m,1\le j\le n}$ and $$F^{m,n}_{i}(t,\psi,\Phi)=\sum_{p=1}^{m}\frac{a_{m,i}b_{m,p}}{n^2}\sum_{j=1}^n
\mathbbm{1}_{A^m_p}(y^{m,n}_{(i-1)n+j})\sum_{q=1}^{n}g(t,\psi,\phi_{(p-1)n+q})+h^m(t,x^m_i,\psi).$$

Then by Proposition~\ref{prop-well-posed-lattice}, there exists a unique solution $\phi^{m,n}(t)=(\phi^{m,n}_{(i-1)n+j}(t))$ to \eqref{lattice-SEIRS},
for $m,n\in\N$.

For $t\in\cT$, define \begin{equation}\label{Eq-approx-SEIRS}
(\nu_t^{m,n})_x\colon=\sum_{i=1}^m\mathbbm{1}_{A^m_i}(x)
\frac{a_{m,i}}{n}\sum_{j=1}^{n}\delta_{\phi^{m,n}_{(i-1)n+j}(t)},\quad x\in X.
\end{equation}
\begin{theorem}
  Assume ($\mathbf{A1}$), and ($\mathbf{H1}$)-($\mathbf{H6}$). Then there exists a unique uniformly weak solution $\rho(t,x,\phi)$ to \eqref{Vlasov-SEIRS}. Assume additionally $\rho_0(x,\phi)$ is continuous in $x\in X$ for $\mathfrak{m}$-a.e. $\phi\in Y$ such that $\rho_0\in {\sf L}^1_+(X\times Y,\mu_X\otimes\mathfrak{m})$ and $$\sup_{x\in X}\|\rho_0(x,\cdot)\|_{{\sf L}^1(Y,\mathfrak{m})}<\infty.$$ Let $\nu_{\cdot}\in\mathcal{C}(\cT;\mathcal{B}_{\mu_X,1}(X,\cM_{\abs}(Y)))$ be the measure-valued function defined in terms of the uniformly weak solution to \eqref{Vlasov-SEIRS}:
$$\rd(\nu_t)_x=\rho(t,x,\phi)\rd\phi,\quad \text{for every}\quad t\in\cT,\quad \text{and}\quad x\in X.$$ Then $\nu_t\in \mathcal{C}(\cT,\mathcal{C}_{\mu_X,1}(X,\cM_+(Y)))$. Moreover, let
$\nu^{m,n}_0\in\mathcal{B}_{\mu_X,1}(X,\cM_+(Y))$, $\eta^{m,n}\in \mathcal{B}(X,$ $\cM_+(Y))$, and $h^m\in \mathcal{C}(\cT\times X\times Y,\R^2)$ be defined in \eqref{discretization-SEIRS},
and $\nu^{m,n}_{\cdot}$ be defined in \eqref{Eq-approx-SEIRS}. Then $$\lim_{n\to\infty}d_{\infty}(\nu_t^{m,n},\nu_t)=0.$$
\end{theorem}
\begin{remark}
  When $\sigma\equiv0$, the SEIRS model reduces to SEIR model. Similar results apply to other types of epidemic models (e.g., SIRS, SIR, SIS models).
\end{remark}
\begin{proof}
It is straightforward to verify that ($\mathbf{H3}$) implies ($\mathbf{A2}$); ($\mathbf{H1}$) and ($\mathbf{H4}$)-($\mathbf{H6}$) together imply ($\mathbf{A3}$) and ($\mathbf{A7}$); It is easy to verify that the $\eta_{\cdot}$ given in Example~\ref{ex-discrete} satisfies $\mathbf{(A4)'}$. Let

It remains to show ($\mathbf{A6}$) is fulfilled with $Y$ defined in \eqref{Eq-Y-SEIRS}.
\begin{enumerate}
\item[(i)] For $\phi\in\{z\in Y\colon \sum_{i=1}^4z_i=M\}$, we have $\upsilon(\phi)=(1,1,1,1)$, and
\begin{align*}
  V[\eta,\nu_{\cdot},h](t,x,\phi)\cdot\upsilon(\phi)=&\frac{\partial(S_x+E_x+I_x+R_x)}{\partial t}\Big|_{(S_x,E_x,I_x,R_x)=\phi}\\
  =&\Lambda_x-d_1(t,x)\phi_1-d_2(t,x)\phi_2-d_3(t,x)\phi_3-d_4(t,x)\phi_4\\
  \le& \Lambda_x-d(x)\sum_{i=1}^4\phi_i=\Lambda_x-d(x)M\le0.
\end{align*}
\item[(ii)] For $\phi\in\{z\in Y\colon z_1=0,\quad \sum_{j\neq 1}z_j=M\}$, we have  $\upsilon(\phi)=(-1,0,0,0)$, and
\begin{align*}
  V[\eta,\nu_{\cdot},h](t,x,\psi)\cdot\upsilon(\phi)=-\frac{\partial(S_x)}{\partial t}\Big|_{(S_x,E_x,I_x,R_x)=\phi}=-(\Lambda(x)+\gamma(t,x,\phi_4))\le0.
\end{align*}
\item[(iii)] For $\phi\in\{z\in Y\colon z_2=0,\quad \sum_{j\neq 2}z_j=M\}$, we have $\upsilon(\phi)=(0,-1,0,0)$, and
\begin{align*}
  V[\eta,\nu_{\cdot},h](t,x,\phi)\cdot\upsilon(\phi)=&-\frac{\partial(E_x)}{\partial t}\Big|_{(S_x,E_x,I_x,R_x)=\psi}\\
  =&-\int_X\int_Y\beta(t,\psi_1,\phi_3)\rd(\nu_t)_y(\phi)\rd\eta_x(y)\le0.
\end{align*}
\item[(iv)] For $\phi\in\{z\in Y\colon z_3=0,\quad \sum_{j\neq 3}z_j=M\}$, we have $\upsilon(\phi)=(0,0,-1,0)$, and
\begin{align*}
  V[\eta,\nu_{\cdot},h](t,x,\phi)\cdot\upsilon(\phi)=-\frac{\partial(I_x)}{\partial t}\Big|_{(S_x,E_x,I_x,R_x)=\phi}=-\iota(t,x,\phi_2)\le0.
\end{align*}
\item[(v)] For $\phi\in\{z\in Y\colon z_4=0,\quad \sum_{j\neq 4}z_j=M\}$, we have $\upsilon(\phi)=(0,0,0,-1)$, and
\begin{align*}
  V[\eta,\nu_{\cdot},h](t,x,\phi)\cdot\upsilon(\phi)=-\frac{\partial(R_x)}{\partial t}\Big|_{(S_x,E_x,I_x,R_x)=\phi}=-\beta(t,x,\phi_3)\le0.
\end{align*}
\end{enumerate}
Hence ($\mathbf{A6}$) is fulfilled.
\end{proof}
\begin{remark}
  We remark that for this epidemic model, involved functions are locally Lipschitz but not globally Lipschitz.
\end{remark}

\subsection{Lotka-Volterra multi-patch model}

Let $(X,\mathfrak{B}(X),\mu_X)$ be a compact probability space satisfying ($\mathbf{A1}$). Assume that
\medskip

\noindent ($\mathbf{H1}$) $0\le W_1(u),W_2(u)\le u$ for all $u\in \R_+$, and $W_1$ and $W_2$ are odd functions and locally Lipschitz.

\medskip

\noindent ($\mathbf{H2}$) $\eta^1,\ \eta^2\in \mathcal{B}(X,\cM_+(X))$.

\medskip

Let $\Lambda_1,\ \Lambda_2>0$ be an arbitrary positive number satisfying
\begin{equation}
  \label{Lambdacondition}
\Lambda_1\ge\frac{\alpha}{\beta},\quad \Lambda_2\ge-\frac{\iota}{\theta}+\frac{\sigma}{\theta}\Lambda_1.
\end{equation}
Let $Y=\{\phi\in\R^2_+\colon \phi_1\le\Lambda_1,\ \phi_2\le\Lambda_2\}$ be the cube in the positive cone, which is a convex compact set.

\medskip

\noindent{$\mathbf{(H3)}$} $\nu_{\cdot}\in \mathcal{C}(\cT,\mathcal{C}_{\mu_X,1}(X,\cM_+(\R^2)))$  is uniformly compactly supported within $Y\subseteq\R_+^2$.

\medskip

Under ($\mathbf{H1}$)-($\mathbf{H3}$), consider the general Lotka-Volterra nonlocal patch prey-predator model \cite{L74,SL76,C20}:
\begin{equation}\label{LS}
\begin{split}
  \frac{\partial \phi_1(t,x)}{\partial t}=\phi_1(t,x)(\alpha-\beta\phi_1(t,x)-\gamma\phi_2(t,x))
  +\int_X\int_YW_1(\psi_1-\phi_1(t,x))\rd(\nu_t)_y(\psi)\rd\eta^1_x(y)\\
  \frac{\partial \phi_2(t,x)}{\partial t}=\phi_2(t,x)(-\iota+\sigma\phi_1(t,x)-\theta\phi_2(t,x))
  +\int_X\int_YW_2(\psi_2-\phi_2(t,x))\rd(\nu_t)_y(\psi)\rd\eta^2_x(y)
  \end{split}
\end{equation}
where $\phi_1(t)$ and $\phi_2(t)$ stand for population densities of the prey and predator at time $t$, respectively, and all given functions and parameters are non-negative.

Let $$g_1(\psi,\phi)=\begin{pmatrix}
  W_1(\psi_1-\phi_1)\\
  0
\end{pmatrix},\quad g_2(\psi,\phi)=\begin{pmatrix}
  0\\
  W_2(\psi_2-\phi_2)
\end{pmatrix},$$ \[h(\phi)=\begin{pmatrix}
  \phi_1(\alpha-\beta\phi_1-\gamma\phi_2)\\
  \phi_2(-\iota+\sigma\phi_1-\theta\phi_2)
\end{pmatrix},\] and \begin{equation*}
\widehat{V}[\eta,\rho_{\cdot},h](t,x,\phi)=\sum_{\ell=1}^2
\int_X\int_Yg_{\ell}(t,\psi,\phi)\rho(t,y,\psi)\rd\psi\rd\eta^{\ell}_x(y)+h(\phi).
\end{equation*}
Consider the VE
\begin{equation}\label{Vlasov-LS}
\begin{split}
&\frac{\partial\rho(t,x,\phi)}{\partial t}+\textrm{div}_{\phi}\left(\rho(t,x,\phi)\widehat{V}[\eta,\rho(\cdot),h](\phi)\right)=0,\quad t\in(0,T],\ x\in X,\ \mathfrak{m}\text{-a.e.}\ \phi\in Y,\\
 &\rho(0,\cdot)=\rho_0(\cdot).
\end{split}
  \end{equation}

 From Lemmas~\ref{le-partition}-\ref{le-graph}, there exists
 \begin{enumerate}
 \item[$\bullet$] a partition $\{A^m_i\}_{1\le i\le m}$ of $X$ and points $x^m_i\in A^m_i$ for $i=1,\ldots,m$, for every $m\in\N$,
 \item[$\bullet$]  a sequence $\{\varphi^{m,n}_{(i-1)n+j}\colon i=1,\ldots,m,j=1,\ldots,n\}_{n,m\in\N}\subseteq Y$ and $\{a_{m,i}\colon i=1,\ldots,$ $m\}_{m\in\N}\subseteq\R_+$,
  \item[$\bullet$] a sequence $\{y^{\ell,m,n}_{(i-1)n+j}\colon i=1,\ldots,m,j=1,\ldots,n\}_{m,n\in\N}\subseteq Y$ and $\{b_{\ell,m,i}\colon i=1,\ldots,$ $m\}_{m\in\N}\subseteq\R_+$, for $\ell=1,2$,
 \end{enumerate}
 such that
  $$\lim_{m\to\infty}\lim_{n\to\infty}d_{\infty}(\nu_0^{m,n},\nu_0)=0,$$
   $$\lim_{m\to\infty}\lim_{n\to\infty}d_{\infty}(\eta^{\ell,m,n},\eta)=0,\quad \ell=1,\ldots,$$
     $$\lim_{m\to\infty}\int_0^T\int_Y\sup_{x\in X}\lt|h^m(t,x,\phi)-h(t,x,\phi)\rt|\rd\phi\rd t=0,$$
 where
 \begin{subequations}
   \label{discretization-LS}
 \begin{alignat}{2}
(\nu_0^{m,n})_x\colon=\sum_{i=1}^m\mathbbm{1}_{A^m_i}(x)\frac{a_{m,i}}{n}\sum_{j=1}^n
\delta_{\varphi^{m,n}_{(i-1)n+j}},\quad x\in X,\\ \eta^{\ell,m,n}_x\colon=\sum_{i=1}^m\mathbbm{1}_{A^m_i}(x)\frac{b_{\ell,m,i}}{n}
\sum_{j=1}^n\delta_{y^{\ell,m,n}_{(i-1)n+j}},\quad x\in X,\\
 h^m(t,z,\phi)=\sum_{i=1}^m\mathbbm{1}_{A^m_i}(z)h(t,x_i^m,\phi),\quad t\in\cT,\ z\in X,\ \phi\in Y.
 \end{alignat}
\end{subequations}
 Consider the following IVP of a coupled ODE system:
\begin{multline}
  \label{lattice-LS}
  \dot{\phi}_{(i-1)n+j}=F^{m,n}_{i}(t,\phi_{(i-1)n+j},\Phi),\quad 0<t\le T,\quad \phi_{(i-1)n+j}(0)=\varphi_{(i-1)n+j},\\ i=1,\ldots,m,\ j=1,\ldots,n,
\end{multline}
where $\Phi=(\phi_{(i-1)n+j})_{1\le i\le m,1\le j\le n}$ and $$F^{m,n}_{i}(t,\psi,\Phi)=\sum_{p=1}^{m}\frac{a_{m,i}b_{\ell,m,p}}{n^2}\sum_{j=1}^n
\mathbbm{1}_{A^m_p}(y^{\ell,m,n}_{(i-1)n+j})\sum_{q=1}^{n}g(t,\psi,\phi_{(p-1)n+q})+h^m(t,x^m_i,\psi).$$

Then by Proposition~\ref{prop-well-posed-lattice}, there exists a unique solution $\phi^{m,n}(t)=(\phi^{m,n}_{(i-1)n+j}(t))$ to \eqref{lattice-LS},
for $m,n\in\N$.

For $t\in\cT$, define \begin{equation}\label{Eq-approx-LS}
(\nu_t^{m,n})_x\colon=\sum_{i=1}^m\mathbbm{1}_{A^m_i}(x)\frac{a_{m,i}}{n}\sum_{j=1}^{n}\delta_{\phi^{m,n}_{(i-1)n+j}(t)},\quad x\in X.
\end{equation}

\begin{theorem}
  Assume ($\mathbf{A1}$), ($\mathbf{H1}$)-($\mathbf{H2}$), and $\Lambda_1,\ \Lambda_2$ satisfy \eqref{Lambdacondition}. Then there exists a unique uniformly weak solution $\rho(t,x,\phi)$ to \eqref{Vlasov-LS}. Assume additionally $\rho_0(x,\phi)$ is continuous in $x\in X$ for $\mathfrak{m}$-a.e. $\phi\in Y$ such that $\rho_0\in {\sf L}^1_+(X\times Y,\mu_X\otimes\mathfrak{m})$ and $$\sup_{x\in X}\|\rho_0(x,\cdot)\|_{{\sf L}^1(Y,\mathfrak{m})}<\infty.$$ Let $\nu_{\cdot}\in\mathcal{C}(\cT;\mathcal{B}_{\mu_X,1}(X,\cM_{\abs}(Y)))$ be the measure-valued function defined in terms of the uniformly weak solution to \eqref{Vlasov-LS}:
$$\rd(\nu_t)_x=\rho(t,x,\phi)\rd\phi,\quad \text{for every}\quad t\in\cT,\quad \text{and}\quad x\in X.$$ Then $\nu_t\in \mathcal{C}(\cT,\mathcal{C}_{\mu_X,1}(X,\cM_+(Y)))$. Moreover, let
$\nu^{m,n}_0$, $\eta^{\ell,m,n}$ for $\ell=1,2$, and $h^m$ be defined in \eqref{discretization-LS},
and $\nu^{m,n}_{\cdot}$ be defined in \eqref{Eq-approx-LS}. Then $$\lim_{n\to\infty}d_{\infty}(\nu_t^{m,n},\nu_t)=0.$$
\end{theorem}

\begin{proof}
 First note that ($\mathbf{H1}$) implies ($\mathbf{A2}$). It is readily verified that  ($\mathbf{A3}$) and  ($\mathbf{A7}$) are fulfilled since $Y$ is compact. In addition, $\mathbf{(A4)'}$ follows from ($\mathbf{H2}$). Hence it suffices to show that ($\mathbf{A6}$) holds with $Y$ for some $c,\Lambda>0$.

  Note that $\partial Y=\{\phi_1=0\}\cup\{\phi_2=0\}\cup\{\phi_1=\Lambda_1\}\cup\{\phi_2=\Lambda_2\}$. In what follows, we will show that \[V[\eta,\nu_{\cdot},h](t,x,\phi)\cdot\upsilon(\phi)\le0,\quad \text{for all}\ t\in\mathcal{T},\ x\in X,\quad \phi\in\partial Y,\]
  where $\upsilon(\phi)$ is the outer normal vector at $\phi$. We prove it case by case.

\begin{enumerate}
\item[(i)] For $\phi\in\{\varphi\in Y\colon \varphi_1=0\}$, $\upsilon(\phi)=(-1,0)$, and
\begin{align*}
  V[\eta,\nu_{\cdot},h](t,x,\phi)\cdot\upsilon(\phi)
  =-\int_X\int_YW_1(\psi_1-\phi_1)\rd(\nu_t)_y(\phi)\rd\eta^1_x(y)\le0.
\end{align*}
\item[(ii)] For $\phi\in\{\varphi\colon \varphi_2=0\}$, $\upsilon(\phi)=(0,-1)$, and
\begin{align*}
  V[\eta,\nu_{\cdot},h](t,x,\phi)\cdot\upsilon(\phi)
  =-\int_X\int_YW_2(\psi_2-\phi_2)\rd(\nu_t)_y(\phi)\rd\eta^2_x(y)\le0.
\end{align*}
\item[(iii)] For $\phi\in\{\varphi\colon \varphi_1=\Lambda_1\}$, $\upsilon(\phi)=(1,0)$. By ($\mathbf{H1}$),
\begin{align*}
  &V[\eta,\nu_{\cdot},h](t,x,\phi)\cdot\upsilon(\phi)\\
  =&\Lambda_1(\alpha-\beta\Lambda_1-\gamma\phi_2)
  +\int_X\int_YW_1(\psi_1-\Lambda_1)\rd(\nu_t)_y(\phi)\rd\eta^1_x(y)\\
  \le&\Lambda_1(\alpha-\beta\Lambda_1)\le0,
\end{align*}
since $\Lambda_1\ge\frac{\alpha}{\beta}$.
\item[(iv)] For $\phi\in\{\varphi\colon \varphi_2=\Lambda_2\}$, $\upsilon(\phi)=(1,0)$. By ($\mathbf{H1}$),
\begin{align*}
  &V[\eta,\nu_{\cdot},h](t,x,\phi)\cdot\upsilon(\phi)\\
  =&\Lambda_2(-\iota+\sigma\phi_1-\theta\Lambda_2)+\int_X\int_YW_2(\psi_2-\Lambda_2)\rd(\nu_t)_y(\phi)
  \rd\eta^2_x(y)\\
  \le&\Lambda_2(-\iota+\sigma\Lambda_1-\theta\Lambda_2)\le0,
\end{align*}
since $\Lambda_2\ge-\frac{\iota}{\theta}+\frac{\sigma}{\theta}\Lambda_1$.
\end{enumerate}
\end{proof}

\begin{remark}
 \begin{enumerate}
    \item[$\bullet$] Analogous results can be derived for models describing interaction between an abundant prey and a rare predator with an Allee effect \cite{L74}, by alternating the sign of $\theta$ in \eqref{LS} to represent the self-activation (i.e., Allee effect) instead of the self-inhibition for the predator. In particular, application of our main results covers the graphon model proposed in \cite{C20}, where \eqref{LS} with $\int_X\int_YW_1(\psi_1-\phi_1(t,x))\rd(\nu_t)_y(\psi)$ $\rd\eta^1_x(y)$ replaced by
     \[D_1\int_XK(x,y)(\phi_1(y,t)-\phi_1(x,t))\rd y\] and $\int_X\int_YW_2(\psi_2-\phi_2(t,x))\rd(\nu_t)_y(\psi)\rd\eta^2_x(y)$ replaced by
     \[D_2\int_XK(x,y)(\phi_2(y,t)-\phi_2(x,t))\rd y,\] with $X=[0,1]$ and $K$ the adjacency function.
     \item[$\bullet$] When $\sigma$ is negative, the system can be used to model sRNA pathways with heterogeneous structure \cite{BF11}, where it was shown that a convex compact positively invariant set can be constructed.
  \item[$\bullet$] The integro-differential equation \eqref{LS} can be regarded as a generalization of the reaction diffusion Lotka-Volterra model.
  \end{enumerate}
\end{remark}

\subsection{Hegselmann-Krause opinion dynamics model}

To model opinion dynamics of multi-agents, consider the Hegselmann-Krause model \cite{HK02}:
\begin{alignat*}{2}
\dot{\phi}_i=&\frac{1}{N}\sum_{j=1}^NG(|\phi_j-\phi_i|)(\phi_j-\phi_i),\quad 0<t\le T,\\
\phi_i(0)=&\varphi_i,\quad i=1,\ldots,N.
\end{alignat*}
where $\phi_i^N\in\R^d$ stands for the opinion of agent $i$, and $G\colon\R_+\to\R_+$ the interaction function.

Let $(X,\mu_X)$ be a compact measurable Polish space satisfying ($\mathbf{A1}$). Assume that

\medskip

\noindent ($\mathbf{H1}$) $G\colon\R^d\to\R_+$ is locally Lipschitz continuous.

\medskip
\noindent{$\mathbf{(H2)}$} $\eta\in \mathcal{C}(X,\cM_+(X))$.

Let $\Lambda>0$ be an arbitrary positive real number, and \begin{equation}
  \label{Eq-Y-HK}
  Y=\{\varphi\in\R^d\colon|\varphi_i|\le\Lambda,\quad i=1,\ldots,d\}\subseteq\R^d
\end{equation} be a cube centered at the origin. Hence $Y$ is a convex compact set.

\medskip
\noindent{$\mathbf{(H3)}$} $\nu_{\cdot}\in \mathcal{C}(\cT,\mathcal{B}_{\mu_X,1}(X,\cM_+(\R^d)))$  is uniformly compactly supported within $Y$.

To generalize this network model, consider the following non-local model with heterogeneous structure (in terms of $\eta$)  under $\mathbf{(H1)}$-$\mathbf{(H3)}$:
\begin{equation*}
\begin{split}
  \frac{\partial \phi}{\partial t}(t,x)=&\int_{X}\int_{\R^d}G(|\psi-\phi(t,x)|)(\psi-\phi(t,x))\rd(\nu_t)_y(\psi)\rd \eta_x(y)\\
  \phi(0,x)=&\varphi(x),
\end{split}
\end{equation*}
where $\phi(t,x)\in\R^d$.

Consider the VE
\begin{equation}\label{Vlasov-HK}
\begin{split}
&\frac{\partial\rho(t,x,\phi)}{\partial t}+\textrm{div}_{\phi}\left(\rho(t,x,\phi)\widehat{V}[\eta,\rho(\cdot)](t,x,\phi)\right)=0,\quad t\in(0,T],\ x\in X,\ \mathfrak{m}\text{-a.e.}\ \phi\in Y,\\
 &\rho(0,\cdot)=\rho_0(\cdot),
\end{split}
  \end{equation}
  where $\widehat{V}[\eta,\rho(\cdot)](t,x,\phi)
  =\int_{X}\int_{\R^d}G(|\psi-\phi(t,x)|)(\psi-\phi(t,x))\rho(t,y,\psi)\rd\psi\rd \eta_x(y)$.

 From Lemmas~\ref{le-partition}-\ref{le-graph}, there exists
 \begin{enumerate}
 \item[$\bullet$] a partition $\{A^m_i\}_{1\le i\le m}$ of $X$ and points $x^m_i\in A^m_i$ for $i=1,\ldots,m$, for every $m\in\N$,
 \item[$\bullet$]  a sequence $\{\varphi^{m,n}_{(i-1)n+j}\colon i=1,\ldots,m,j=1,\ldots,n\}_{n,m\in\N}\subseteq Y$ and $\{a_{m,i}\colon i=1,\ldots,$ $m\}_{m\in\N}\subseteq\R_+$,
  \item[$\bullet$] a sequence $\{y^{m,n}_{(i-1)n+j}\colon i=1,\ldots,m,j=1,\ldots,n\}_{m,n\in\N}\subseteq Y$ and $\{b_{m,i}\colon i=1,\ldots,$ $m\}_{m\in\N}\subseteq\R_+$,
 \end{enumerate}
 such that
  $$\lim_{m\to\infty}\lim_{n\to\infty}d_{\infty}(\nu_0^{m,n},\nu_0)=0,$$
   $$\lim_{m\to\infty}\lim_{n\to\infty}d_{\infty}(\eta^{m,n},\eta)=0,$$
     $$\lim_{m\to\infty}\int_0^T\int_Y\sup_{x\in X}\lt|h^m(t,x,\phi)-h(t,x,\phi)\rt|\rd\phi\rd t=0,$$
 where
 \begin{subequations}
   \label{discretization-HK}
 \begin{alignat}{2}
(\nu_0^{m,n})_x\colon=\sum_{i=1}^m\mathbbm{1}_{A^m_i}(x)\frac{a_{m,i}}{n}\sum_{j=1}^n
\delta_{\varphi^{m,n}_{(i-1)n+j}},\quad x\in X,\\ \eta^{m,n}_x\colon=\sum_{i=1}^m\mathbbm{1}_{A^m_i}(x)\frac{b_{m,i}}{n}
\sum_{j=1}^n\delta_{y^{m,n}_{(i-1)n+j}},\quad x\in X,\\
 h^m(t,z,\phi)=\sum_{i=1}^m\mathbbm{1}_{A^m_i}(z)h(t,x_i^m,\phi),\quad t\in\cT,\ z\in X,\ \phi\in Y.
 \end{alignat}
 \end{subequations}
 Consider the following IVP of a coupled ODE system:
\begin{multline}
  \label{lattice-HK}
  \dot{\phi}_{(i-1)n+j}=F^{m,n}_{i}(t,\phi_{(i-1)n+j},\Phi),\quad 0<t\le T,\quad \phi_{(i-1)n+j}(0)=\varphi_{(i-1)n+j},\\ i=1,\ldots,m,\ j=1,\ldots,n,
\end{multline}
where $\Phi=(\phi_{(i-1)n+j})_{1\le i\le m,1\le j\le n}$ and $$F^{m,n}_{i}(t,\psi,\Phi)=\sum_{p=1}^{m}\frac{a_{m,i}b_{m,p}}{n^2}\sum_{j=1}^n
\mathbbm{1}_{A^m_p}(y^{m,n}_{(i-1)n+j})\sum_{q=1}^{n}g(t,\psi,\phi_{(p-1)n+q})+h^m(t,x^m_i,\psi).$$

Then by Proposition~\ref{prop-well-posed-lattice}, there exists a unique solution $\phi^{m,n}(t)=(\phi^{m,n}_{(i-1)n+j}(t))$ to \eqref{lattice-HK},
for $m,n\in\N$.

For $t\in\cT$, define \begin{equation}\label{Eq-approx-HK}
(\nu_t^{m,n})_x\colon=\sum_{i=1}^m\mathbbm{1}_{A^m_i}(x)\frac{a_{m,i}}{n}
\sum_{j=1}^{n}\delta_{\phi^{m,n}_{(i-1)n+j}(t)},\quad x\in X.
\end{equation}
\begin{theorem}
  Assume ($\mathbf{A1}$), and ($\mathbf{H1}$)-($\mathbf{H2}$). Then there exists a unique uniformly weak solution $\rho(t,x,\phi)$ to \eqref{Vlasov-HK}. Assume additionally $\rho_0(x,\phi)$ is continuous in $x\in X$ for $\mathfrak{m}$-a.e. $\phi\in Y$ such that $\rho_0\in {\sf L}^1_+(X\times Y,\mu_X\otimes\mathfrak{m})$ and $$\sup_{x\in X}\|\rho_0(x,\cdot)\|_{{\sf L}^1(Y,\mathfrak{m})}<\infty.$$ Let $\nu_{\cdot}\in\mathcal{C}(\cT;\mathcal{B}_{\mu_X,1}(X,\cM_{\abs}(Y)))$ be the measure-valued function defined in terms of the uniformly weak solution to \eqref{Vlasov-HK}:
$$\rd(\nu_t)_x=\rho(t,x,\phi)\rd\phi,\quad \text{for every}\quad t\in\cT,\quad \text{and}\quad x\in X.$$ Then $\nu_t\in \mathcal{C}(\cT,\mathcal{C}_{\mu_X,1}(X,\cM_+(Y)))$. Moreover, let
$\nu^{m,n}_0\in\mathcal{B}_{\mu_X,1}(X,\cM_+(Y))$, $\eta^{m,n}\in \mathcal{B}(X,$ $\cM_+(Y))$, and $h^m\in \mathcal{C}(\cT\times X\times Y,\R^{2d})$ be defined in \eqref{discretization-HK},
and $\nu^{m,n}_{\cdot}$ be defined in \eqref{Eq-approx-HK}. Then $$\lim_{n\to\infty}d_{\infty}(\nu_t^{m,n},\nu_t)=0.$$
\end{theorem}
\begin{proof}
Let $g(\psi,\phi)=G(|\psi-\phi|)(\psi-\phi)$. It follows from $\mathbf{(H1)}$ that $g$ satisfies $\mathbf{(A2)}$. Indeed, $\mathbf{(H1)}$ implies $G$ is locally bounded Lipschitz, and hence for $(\psi,\phi),(\widetilde{\psi},\widetilde{\psi})\in\mathcal{N}$, where $\mathcal{N}$ is an open set in $\R^{2d}$, we have
\[\begin{split}
  &|G(|\psi-\phi|)(\psi-\phi)-G(|\widetilde{\psi}-\phi|)(\widetilde{\psi}-\phi)|\\
  \le&
  |G(|\psi-\phi|)(\psi-\phi)-G(|\psi-\phi|)(\widetilde{\psi}-\phi)|
  +|G(|\psi-\phi|)(\widetilde{\psi}-\phi)-G(|\widetilde{\psi}-\phi|)(\widetilde{\psi}-\phi)|\\
  \le&|G(|\psi-\phi|)||\psi-\widetilde{\psi}|
  +\frac{|G(|\psi-\phi|)-G(|\widetilde{\psi}-\phi|)|}{|\psi-\widetilde{\psi}|}|\psi-\widetilde{\psi}|
  |\widetilde{\psi}-\phi|\\
  \le&\BL(G|_{\mathcal{N}})(1+\max_{(w,u)\in\mathcal{N}}(|w|+|u|))|\psi-\widetilde{\psi}|,
\end{split}\]i.e., $g$ is locally Lipschitz in $\psi$.
Similarly, one can show $g$ is locally Lipschitz in $\phi$. Hence $g$ is locally Lipschitz in $(\psi,\phi)$. Moreover, $\mathbf{(A3)}$ and $\mathbf{(A7)}$ are automatically fulfilled with a trivial $h\equiv0$; $\mathbf{(H2)}$ implies $\mathbf{(A4)'}$.
  Finally we show ($\mathbf{A6}$) is satisfied with the $Y$ given in \eqref{Eq-Y-HK}, under the assumption of $\mathbf{(H3)}$. To see this, let $D_j^+=\{\varphi\in Y\colon \varphi_j=\Lambda\}$ and $D_j^-=\{\varphi\in Y\colon \varphi_j=-\Lambda\}$, $j=1,\ldots,d$. Hence $\partial Y=\cup_{j=1}^d(\cup D_j^+\cup D_j^-)$.
  Let $1\le j\le d$. Then for $\phi\in D_j^+$, $\upsilon(\phi)=e_j$, where $(e_j)_{1\le j\le d}$ forms the standard orthonormal basis of $\R^d$. Hence \begin{align*}
  &V[\eta,\nu_{\cdot}](t,x,\phi)\cdot\upsilon(\phi)
  =\int_{X}\int_{Y}G(|\psi-\phi(t,x)|)(\psi_j-\Lambda)\rd(\nu_t)_y(\psi)\rd \eta_x(y)\le0.
\end{align*}
Similarly, for $\phi\in D_j^+$, $\upsilon(\phi)=-e_j$, we have
\begin{align*}
  &V[\eta,\nu_{\cdot}](t,x,\phi)\cdot\upsilon(\phi)
  =-\int_{X}\int_{Y}G(|\psi-\phi(t,x)|)(\psi_j+\Lambda)\rd(\nu_t)_y(\psi)\rd \eta_x(y)\le0.
\end{align*}
This shows that ($\mathbf{A6}$) is satisfied.
\end{proof}

\section{Discussion and outlooks}\label{sect-discussion}

The approach in this paper naturally extends to the setting on a compact Riemannian manifold, e.g., the torus $\mathbb{T}^m$ or sphere $\mathbb{S}^m$, which are typically used in the Kuramoto models of lower or higher order interactions \cite{BBK20}, the swarm sphere model \cite{O06}, as well as models for opinion dynamics \cite{CLP15}.

Nevertheless, the compactness of the vertex space $X$ is technically crucial. Without this, there may exist no sequence of partitions of $X$ with the maximal diameters decreasing to zero. This property is indispensable for the topology induced by $d_{\infty}$, as the continuity may not be sufficient to ensure the pointwise convergence of the distances between the fiber measures $\eta_x$ and their (best) uniform discrete approximations. Nonetheless, such compactness can be sacrificed, e.g., by adding some other mild condition on the homogeneity as well as uniform boundedness in total variation norm of $\eta_x$ for large $x$.

Although we only allow finitely many DGMs for the IPS, one may further consider IPS on countably many DGMs. This may lead to a deeper understanding how complex the geometry of digraphs (in other words, the interactions among different particles) will change the dynamics of the IPS. One other crucial matter that may arise is that the union of underlying vertex spaces may be of infinite dimensional (no longer a set in a Euclidean space).

In addition, the same topic remains open when we lose continuity (up to a finite set of discontinuity points) of underlying DGMs as well as the initial measure of the VEs in the vertex variable. Such continuity conditions seem crucial due to the topology induced by the uniform bounded Lipschitz metric. Since allowing countably many discontinuity points means the mass of fiber measures $\eta_x$ for $x\in X$ may not be uniformly bounded. Hence the distance between the DGM $\eta$ (even still in $\mathcal{B}(X,\cM_+(X))$) and any finitely supported approximation ($\eta^{m,n}$) can be uniformly away from zero. For instance, consider the DGM in Remark~\ref{re-not-dense}. This may indicate a nice weaker topology than the uniform weak topology given in this paper will be helpful in further generalizations.

Moreover, to find a better analytical perspective or to enlarge the space of digraph measures may allow us to address the MFL of dynamical systems on heterogeneous graphs (particularly sparse ones). We will leave the above questions for our future work.

\section{Proofs of main results}\label{sect-proof}
\subsection{Proof of Theorem~\ref{theo-well-posedness-characteristic}}
\begin{proof}
 By Proposition~\ref{Vlasovf}, the Picard-Lindel\"{o}f's iteration \cite[Theorem~2.2]{T12} yields the unique existence of solutions $\phi(t,x)$ to \eqref{Charac}; moreover, it is standard to extend the solution $\phi(t,x)$ to the maximal existence interval $(\tau_{x,t_0}^-,\tau_{x,t_0}^+)\cap\cT$ for every $x\in X$ and $t_0\in\cT$ satisfying the dichotomy in (i)-(ii).

  The positive invariance of $Y$ follows directly from ($\mathbf{A6}$), using Bony's condition \cite{B69} (see also \cite{C72,R72,W98}). This further yields that $\phi(t,x)\in Y$ for all $t\in[t_0,\tau_{x,t_0}^+)\cap\cT$ and thus (i-b) is impossible. This shows (i-a) holds.
  Since $t_0\in\cT$ is arbitrary, the solution maps $\left\{\mathcal{S}^x_{t,s}[\eta,\nu_{\cdot},h]\right\}_{t,s\in\cT}$ form a group on $Y$ such that
$$\phi(t,x)=\mathcal{S}^x_{t,s}[\eta,\nu_{\cdot},h]\phi(s,x),\quad \text{for all}\ s,t\in\cT.$$
\end{proof}

\subsection{Proof of Theorem~\ref{th-equi}}
\begin{proof}
For $\nu_0\in\mathcal{B}_{\mu_X,1}(X;\cM_{+,\abs}(Y))$ defined in terms of $\rho_0$, let $\nu_{\cdot}$ be the unique solution to the fixed point equation \eqref{Fixed} associated with $\nu_0$. Hence by Proposition~\ref{prop-continuousdependence}(v)-(vi), $\nu_t\in\mathcal{B}_{\mu_X,1}(X;\cM_{+,\abs}(Y))$ for every $t\in\cT$. Let \begin{equation}\label{eq-weak}\rho(t,x,\phi)=\frac{\rd(\nu_t)_x(\phi)}{\rd\phi},\quad t\in\cT,\ x\in X,\ \mathfrak{m}\text{-a.e.}\ \phi\in Y.\end{equation}
To conclude the proof, (1) we show $\rho(t,x,\phi)$ is a weak solution to \eqref{Vlasov}. (2) For every weak solution $\rho(t,x,\phi)$ to the IVP of \eqref{Vlasov}, let
\[\rd(\nu_t)_x(\phi)=\rho(t,x,\phi)\rd\phi\rd\mu_X(x).\]
We show that $\nu_{\cdot}\in \mathcal{C}(\cT,\mathcal{B}_{\mu_X,1}(X,\cM_+(Y)))$, is a solution to the fixed point equation associated with $\nu_0$. Then by the uniqueness of solutions to the fixed point equation, we obtain the uniqueness of weak solutions to the IVP of \eqref{Vlasov}.
\begin{enumerate}
\item[Step I.] We show $\rho(t,x,\phi)$ in \eqref{eq-weak} is a weak solution to \eqref{Vlasov}.
First, by Proposition~\ref{prop-nu}(ii), $t\mapsto\rho(t,x,\phi)$ is uniformly weakly continuous. Moreover, by Proposition~\ref{prop-continuousdependence}(v), $$\int_X\int_Y\rho(t,x,\phi)\rd\phi\rd x=1,\quad \text{for all}\ t\in\cT.$$ It remains to verify that $\rho(t,x,\phi)$ solves \eqref{eq-test}. Let $w\in \mathcal{C}^1(\cT\times Y)$ with $\supp w\subseteq [0,T[\times U$ and $U\subset\subset Y$. Then
\begin{align*}
&\int_0^T\int_Y\frac{\partial w(t,\phi)}{\partial t}\rho(t,x,\phi)\rd\phi\rd t\\
  =&\int_0^T\int_Y\frac{\partial w(t,\phi)}{\partial t}\rd(\nu_t)_x(\phi)\rd t\\
  =&\int_0^T\int_Y\frac{\partial w(t,\phi)}{\partial t}\rd((\nu_0)_x\circ \mathcal{S}_{0,t}^x[\eta,\nu_{\cdot},h])(\phi)\rd t\\
  &\hspace{-.4cm}\xlongequal{\mathcal{S}_{0,t}^x[\eta,\nu_{\cdot},h]\phi\mapsto\phi}
  \int_0^T\int_Y\partial_1 w(t,\mathcal{S}_{t,0}^x[\eta,\nu_{\cdot},h]\phi)\rd(\nu_0)_x(\phi)\rd t\\
  =&\int_Y\int_0^T\partial_1 w(t,\mathcal{S}_{t,0}^x[\eta,\nu_{\cdot},h]\phi)\rd t\rd(\nu_0)_x(\phi)\\
  =&-\int_Yw(0,\phi)\rd(\nu_0)_x(\phi)\\&-\int_Y\int_0^T\partial_2 w(t,\mathcal{S}^x_{t,0}[\eta,\nu_{\cdot},h]\phi)\cdot V[\eta,\nu_{\cdot},h](t,x,\mathcal{S}_{t,0}^x[\eta,\nu_{\cdot},h]\phi)\rd t\rd(\nu_0)_x(\phi)\\
  &\hspace{-.4cm}\xlongequal{\mathcal{S}_{t,0}^x[\eta,\nu_{\cdot},h]\phi\mapsto\phi}
  -\int_Yw(0,\phi)\rd(\nu_0)_x(\phi)-\int_Y\int_0^T\frac{\partial w(t,\phi)}{\partial \phi}\cdot V[\eta,\nu_{\cdot},h](t,x,\phi)\rd t\rd(\nu_t)_x(\phi)\\
  =&-\int_Yw(0,\phi)\rho_0(x,\phi)\rd\phi-\int_0^T\int_Y\frac{\partial w}{\partial \phi}(t,\phi)\cdot \widehat{V}[\eta,\rho_{\cdot},h](t,x,\phi)\rho(t,x,\psi)\rd\phi\rd t,
\end{align*}i.e., \eqref{eq-test} holds.
\item[Step II.] Let $\cT\times X\ni(t,x)\mapsto(\nu_t)_x\in\cM_+(Y)$ be such that $$\rd(\nu_t)_x(\phi)=\rho(t,x,\phi)\rd\phi,\quad t\in\cT,\ x\in X,\ \mathfrak{m}\text{-a.e.}\ \phi\in Y.$$ Since $\rho$ is a weak solution to \eqref{Vlasov}, it is ready to show that $\nu_{\cdot}\in\mathcal{C}(\cT,\mathcal{B}_{\mu_X,1}(X,$ $\cM_+(Y)))$, by
    Definition~\ref{def-weak-sol}(i)-(ii) as well as Proposition~\ref{prop-nu}(i).

    Then it remains to show that $\nu_t$ satisfies the fixed point equation. For $x\in X$, for every $w\in \mathcal{C}^1(\cT\times U)$ with $\supp w\subseteq [0,T[\times U$ and $U\subset\subset Y$, we have  \begin{align*}
&\int_0^T\int_Y\frac{\partial w(t,\phi)}{\partial t}\rd(\nu_t)_x(\phi)\rd t\\
=&-\int_Yw(0,\phi)\rd(\nu_0)_x(\phi)-\int_Y\int_0^T\frac{\partial w(t,\phi)}{\partial \phi}\cdot V[\eta,\nu_{\cdot},h](t,x,\phi)\rd t\rd(\nu_t)_x(\phi)\\
  =&-\int_Yw(0,\phi)\rd(\nu_0)_x(\phi)\\
  &-\int_Y\int_0^T\partial_2 w(t,\mathcal{S}^x_{t,0}[\eta,\nu_{\cdot},h]\phi)\cdot V[\eta,\nu_{\cdot},h](t,x,\mathcal{S}_{t,0}^x[\eta,\nu_{\cdot},h]\phi)\rd t\rd(\nu_0)_x(\phi)\\
  =&\int_Y\int_0^T\partial_1 w(t,\mathcal{S}_{t,0}^x[\eta,\nu_{\cdot},h]\phi)\rd t\rd(\nu_0)_x(\phi)\\
=&\int_0^T\int_Y\partial_1 w(t,\mathcal{S}_{t,0}^x[\eta,\nu_{\cdot},h]\phi)\rd(\nu_0)_x(\phi)\rd t\\
  =&\int_0^T\int_Y\frac{\partial w(t,\phi)}{\partial t}\rd((\nu_0)_x\circ \mathcal{S}_{0,t}^x[\eta,\nu_{\cdot},h])(\phi)\rd t,
\end{align*}
Let $v=\partial_1w$. Since  $w\in \mathcal{C}^1(\cT\times U)$ with $\supp w\subseteq [0,T[\times U$, we have $v\in\mathcal{C}(\cT\times U)$ with $\supp v\subseteq [0,T[\times U$ satisfying
\begin{equation}
  \label{identity-weak}
  \int_0^T\int_Yv(t,\phi)\rd(\nu_t)_x(\phi)\rd t=\int_0^T\int_Yv(t,\phi)\rd((\nu_0)_x\circ \mathcal{S}_{0,t}^x[\eta,\nu_{\cdot},h])(\phi)\rd t.
\end{equation}
 Let $t\in\cT$, $x\in X$, and $f\in \mathcal{C}(Y)$. Let $u=\mathbbm{1}_{\{t\}}$ and $$v(\tau,\phi)=u(\tau)f(\phi),\quad \tau\in\cT,\ \phi\in Y.$$ Since $\mathcal{C}(\cT)$ is dense in ${\sf L}^1(\cT)$, it is easy to construct mollifiers $\{u^n\}_{n\in\N}\subseteq\mathcal{C}(\cT)$ of $u$ with $\supp u^n\subseteq[0,T[$ and $0\le u^n\le1$ such that $$\lim_{n\to\infty}\int_0^T|u^n(\tau)-u(\tau)|\rd\tau=0.$$ Let $v^n(\tau,\phi)=u^n(\tau)f(\phi)$. Then $v^n\in\mathcal{C}(\cT\times U)$ with $\supp v\subseteq [0,T[\times U$. Substituting $v^n$ into \eqref{identity-weak} and taking the limit as $n\to\infty$, by Dominated Convergence Theorem, we have
\begin{equation*}
  \int_Yf(\phi)\rd(\nu_t)_x(\phi)=\int_Yf(\cS_{t,0}^x[\eta,\nu_{\cdot},h]\phi)\rd(\nu_0)_x(\phi).
\end{equation*}
Since $x\in X$ and $f$ were arbitrary, we have $$(\nu_t)_x=(\nu_0)_x\circ\cS_{0,t}^x[\eta,\nu_{\cdot},h].$$
Hence $\nu_{\cdot}$ is a solution to the fixed point equation.
\end{enumerate}
\end{proof}

\subsection{Proof of Theorem~\ref{th-approx}}
\begin{proof}
We first show that $\nu_{\cdot}\in \mathcal{C}(\cT,\mathcal{C}_{\mu_X,1}(X,\cM_+(Y)))$ and then prove the approximation result.

\begin{enumerate}
\item[$\bullet$] $\nu_{\cdot}\in \mathcal{C}(\cT,\mathcal{C}_{\mu_X,1}(X,\cM_+(Y)))$.  Since $x\mapsto\rho_0(x,\phi)$ is continuous  for $\mathfrak{m}$-a.e. $\phi\in Y$ and $$\sup_{x\in X}\|\rho_0(x,\cdot)\|_{{\sf L}^1(Y,\mu_Y)}<\infty,$$ by Proposition~\ref{prop-1} and Dominated Convergence Theorem,
\begin{align*}
  &d_{\sf BL}((\nu_0)_x,(\nu_0)_y)\le\|\rho_0(x,\cdot)-\rho_0(y,\cdot)\|_{{\sf L}^1(Y,\mu_Y)}\to0,\quad \text{as}\ |x-y|\to0,
\end{align*}
which implies that $\nu_0\in \mathcal{C}_{\mu_X,1}(X,\cM_+(Y))$. Since $\mathcal{C}(\cT,\mathcal{C}_{\mu_X,1}(X,\cM_+(Y)))$ is complete by Proposition~\ref{alpha-complete}, applying Banach fixed point theorem as in the proof of Proposition~\ref{prop-sol-fixedpoint} to $\mathcal{C}(\cT,\mathcal{C}_{\mu_X,1}(X,\cM_+(Y)))$  as well as using the uniqueness of solutions to VE yields $\nu_{\cdot}\in \mathcal{C}(\cT,\mathcal{C}_{\mu_X,1}(X,\cM_+(Y)))$.
\item[$\bullet$] Approximation of the VE. We prove the approximation result in four steps. Roughly, we first show $\nu^{m,n}_{\cdot}$ is the solution to a fixed point equation. Then constructing  solutions to two other auxiliary fixed point equations, using continuous dependence of solutions of the fixed point equation on $\eta_{\cdot}$, $h$, as well as the initial measure, we show the approximation result by triangle inequalities.

\medskip
\noindent Step I. $\nu^{m,n}_{\cdot}$ defined in \eqref{Eq-approx} is the solution to the fixed point equation associated with $\eta^{m,n}$ and $h^m$: \begin{equation}\label{Eq-fixed-approx}\nu^{m,n}_{\cdot}
=\cA[\eta^{m,n},\nu^{m,n}_{\cdot},h^m]\nu^{m,n}_{\cdot}.\end{equation}

To prove this, we calculate $\cA[\eta^{m,n},\nu^{m,n}_{\cdot},h^m]$ explicitly and show $\nu^{m,n}_{\cdot}$ satisfies \eqref{Eq-fixed-approx}. By uniqueness of solutions, we prove that $\nu^{m,n}_{\cdot}$ is the unique solution to \eqref{Eq-fixed-approx}. We express the Vlasov operator first. For $x\in A^m_i$,
\begin{align*}
  &V[\eta^{m,n},\nu^{m,n}_{\cdot},h^m](t,x,\phi)\\
  =&\sum_{\ell=1}^r\int_X\int_Yg_{\ell}(t,\psi,\phi)\rd(\nu_t^{m,n})_y(\psi)\rd\eta^{\ell,m,n}_x(y)
  +h^m(t,x,\phi)\\
  =&\sum_{\ell=1}^r\frac{b_{\ell,m,i}}{n}
  \sum_{j=1}^n\int_Yg_{\ell}(t,\psi,\phi)\rd(\nu_t^{m,n})_{y^{m,n}_{(i-1)n+j}}(\psi)+h^m(t,x_i^m,\phi)\\
  =&\sum_{\ell=1}^r\frac{b_{\ell,m,i}}{n}
  \sum_{j=1}^n\sum_{p=1}^m\frac{a_{m,p}}{n}\mathbbm{1}_{A^m_p}(y^{\ell,m,n}_{(i-1)n+j})
  \sum_{q=1}^ng_{\ell}(t,\phi^{m,n}_{(p-1)n+q}(t),\phi)+h^m(t,x_i^m,\phi)\\
  =&F^{m,n}_i(t,\phi,\Phi^{m,n}(t)),
\end{align*}
where $\Phi^{m,n}(t)=(\phi^{m,n}_{(i-1)n+j}(t))_{1\le i\le m,1\le j\le n}$.

Consider the equation of characteristics for every $x\in X$:
\[\frac{\partial\phi(t,x)}{\partial t}=V[\eta^{m,n},\nu^{m,n}_{\cdot},h^m](t,x,\phi),\]
and let $\cS^x_{t,0}[\eta^{m,n},\nu^{m,n}_{\cdot},h^m]$ be its flow. Let $(\phi^{m,n}_{(i-1)n+j}(t))_{1\le i\le m,1\le j\le n}$ be the solution to the coupled ODE system \eqref{lattice} subject to the initial condition $$(\varphi^{m,n}_{1},\ldots,\varphi^{m,n}_{n},\ldots,
\varphi^{m,n}_{(i-1)n+1},\ldots,\varphi^{m,n}_{(i-1)n+n},\ldots,
\varphi^{m,n}_{(m-1)n+1},\ldots,\varphi^{m,n}_{mn}).$$ We first show that for $i=1,\ldots,m$, for $x\in A^m_i$, $j=1,\ldots,n$,
\begin{equation}\label{Eq-15}\phi^{m,n}_{(i-1)n+j}(t)
=\cS^x_{t,0}[\eta^{m,n},\nu^{m,n}_{\cdot},h^m]\varphi^{m,n}_{(i-1)n+j}.\end{equation}
Indeed, \begin{align*}
  &\varphi^{m,n}_{(i-1)n+j}+\int_0^tV[\eta^{m,n},\eta^{m,n}_{\cdot},h^m]
  (\tau,x,\phi^{m,n}_{(i-1)n+j}(\tau))\rd\tau\\
  =&\varphi^{m,n}_{(i-1)n+j}+\int_0^tF^{m,n}_i(\tau,\phi^{m,n}_{(i-1)n+j}
  (\tau),\Phi^{m,n}(\tau))\rd\tau=\phi^{m,n}_{(i-1)n+j}(t),
\end{align*}
for $\phi^{m,n}_{(i-1)n+j}(t)$ is the solution to \eqref{lattice}. This verifies \eqref{Eq-15}, from which we can conclude that \begin{equation}\label{Eq-17}(\nu^{m,n}_t)_x=(\nu_0^{m,n})_x\circ \cS^x_{0,t}[\eta^{m,n},\eta^{m,n}_{\cdot},h^m],\quad x\in X,\end{equation} and hence \eqref{Eq-fixed-approx} holds. To see this,
pick an arbitrary Borel measurable set $B\in\mathcal{B}(Y)$, let $f=\mathbbm{1}_B$.
Then for $x\in A^m_i$,
\begin{align*}
  &\int f\rd(\nu^{m,n}_0)_x\circ \cS^x_{0,t}[\eta^{m,n},\eta^{m,n}_{\cdot},h^m]\\
  =&\int f\circ \cS^x_{t,0}[\eta^{m,n},\eta^{m,n}_{\cdot},h^m]\rd(\nu^{m,n}_0)_x\\
  =&\frac{a_{m,i}}{n}
  \sum_{j=1}^nf\lt(\cS^x_{t,0}[\eta^{m,n},\eta^{m,n}_{\cdot},h^m]\varphi^{m,n}_{(i-1)n+j}\rt)\\
  =&\frac{a_{m,i}}{n}
  \sum_{j=1}^nf\lt(\phi^{m,n}_{(i-1)n+j}(t)\rt)=\int_Yf\rd(\nu^{m,n}_t)_x,
\end{align*}
which shows that \eqref{Eq-17} holds since $B$ was arbitrary and $X=\cup_{i=1}^mA^m_i$.

\medskip

\noindent Step II. Construct an auxiliary approximation based on continuous dependence on DGMs. Since $\nu_0\in \mathcal{C}_{\mu_X,1}(X,\cM_+(Y))$, let $\widehat{\nu}^{m,n}_{\cdot}$ be the solution to the fixed point equation confined to $\mathcal{C}(\cT,\mathcal{C}_{\mu_X,1}(X,\cM_+(Y)))$
\[\widehat{\nu}^{m,n}_{\cdot}=\cA[\eta^{m,n},\widehat{\nu}^{m,n}_{\cdot},h]\widehat{\nu}^{m,n}_{\cdot}\] with $\widehat{\nu}^{m,n}_0=\nu_0$.
By Proposition~\ref{prop-sol-fixedpoint}(iii), \begin{equation}\label{Eq-18}
\lim_{m\to\infty}\lim_{n\to\infty}d_{\infty}(\nu_t,\widehat{\nu}_t^{m,n})=0,
\end{equation}
since $$\lim_{m\to\infty}\lim_{n\to\infty}d_{\infty}(\eta^{\ell},\eta^{\ell,m,n})=0,\quad \ell=1,\ldots,r.$$

\medskip

\noindent Step III. Construct another auxiliary approximation based on continuous dependence on $h$. Let $\bar{\nu}^{m,n}_{\cdot}$ be the solution to the fixed point equation
\[\bar{\nu}^{m,n}_{\cdot}=\cA[\eta^{m,n},\bar{\nu}^{m,n}_{\cdot},h^m]\bar{\nu}^{m,n}_{\cdot}\] with $\bar{\nu}^{m,n}_0=\nu_0$.
By Lemma~\ref{le-graph} and Lemma~\ref{le-h}, $$\sup_{m,n\in\N}\|\eta^{m,n}\|+\sup_{m\in\N}|h^m|<\infty$$ are uniformly bounded. Hence $\sup_{m,n\in\N}\|\bar{\nu}^{m,n}_{\cdot}\|<\infty$ is also uniformly bounded. Moreover, $d_{\infty}(\widehat{\nu}^{m,n}_{t},\nu_{t})\to0$ also implies $$\sup_{m,n\in\N}\|\widehat{\nu}^{m,n}_{\cdot}\|<\infty.$$

By Proposition~\ref{prop-sol-fixedpoint}(ii), \begin{alignat*}{2}
&d_{\infty}(\bar{\nu}_t^{m,n},\widehat{\nu}_t^{m,n})
\le C\|h-h^m\|_{\infty},
\end{alignat*}
where\[C=\sup_{m,n\in\N}\frac{1}{L_3(\widehat{\nu}^{m,n}_{\cdot},h)}\|\bar{\nu}^{m,n}_{\cdot}\|
    \textnormal{e}^{\BL(h)\textnormal{e}^{L_1({\widehat{\nu}}^{m,n}_{\cdot})T}T}\textnormal{e}^{(L_1({\widehat{\nu}}^{m,n}_{\cdot})
    +L_2(\eta^{m,n})
    \|\bar{\nu}^{m,n}_{\cdot}\|)T}<\infty.\]

\medskip

\noindent Step IV. Since $\nu^{m,n}_{\cdot}$ is the solution to the fixed point equation
\[\nu^{m,n}_{\cdot}=\cA[\eta^{m,n},\nu^{m,n}_{\cdot},h^m]\nu^{m,n}_{\cdot}\] with initial condition $\nu^{m,n}_0$, by Lemma~\ref{le-ini-2} and Proposition~\ref{prop-sol-fixedpoint}(i),
\begin{equation}\label{Eq-20}d_{\infty}(\nu_t^{m,n},\bar{\nu}_t^{m,n})\le \textnormal{e}^{(L_1(\bar{\nu}^{m,n}_{\cdot})+L_2\|\nu^{m,n}_{\cdot}\|)t}
d_{\infty}(\nu_0^{m,n},\nu_0)\to0,\quad \text{as}\ n\to\infty, m\to\infty.\end{equation}

In sum, from \eqref{Eq-18}-\eqref{Eq-20}, by triangle inequality it yields that
\[\lim_{m\to\infty}\lim_{n\to\infty}d_{\infty}(\nu_t,\nu^{m,n}_t)=0.\]
\end{enumerate}
\end{proof}

\section*{Acknowledgement}
Both authors thank G.S. Medvedev's comments on Remark~\ref{re-generalization}. CK acknowledges the TUM International Graduate School of Science and Engineering (IGSSE) for support via the project ``Synchronization in Co-Evolutionary Network Dynamics (SEND)'' and a Lichtenberg Professorship funded by the Volkswagen Foundation. CX acknowledges TUM Foundation Fellowship as well as the Alexander von Humboldt Fellowship funded by Alexander von Humboldt Foundation.
\bibliographystyle{plain}
{\footnotesize\bibliography{references}}
\appendix

\section{Proof of Proposition~\ref{Vlasovf}}\label{appendix-Vlasovf}
\begin{proof}
  Let $E_{\nu_{\cdot}}$ be as in $\mathbf{(A5)}$.  Then \begin{multline*}
    V[\eta,\nu_{\cdot},h](t,x,\phi)=\sum_{i=1}^r\int_X\int_{E_{\nu_{\cdot}}}g_i(t,\psi,\phi)\rd(\nu_t)_y(\psi)
  \rd\eta^i_x(y)+h(t,x,\phi),\\ \quad t\in\cT,\ x\in X,\ \phi\in \R^{r_2}.\end{multline*}

We prove the properties of the Vlasov operator case by case. For $i=1,\ldots,r$, let $g_i=(g_{i,1},\ldots,g_{i,r_2})$.
\begin{enumerate}
\item[(i)] $V[\eta,\nu_{\cdot},h](t,x,\phi)$ is continuous in $t\in\cT$. It suffices to show that for all $i=1,\ldots,r$ and $j=1,\ldots,r_2$, $\int_X\int_{\R^{r_2}}g_{i,j}(t,\psi,\phi)\rd(\nu_t)_y(\psi)\rd\eta^i_x(y)$ is continuous in $t$. Take any sequence $(t_k)_{k}\subseteq\cT$ converging to $t$. By $(\mathbf{A5})$ and Proposition~\ref{prop-nu}(ii), we have $\nu_{\cdot}$ is weakly continuous on $\mathcal{T}$, and hence $\int_{\R^{r_2}}g_{i,j}(t,\psi,\phi)$ $\rd(\nu_t)_y(\psi)$ is continuous in $t\in\cT$, for every $y\in X$. Moreover, by  $(\mathbf{A5})$ and Proposition~\ref{prop-nu}(ii) again,
    \begin{align*}
    \sup_{y\in X}\lt|\int_{\R^{r_2}}g_{i,j}(t_k,\psi,\phi)\rd(\nu_{t_k})_y(\psi)\rt|
    =&\sup_{y\in X}\lt|\int_{E_{\nu_{\cdot}}}g_{i,j}(t_k,\psi,\phi)\rd(\nu_{t_k})_y(\psi)\rt|\\
    \le&\|g_i(\cdot,\phi)\|_{\infty}\sup_{\tau\in\cT}\sup_{y\in X}(\nu_{\tau})_y(\R^{r_2})<\infty,
    \end{align*} where  by $(\mathbf{A1})$ $$\|g_i(\cdot,\phi)\|_{\infty}=\sup_{\tau\in\cT}\sup_{\psi\in E_{\nu_{\cdot}}}|g_i(\tau,\psi,\phi)|<\infty.$$ Hence $$\int_{\R^{r_2}}g_{i,j}(t_k,\psi,\phi)\rd(\nu_t)_y(\psi)$$ is integrable w.r.t. $\eta_x$ since $\|\eta\|<\infty$ by $(\mathbf{A1})$ and Proposition~\ref{prop-fibercomplete}. By Dominated Convergence Theorem, we have \[\lim_{k\to\infty}\int_X\int_{\R^{r_2}}g_{i,j}(t_k,\psi,\phi)\rd(\nu_t)_y(\psi)\rd\eta^i_x(y)
    =\int_X\int_{\R^{r_2}}g_{i,j}(t,\psi,\phi)\rd(\nu_t)_y(\psi)\rd\eta^i_x(y).\]
\item[(ii)]  We will show $V[\eta,\nu_{\cdot},h](t,x,\phi)$ is locally Lipschitz continuous in $\phi$, uniformly in $(t,x)$.

We first show that $$G(t,x,\phi)=\sum_{i=1}^r\int_X\int_{E_{\nu_{\cdot}}}g_i(t,\psi,\phi)\rd(\nu_t)_y(\psi)\rd\eta^i_x(y)$$ is locally Lipschitz in $\phi$, uniformly in $(t,x)$. Since $E_{\nu_{\cdot}}$ is compact, and $g_i$ are locally Lipschitz in $(\psi,\phi)$, uniformly in $t\in\cT$, by finite covering theorem, for a neighborhood $\mathcal{N}\subseteq\R^{r_2}$ of $\phi$, $g_i(t,\psi,\phi)$ are locally Lipschitz in $\phi$ with Lipschitz constant $\mathcal{L}_{\mathcal{N}}(g_i)$, uniformly in $\psi\in E_{\nu_{\cdot}}$ and $t\in\cT$, where
$$\mathcal{L}_{\mathcal{N}}(g_i)=\sup_{t\in\cT}\sup_{x\in X}\sup_{\psi\in E_{\nu_{\cdot}}}\sup_{\tiny\begin{array}{l}
  \phi_1\neq \phi_2,\\
\phi_1,\phi_2\in\mathcal{N}
\end{array}}\frac{|g_i(t,\psi,\phi_1)-g_i(t,\psi,\phi_2)|}{|\phi_1-\phi_2|}<\infty.$$ This shows that for $\phi_1,\phi_2\in\mathcal{N}$, \begin{align*}
  |G(t,x,\phi_1)-G(t,x,\phi_2)|\le&\sum_{i=1}^r\int_X\int_{E_{\nu_{\cdot}}}|g_i(t,\psi,\phi_1)
  -g_i(t,\psi,\phi_2)|
  \rd(\nu_t)_y(\psi)\rd\eta^i_x(y)\\
  \le&\sum_{i=1}^r \mathcal{L}_{\mathcal{N}}(g_i)|\phi_1-\phi_2|(\nu_t)_y(\R^{r_2})\eta^i_x(X)\\
  \le&\|\nu_{\cdot}\|\sum_{i=1}^r \mathcal{L}_{\mathcal{N}}(g_i)\|\eta^i\||\phi_1-\phi_2|,
\end{align*}
where $\|\eta^i\|<\infty$ by Proposition~\ref{prop-fibercomplete}(i) and  $\|\nu_{\cdot}\|<\infty$ by $\mathbf{(A5)}$ and Proposition~\ref{prop-nu}(ii).

Similarly, for $\phi_1,\phi_2\in\mathcal{N}$, \[|h(t,x,\phi_1)-h(t,x,\phi_2)|\le\mathcal{L}_{\mathcal{N}}(h)|\phi_1-\phi_2|,\] where
$$\mathcal{L}_{\mathcal{N}}(h)=\sup_{t\in\cT}\sup_{x\in X}\sup_{\tiny\begin{array}{l}
  \phi_1\neq \phi_2,\\
\phi_1,\phi_2\in\mathcal{N}
\end{array}}\frac{|h(t,x,\phi_1)-h(t,x,\phi_2)|}{|\phi_1-\phi_2|}<\infty.$$
Hence $V[\nu_{\cdot},\eta,h](t,x,\phi)$ is  locally Lipschitz in $\phi$ uniformly $(t,x)$. Altogether it yields
\[\sup_{t\in\cT}\sup_{x\in X}\lt|V[\eta,\nu_{\cdot},h](t,x,\phi_1)-V[\eta,\nu_{\cdot},h](t,x,\phi_2)\rt|\le L_1(\nu_{\cdot})|\phi_1-\phi_2|,\]
where
$$L_1(\nu_{\cdot})\colon=\|\nu_{\cdot}\|\sum_{i=1}^r\mathcal{L}_{\mathcal{N}}(g_i)\|\eta^i\|
+\mathcal{L}_{\mathcal{N}}(h)<\infty.$$
\end{enumerate}
We assume $(\mathbf{A6})$ for the rest.
For all $t\in\cT$ and $x\in X$, we can rewrite $V$ as
\begin{multline*}
  V[\eta,\nu_{\cdot},h](t,x,\phi)=\sum_{i=1}^r\int_X\int_Yg_i(t,\psi,\phi)\rd(\nu_t)_y(\psi)\rd\eta^i_x(y)
+h(t,x,\phi),\\ \quad t\in\cT,\ x\in X,\ \phi\in Y.\end{multline*}

Since $g_i$ and $h$ are local Lipschitz, again by finite covering theorem, when restricted to $Y$, they are globally bounded Lipschitz with $$\mathcal{L}(g_i)=\sup_{t\in\cT}\sup_{\tiny\begin{array}{l}
  (\psi_1,\phi_1)\neq (\psi_2,\phi_2),\\
(\psi_1,\phi_1),(\psi_2,\phi_2)\in Y^2
\end{array}}\frac{|g_i(t,\psi_1,\phi_1)-g_i(t,\psi_2,\phi_2)|}{|\psi_1-\psi_2|+|\phi_1-\phi_2|}<\infty,$$

$$\mathcal{L}(h)=\sup_{t\in\cT}\sup_{x\in X}\sup_{\tiny\begin{array}{l}
  \phi_1\neq \phi_2,\\
\phi_1,\phi_2\in Y
\end{array}}\frac{|h(t,x,\phi_1)-h(t,x,\phi_2)|}{|\phi_1-\phi_2|}<\infty,$$
$$\|g_i\|_{\infty}=\sup_{t\in\cT}\sup_{
(\psi,\phi)\in Y^2}|g_i(t,\psi,\phi)|,\quad \|h\|_{\infty}=\sup_{t\in\R}\sup_{x\in X}\sup_{
\phi\in Y}|h(t,x,\phi)|,$$
$$\BL(g_i)=\|g_i\|_{\infty}+\mathcal{L}(g_i),\quad \BL(h)=\|h\|_{\infty}+\mathcal{L}(h).$$
\begin{enumerate}
\item[(iii)] $V[\eta,\nu_{\cdot},h](t,x,\phi)$ is Lipschitz continuous in $h$, since
\[
  |V[\eta,\nu_{\cdot},h_1](t,x,\phi)-V[\eta,\nu_{\cdot},h_2](t,x,\phi)|=|h_1(t,x,\phi)-h_2(t,x,\phi)|,
\]
which implies
\[\sup_{t\in\cT}\sup_{x\in X}\sup_{\phi\in Y}|V[\eta,\nu_{\cdot},h_1](t,x,\phi)-V[\eta,\nu_{\cdot},h_2](t,x,\phi)|\le\|h_1-h_2\|_{\infty}.\]
\item[(iv)] We will show $V[\eta,\nu_{\cdot},h](t,x,\phi)$ is Lipschitz continuous in $\nu_{\cdot}$:
\begin{equation}\label{L3}
  \sup_{x\in X}\sup_{\phi\in Y}|V[\nu^1_{\cdot}](t,x,\phi)-V[\nu^2_{\cdot}](t,x,\phi)|
  \le L_2d_{\infty}(\nu^1_t,\nu^2_t),
\end{equation} for some positive and finite constant $L_2$.

For all $x\in X$ and $\phi\in Y$,
\begin{align*}
&\lt|V[\nu^1_{\cdot}](t,x,\phi)-V[\nu^2_{\cdot}](t,x,\phi)\rt|\\
=&\lt|\sum_{i=1}^r\int_{X\times Y}g_i(t,\psi,\phi)\mathrm{d}((\nu_t^1)_y(\psi)-(\nu^2_t)_y(\psi))\mathrm{d}\eta^i_x(y)\rt|\\
\le&\sum_{i=1}^r(\BL(g_i)+1)\int_{X}\lt|\int_Y\frac{g_i(t,\psi,\phi)}{\BL(g_i)+1}
\mathrm{d}((\nu^1_t)_y(\psi)-(\nu^2_t)_y(\psi))\rt|
\mathrm{d}\eta^i_x(y)\\
=&\sum_{i=1}^r(\BL(g_i)+1)\int_{X}\sum_{j=1}^{r_2}\lt|\int_Y\frac{g_{i,j}(t,\psi,\phi)}{\BL(g_i)+1}
\mathrm{d}((\nu^1_t)_y(\psi)-(\nu^2_t)_y(\psi))\rt|
\mathrm{d}\eta^i_x(y)\\
\le&r_2\sum_{i=1}^r(\BL(g_i)+1)\int_Xd_{\sf BL}((\nu^1_t)_y,(\nu^2_t)_y)\mathrm{d}\eta^i_x(y)\\
\le& r_2\sum_{i=1}^r(\BL(g_i)+1)\eta^i_x(X)\sup_{y\in X}d_{\sf BL}((\nu^1_t)_y,(\nu^2_t)_y)\le L_3d_{\infty}(\nu^1_t,\nu^2_t),
\end{align*}
where $L_2\colon=r_2\sum_{i=1}^r(\BL(g_i)+1)\|\eta^i\|<\infty$ by $(\mathbf{A1})$. This shows \eqref{L3} holds.
\item[(v)] Assume $\nu_{\cdot}\in \mathcal{C}(\cT,\mathcal{C}_{\mu_X,1}(X,\cM_+(Y)))$. We will show $V[\eta,\nu_{\cdot},h](t,x,\phi)$ is continuous in $\eta$.
Since $\nu_{\cdot}\in \mathcal{C}(\cT,\mathcal{C}_{\mu_X,1}(X,\cM_+(Y)))$, by Proposition~\ref{prop-nu}, for every $t\in\cT$, $(\nu_t)_x$ is weakly continuous in $x$: For every $f\in \mathcal{C}(Y)$, $x\mapsto(\nu_t)_x(f)$ is continuous.

Let $(\eta^{K,i})_K\subseteq \mathcal{B}(X,\cM_+(Y))$ for $i=1,\ldots,r$ such that $$\lim_{K\to\infty}d_{\sf BL}(\eta^{K,i},\eta^i)=0,\quad i=1,\ldots,r.$$ The rest is to show that for every $\xi_{\cdot}\in \mathcal{B}(\cT,\cM_+(Y))$,
\[\lim_{K\to\infty}\int_0^t\int_Y\sup_{x\in X}|V[\eta,\nu_{\cdot},h](\tau,x,\phi)-V[\eta^K,\nu_{\cdot},h](\tau,x,\phi)|\rd\xi_{\tau}(\phi)\rd\tau=0,
\quad \forall t\in\cT.\]

Note that for every given $t\in\cT$,
\begin{align*}
&\lt|V[\eta,\nu_{\cdot},h](t,x,\phi)-V[\eta^K,\nu_{\cdot},h](t,x,\phi)\rt|\\
=&\lt|\sum_{i=1}^r\int_{X}\int_Yg_i(t,\psi,\phi)\mathrm{d}(\nu_t)_y(\psi)
\mathrm{d}\bigl(\eta^i_x(y)-\eta^{K,i}_x(y)\bigr)\rt|\\
=&\lt|\sum_{i=1}^r\sum_{j=1}^{r_2}\int_{X}G_{i,j}(t,y,\phi)\mathrm{d}\bigl(\eta^i_x(y)
-\eta^{K,i}_x(y)\bigr)\rt|,
\end{align*}
where $G_{i,j}(t,y,\phi)\colon=(\nu_t)_y(g_{i,j}(t,\cdot,\phi))$. By $(\mathbf{A4})$ and Proposition~\ref{prop-nu}, we have $\|\nu_{\cdot}\|<\infty$.

By $\nu_{\cdot}\in \mathcal{C}(\cT,\mathcal{C}_{\mu_X,1}(X,\cM_+(Y)))$ and Proposition~\ref{prop-nu}(iii), we have $G_{i,j}(t,y,\phi)$ is continuous in $y$ for each $j=1,\ldots,r_2$.

Next, we construct bounded Lipschitz approximations of $G_{i,j}(t,\cdot,\phi)$.

For every $n\in\N$, let $$G^n_{i,j}(t,y,\phi)=\inf_{z\in X}(G_{i,j}(t,z,\phi)+n|y-z|),\quad y\in X.$$ It is readily verified that $$\BL(G^n_{i,j}(t,\cdot,\phi))\le n+\sup_{z\in X}|G_{i,j}(t,z,\phi)|\le n+\BL(g_i)\|\nu_t\|$$ and $G^n_{i,j}(t,\cdot,\phi)\in\BL(X)$ converges to $G_{i,j}(t,\cdot,\phi)$ uniformly. Hence for $\varepsilon>0$, there exists $N=N(t,\phi)\in\N$ such that for all $n\ge N$,
$$\sup_{y\in X}|G^n_{i,j}(t,y,\phi)-G_{i,j}(t,y,\phi)|<\varepsilon.$$
By Proposition~\ref{prop-nu}(ii), we have $\|\nu_{\cdot}\|<\infty$. Let $$a_n\colon=n+\|\nu_{\cdot}\|\sup_{1\le i\le r}\BL(g_{i})<\infty.$$
Hence\begin{align*}
&\left|V[\eta,\nu_{\cdot},h](t,x,\phi)-V[\eta^K,\nu_{\cdot},h](t,x,\phi)\right|\\
\le&\sum_{i=1}^r\sum_{j=1}^{r_2}\left(\left|\int_XG^N_{i,j}(t,y,\phi)\rd(\eta^i_x(y)-\eta^{K,i}_x(y))\right|
\right.\\
&\left.+\left|\int_{X}(G_{i,j}(t,y,\phi)-G^N_{i,j}(t,y,\phi))
\rd(\eta^i_x(y)-\eta^{K,i}_x(y))\right|\right)\\
\le&\sum_{i=1}^r\sum_{j=1}^{r_2}\left(a_N\sup_{f\in\BL_1(X)}\left|\int_Xf\rd(\eta^i_x-\eta_x^{K,i})\right|
+\varepsilon(\eta^i_x(X)+\eta^{K,i}_x(X))\right)\\
\le&\sum_{i=1}^r\sum_{j=1}^{r_2}\left(a_Nd_{\sf BL}(\eta_x^i,
\eta_x^{K,i})+\varepsilon(\eta_x^i(X)+\eta_x^{K,i}(X))\right)\\
\le&\sum_{i=1}^r\sum_{j=1}^{r_2}\left((\varepsilon+a_N)d_{\sf BL}(\eta_x^i,
\eta_x^{K,i})+2\varepsilon\eta_x^i(X)\right)\\
=& r_2\sum_{i=1}^r\left((\varepsilon+a_N)d_{\sf BL}(\eta_x^i,\eta_x^{K,i})+2\varepsilon \eta_x^i(X)\right),
\end{align*}
which further implies that
\begin{align*}
  &\sup_{x\in X}\lt|V[\eta,\nu_{\cdot},h](t,x,\phi)-V[\eta^K,\nu_{\cdot},h](t,x,\phi)\rt|\\
  \le& r_2(\varepsilon+a_N)\sum_{i=1}^rd_{\infty}(\eta^i,\eta^{K,i})+2\varepsilon r_2\sum_{i=1}^r\|\eta^i\|,
\end{align*}
Since $$\lim_{K\to\infty}\sum_{i=1}^rd_{\infty}(\eta^i,\eta^{K,i})=0,$$ we can further choose $K_0\in\N$ large enough such that for all $K\ge K_0$,
\[r_2(\varepsilon+a_N)\sum_{i=1}^rd_{\infty}(\eta^i,\eta^{K,i})<\varepsilon.\]
This shows $$\lim_{K\to\infty}\sup_{x\in X}|V[\eta,\nu_{\cdot},h](t,x,\phi)-V[\eta^K,\nu_{\cdot},h](t,x,\phi)|=0.$$ For every $\xi_{\cdot}\in \mathcal{B}(\cT,\cM_+(Y))$, it yields from Proposition~\ref{prop-fibercomplete}(iii)\footnote{Here we replace the compact space $X$ by the compact interval $\cT$.} that $$\sup_{t\in\cT}\xi_t(Y)<\infty.$$ This shows that for $K\ge K_0$,
\begin{align*}
  &\int_Y\sup_{x\in X}|V[\eta,\nu_{\cdot},h](\tau,x,\phi)-V[\eta^K,\nu_{\cdot},h](\tau,x,\phi)|\rd\xi_{\tau}(\phi)\\
  \le&\varepsilon (1+2r_2\sum_{i=1}^r\|\eta^i\|)\sup_{t\in\cT}\xi_{t}(Y)<\infty,
\end{align*}
by Dominated Convergence Theorem, we have
\[\lim_{K\to\infty}\int_Y\sup_{x\in X}|V[\eta,\nu_{\cdot},h](\tau,x,\phi)-V[\eta^K,\nu_{\cdot},h](\tau,x,\phi)|\rd\xi_{\tau}(\phi)=0,\quad \forall \tau\in\cT.\]
Since $\cT$ is compact, it follows from the Dominated Convergence again that
\[\lim_{K\to\infty}\int_0^t\int_Y\sup_{x\in X}|V[\eta,\nu_{\cdot},h](\tau,x,\phi)-V[\eta^K,\nu_{\cdot},h](\tau,x,\phi)|\rd\xi_{\tau}(\phi)\rd\tau=0,
\quad \forall t\in\cT.\]
\end{enumerate}
\end{proof}
\section{Proof of Proposition~\ref{prop-continuousdependence}}\label{appendix-prop-continuousdependence}
\begin{proof}
We will suppress the variables in $V[\eta,\nu_{\cdot},h](t,x,\psi)$ and $\mathcal{S}^x_{s,t}[\eta,\nu_{\cdot},h]$ whenever they are clear and not the emphasis from the context. The properties of $\cA$ follows from that of $\mathcal{S}^x_{0,t}[\eta,\nu_{\cdot},h]$. Hence in the following, we will first establish corresponding continuity and Lipschitz continuity for $\mathcal{S}^x_{0,t}[\eta,\nu_{\cdot},h]$ and then apply the results to derive respective properties for $\cA$.
\begin{enumerate}
\item[(i)]
\begin{enumerate}
\item[$\bullet$] Let $\nu_{\cdot}\in \mathcal{C}(\cT,\mathcal{B}_{\mu_X,1}(X,\cM_+(Y)))$. We will show $$t\mapsto\cA[\eta,h]\nu_{t}\in \mathcal{C}(\cT,\mathcal{B}_{\mu_X,1}(X,\cM_+(Y))).$$
For every $x\in X$ and $t\in\cT$, since $Y$ is positively invariant under the flow $\mathcal{S}^x_{t,0}[\eta,\nu_{\cdot},h]$ by Theorem~\ref{theo-well-posedness-characteristic}, we have
\[\mathcal{S}^x_{t,0}[\eta,\nu_{\cdot},h]Y\subseteq Y,\quad Y\subseteq\mathcal{S}^x_{0,t}[\eta,\nu_{\cdot},h]Y,\]
\[\mathcal{S}^x_{0,t}[\eta,\nu_{\cdot},h]A\subseteq\R^{r_2}\setminus Y,\quad \text{for any Borel set}\ A\subset\R^{r_2}\setminus\mathcal{S}^x_{t,0}[\eta,\nu_{\cdot},h]Y.\]
Hence$$(\cA[\eta,h]\nu_{t})_x(A)=(\nu_0)_x(\mathcal{S}^x_{0,t}[\eta,\nu_{\cdot},h]A)=0,$$
 which implies that $$\supp(\cA[\eta,h]\nu_{t})_x\subseteq\mathcal{S}^x_{t,0}[\eta,\nu_{\cdot},h]Y\subseteq Y,$$
 i.e., $\supp(\cA[\eta,h]\nu_{t})_x\in\cM(Y)$.
 Moreover, since $Y\subseteq\mathcal{S}^x_{0,t}[\eta,\nu_{\cdot},h]Y$,
 we have
 \begin{equation}\label{eq-conservation}
 (\cA[\eta,h]\nu_{t})_x(Y)=(\nu_0)_x(\mathcal{S}^x_{0,t}[\eta,\nu_{\cdot},h]Y)=(\nu_0)_x(Y),
 \end{equation}
 i.e., the mass conservation law holds. Since $\nu_t\in \mathcal{B}_{\mu_X,1}(X,\cM_+(Y))$, integrating both sides of \eqref{eq-conservation} with respect to $\mu_X$ on $X$ yields
$\cA[\eta,h]\nu_{t}\in \mathcal{B}_{\mu_X,1}(X,\cM_+(Y))$.

\item[$\bullet$]
Now we show the continuity in $t$. Indeed,
\begin{align*}
  &d_{\infty}(\cA[\eta,\nu_{\cdot},h]\nu_{t},\cA[\eta,\nu_{\cdot},h]\nu_{s})\\
  =&\sup_{x\in X}d_{\sf BL}((\nu_0)_x\circ \cS^x_{0,t}[\eta,\nu_{\cdot},h],(\nu_0)_x\circ \cS^x_{0,s}[\eta,\nu_{\cdot},h])\\
  =&\sup_{x\in X}\sup_{f\in\mathcal{BL}_1(Y)}\lt|\int_Y\lt(f\circ \cS^x_{t,0}[\eta,\nu_{\cdot},h]\phi-f\circ \cS^x_{s,0}[\eta,\nu_{\cdot},h]\phi\rt)\rd(\nu_0)_x(\phi)\rt|\\
  \le&\sup_{x\in X}\int_Y\lt|\cS^x_{t,0}[\eta,\nu_{\cdot},h]\phi-\mathcal{S}^x_{s,0}[\eta,\nu_{\cdot},h]\phi\rt|
  \rd(\nu_0)_x(\phi)\\
  =&\sup_{x\in X}\int_Y
  \lt|\int_s^tV[\eta,\nu_{\cdot},h](x,\tau,\mathcal{S}^x_{\tau,0}[\eta,\nu_{\cdot},h]\phi)\rd\tau\rt|
  \rd(\nu_0)_x(\phi)\\
  \le&\sup_{x\in X}\int_Y\int_s^t\lt(\sum_{i=1}^r\int_X\int_Y\BL(g_i)\rd(\nu_{\tau})_y\rd\eta^i_x(y)+\BL(h)\rt)\rd\tau
  \rd(\nu_0)_x(\phi)\\
  \le&\sup_{x\in X}\int_Y\int_s^t\lt(\sum_{i=1}^r\BL(g_i)\eta^i_x(X)\|\nu_{\cdot}\|+\BL(h)\rt)\rd\tau\rd(\nu_0)_x(\phi)\\
  \le&|s-t|\lt(\sum_{i=1}^r\BL(g_i)\|\eta^i\|\|\nu_{\cdot}\|+\BL(h)\rt)\|\nu_0\|\to0,
\end{align*}
as $|s-t|\to0$. This shows that $$t\mapsto\cA[\eta,h]\nu_{t}\in \mathcal{C}(\cT,\mathcal{B}_{\mu_X,1}(X,\cM_+(Y))).$$
\item[$\bullet$] Finally, we show the continuity in $x$. Assume $\nu_{\cdot}\in \mathcal{C}(\cT,\mathcal{C}_{\mu_X,1}(X,\cM_+(Y)))$. We will show $\cA[\eta,h]\nu_{\cdot}\in \mathcal{C}(\cT,\mathcal{C}_{\mu_X,1}(X,\cM_+(Y)))$. Based on the above properties of $\cA[\eta,h]$, it suffices to show that
the continuity of measures in $x$ is preserved: $x\mapsto(\nu_0)_x\circ \mathcal{S}_{t,0}^x[\eta,\nu_{\cdot},h]$ is continuous.  Indeed,
\begin{align*}
  &d_{\sf BL}((\nu_0)_x\circ \mathcal{S}_{t,0}^x[\eta,\nu_{\cdot},h],(\nu_0)_y\circ \mathcal{S}_{t,0}^y[\eta,\nu_{\cdot},h])\\
  =&\sup_{f\in\BL_1(Y)}\lt|\int_Yf\circ \cS_{t,0}^x[\eta,\nu_{\cdot},h]\phi\rd(\nu_0)_x(\phi)-f\circ \cS_{t,0}^y[\eta,\nu_{\cdot},h]\phi\rd(\nu_0)_y(\phi)\rt|\\
  \le&\int_Y\lt|\cS_{t,0}^x[\eta,\nu_{\cdot},h]\phi-\mathcal{S}^y_{t,0}[\eta,\nu_{\cdot},h]\phi\rt|
  \rd(\nu_0)_x(\phi)\\&+\sup_{f\in\BL_1(Y)}\lt|\int_Yf\circ\cS_{t,0}^y[\eta,\nu_{\cdot},h]\phi
  \rd((\nu_0)_x(\phi)-(\nu_0)_y(\phi))\rt|.
\end{align*}
By Proposition~\ref{Vlasovf}(ii), $\cS_{t,0}^y[\eta,\nu_{\cdot},h]\phi$ is Lipschitz continuous with constant $1+L_1T$. Indeed,
\begin{align*}
  &|\cS_{t,0}^y[\eta,\nu_{\cdot},h]\phi-\cS_{t,0}^y[\eta,\nu_{\cdot},h]\varphi|\\
  \le&|\phi-\varphi|+\int_0^t|V[\eta,\nu_{\cdot},h](t,y,\cS_{\tau,0}^y[\eta,\nu_{\cdot},h]\phi)
  -V[\eta,\nu_{\cdot},h](t,y,\cS_{\tau,0}^y[\eta,\nu_{\cdot},h]\phi)|\rd\tau\\
  \le&|\phi-\varphi|+L_1\int_0^t|\cS_{\tau,0}^y[\eta,\nu_{\cdot},h]\phi
  -\cS_{\tau,0}^y[\eta,\nu_{\cdot},h]\varphi|\rd\tau,
\end{align*}
which implies by Gronwall's inequality that
\[|\cS_{t,0}^y[\eta,\nu_{\cdot},h]\phi-\cS_{t,0}^y[\eta,\nu_{\cdot},h]\varphi|\le \textnormal{e}^{L_1t}|\phi-\varphi|\le \textnormal{e}^{L_1T}|\phi-\varphi|.\]
Hence \[\frac{f\circ\cS_{t,0}^y[\eta,\nu_{\cdot},h]}{1+\textnormal{e}^{L_1T}}\in\BL_1(Y).\]
In addition,
\begin{align*}
 &|V[\eta,\nu_{\cdot},h](\tau,x,\cS_{\tau,0}^x[\eta,\nu_{\cdot},h]\phi)
  -V[\eta,\nu_{\cdot},h](\tau,y,\cS_{\tau,0}^y[\eta,\nu_{\cdot},h]\phi)|\\
  \le&|V[\eta,\nu_{\cdot},h](\tau,x,\cS_{\tau,0}^x[\eta,\nu_{\cdot},h]\phi)
  -V[\eta,\nu_{\cdot},h](\tau,x,\cS_{\tau,0}^y[\eta,\nu_{\cdot},h]\phi)|\\
  &+|V[\eta,\nu_{\cdot},h](\tau,x,\cS_{\tau,0}^y[\eta,\nu_{\cdot},h]\phi)
  -V[\eta,\nu_{\cdot},h](\tau,y,\cS_{\tau,0}^y[\eta,\nu_{\cdot},h]\phi)|\\
 \le&\sum_{i=1}^r\int_X\int_{Y}\Bigl|g_i(\tau,\psi,\cS_{\tau,0}^x[\eta,\nu_{\cdot},h]\phi)
 -g_i(\tau,\psi,\cS_{\tau,0}^y[\eta,\nu_{\cdot},h]\phi)\Bigr|
 \rd(\nu_{\tau})_z(\psi)\rd\eta^i_x(z)\\
 &+|h(\tau,x,\cS_{\tau,0}^x[\eta,\nu_{\cdot},h]\phi)
 -h(\tau,x,\cS_{\tau,0}^y[\eta,\nu_{\cdot},h]\phi)|\\
 &+\sum_{i=1}^r\Bigl|\int_X\int_{Y}g_i(\tau,\psi,\cS_{\tau,0}^y[\eta,\nu_{\cdot},h]\phi))
 \rd(\nu_{\tau})_z(\psi)\rd(\eta^i_x(z)-\eta^i_y(z))\Bigr|\\
 &+|h(\tau,x,\cS_{\tau,0}^y[\eta,\nu_{\cdot},h]\phi)-h(\tau,y,\cS_{\tau,0}^y[\eta,\nu_{\cdot},h]\phi)|\\
 \le&\sum_{i=1}^r\int_X\int_{Y}\cL(g_i)
 |\cS_{\tau,0}^x[\eta,\nu_{\cdot},h]\phi-\cS_{\tau,0}^y[\eta,\nu_{\cdot},h]\phi|
 \rd(\nu_{\tau})_z(\psi)\rd\eta^i_x(z)\\
 &+\cL(h)|\cS_{\tau,0}^x[\eta,\nu_{\cdot},h]\phi-\cS_{\tau,0}^y[\eta,\nu_{\cdot},h]\phi|
 +\sup_{\varphi\in Y}|h(\tau,x,\varphi)-h(\tau,y,\varphi)|\\
 &+\sum_{i=1}^r\Bigl|\int_X\int_{Y}g_i(\tau,\psi,\cS_{\tau,0}^y[\eta,\nu_{\cdot},h]\phi)
 \rd(\nu_{\tau})_z(\psi)\rd(\eta^i_x(z)-\eta^i_y(z))\Bigr|\\
 \le&\left(\cL(h)+\|\nu_{\cdot}\|\sum_{i=1}^r\BL(g_i)\|\eta^i\|\right)
 |\cS_{\tau,0}^x[\eta,\nu_{\cdot},h]\phi-\cS_{\tau,0}^y[\eta,\nu_{\cdot},h]\phi|\\
 &+\sup_{\varphi\in Y}|h(\tau,x,\varphi)-h(\tau,y,\varphi)|\\
 &+\sum_{i=1}^r\Bigl|\int_X\int_{Y}g_i(\tau,\psi,\cS_{\tau,0}^y[\eta,\nu_{\cdot},h]\phi)
 \rd(\nu_{\tau})_z(\psi)\rd(\eta^i_x(z)-\eta^i_y(z))\Bigr|
\end{align*}
This implies that
\begin{align*}
  &|\cS_{t,0}^x[\eta,\nu_{\cdot},h]\phi-\cS_{t,0}^y[\eta,\nu_{\cdot},h]\phi|\\
  =&\int_0^t|V[\eta,\nu_{\cdot},h](\tau,x,\cS_{\tau,0}^x[\eta,\nu_{\cdot},h]\phi)
  -V[\eta,\nu_{\cdot},h](\tau,y,\cS_{\tau,0}^y[\eta,\nu_{\cdot},h]\phi)|\rd\tau\\
  \le&\left(\cL(h)+\|\nu_{\cdot}\|\sum_{i=1}^r\BL(g_i)\|\eta^i\|\right)
 \int_0^t|\cS_{\tau,0}^x[\eta,\nu_{\cdot},h]\phi-\cS_{\tau,0}^y[\eta,\nu_{\cdot},h]\phi|\rd\tau\\
 &+\sup_{\varphi\in Y}\int_0^t|h(\tau,x,\varphi)-h(\tau,y,\varphi)|\rd\tau\\
 &+\sum_{i=1}^r\int_0^t\Bigl|\int_X\int_{Y}g_i(\tau,\psi,\cS_{\tau,0}^y[\eta,\nu_{\cdot},h]\phi)
 \rd(\nu_{\tau})_z(\psi)\rd(\eta^i_x(z)-\eta^i_y(z))\Bigr|\rd\tau.
\end{align*}

By Gronwall's inequality,
\begin{align*}
  &|\cS_{t,0}^x[\eta,\nu_{\cdot},h]\phi-\cS_{t,0}^y[\eta,\nu_{\cdot},h]\phi|\\
  \le&\Bigl(\sup_{\varphi\in Y}\int_0^t|h(\tau,x,\varphi)-h(\tau,y,\varphi)|\rd\tau\\
 &+\sum_{i=1}^r\int_0^t\Bigl|\int_X\int_{Y}g_i(\tau,\psi,\cS_{\tau,0}^y[\eta,\nu_{\cdot},h]\phi)
 \rd(\nu_{\tau})_z(\psi)\rd(\eta^i_x(z)-\eta^i_y(z))\Bigr|\rd\tau\Bigr)
 \textnormal{e}^{C_1t},
\end{align*}
where $C_1=\cL(h)+\|\nu_{\cdot}\|\sum_{i=1}^r\BL(g_i)\|\eta^i\|$. This further shows
\begin{align*}
  &d_{\sf BL}((\nu_0)_x\circ \mathcal{S}_{t,0}^x[\eta,\nu_{\cdot},h],(\nu_0)_y\circ \mathcal{S}_{t,0}^y[\eta,\nu_{\cdot},h])\\
  \le&\int_0^t\int_Y|V[\eta,\nu_{\cdot},h](\tau,x,\cS_{\tau,0}^x[\eta,\nu_{\cdot},h]\phi)
  -V[\eta,\nu_{\cdot},h](\tau,y,\cS_{\tau,0}^y[\eta,\nu_{\cdot},h]\phi)|\rd(\nu_0)_x(\phi)\rd\tau\\
  &+ (1+\textnormal{e}^{L_1T})d_{\sf BL}((\nu_0)_x,(\nu_0)_y)\\
  \le&\left(\cL(h)+\|\nu_{\cdot}\|\sum_{i=1}^r\BL(g_i)\|\eta^i\|\right)
 \int_Y\int_0^t|\cS_{\tau,0}^x[\eta,\nu_{\cdot},h]\phi-\cS_{\tau,0}^y[\eta,\nu_{\cdot},h]\phi|
 \rd(\nu_0)_x(\phi)\rd\tau\\
 &+\|\nu_0\|\int_0^t\sup_{\varphi\in Y}|h(\tau,x,\varphi)-h(\tau,y,\varphi)|\rd\tau\\
 &+\sum_{i=1}^r\int_0^t\int_Y|\int_X\int_{Y}g_i(\tau,\psi,\cS_{\tau,0}^y[\eta,\nu_{\cdot},h]\phi)
 \rd(\nu_{\tau})_z(\psi)\rd(\eta^i_x(z)-\eta^i_y(z))|\rd(\nu_0)_x(\phi)\rd\tau\\
   &+ (1+\textnormal{e}^{L_1T})d_{\sf BL}((\nu_0)_x,(\nu_0)_y)\\
   \le&\left(\cL(h)+\|\nu_{\cdot}\|\sum_{i=1}^r\BL(g_i)\|\eta^i\|\right)
 \int_Y\int_0^t|\cS_{\tau,0}^x[\eta,\nu_{\cdot},h]\phi-\cS_{\tau,0}^y[\eta,\nu_{\cdot},h]\phi|
 \rd(\nu_0)_x(\phi)\rd\tau\\
 &+\sum_{i=1}^r\int_0^t\int_Y\Bigl|\int_X\int_{Y}g_i(\tau,\psi,\cS_{\tau,0}^y[\eta,\nu_{\cdot},h]\phi)
 \rd(\nu_{\tau})_z(\psi)\rd(\eta^i_x(z)-\eta^i_y(z))\Bigr|\rd(\nu_0)_x(\phi)\rd\tau\\
   &+\|\nu_0\|\int_0^t\sup_{\varphi\in Y}|h(\tau,x,\varphi)-h(\tau,y,\varphi)|\rd\tau+ (1+\textnormal{e}^{L_1T})d_{\sf BL}((\nu_0)_x,(\nu_0)_y)\\
    \le&\|\nu_0\|\left(\cL(h)+\|\nu_{\cdot}\|\sum_{i=1}^r\BL(g_i)\|\eta^i\|\right)\int_0^t
 \textnormal{e}^{C_1\tau}\rd\tau\Bigl(\int_0^t\sup_{\varphi\in Y}|h(\tau,x,\varphi)-h(\tau,y,\varphi)|\rd\tau\\
 &+\sum_{i=1}^r\int_0^t\Bigl|\int_X\int_{Y}g_i(\tau,\psi,\cS_{\tau,0}^y[\eta,\nu_{\cdot},h]\phi)
 \rd(\nu_{\tau})_z(\psi)\rd(\eta^i_x(z)-\eta^i_y(z))\Bigr|\rd\tau\Bigr)\\
 &+\sum_{i=1}^r\int_0^t\int_Y\Bigl|\int_X\int_{Y}g_i(\tau,\psi, \cS_{\tau,0}^y[\eta,\nu_{\cdot},h]\phi)
 \rd(\nu_{\tau})_z(\psi)\rd(\eta^i_x(z)-\eta^i_y(z))\Bigr|\rd(\nu_0)_x(\phi)\rd\tau\\
   &+\|\nu_0\|\int_0^t\sup_{\varphi\in Y}|h(\tau,x,\varphi)-h(\tau,y,\varphi)|\rd\tau+ (1+\textnormal{e}^{L_1T})d_{\sf BL}((\nu_0)_x,(\nu_0)_y).
\end{align*}
Since $\nu_{\cdot}\in \mathcal{C}(\cT,\mathcal{C}_{\mu_X,1}(X,\cM_+(Y))$, by Proposition~\ref{prop-nu}(iii), $(\nu_t)_x$ is weakly continuous in $x$. Since
By $(\mathbf{A2})$, $g_i$ is bounded Lipschitz, and $\cS_{\tau,0}^y[\eta,\nu_{\cdot},h]\phi$ is continuous in $\phi$, we have $g_i(\tau,\psi,\cS_{\tau,0}^y[\eta,\nu_{\cdot},h]\phi)$ is bounded continuous in $\phi$. Hence\\ \noindent $\int_{Y}g_i(\tau,\psi,\cS_{\tau,0}^y[\eta,\nu_{\cdot},h]\phi)
 \rd(\nu_{\tau})_z(\psi)$ is continuous in $z$. Moreover $$|\int_{Y}g_i(\tau,\psi,\cS_{\tau,0}^y[\eta,\nu_{\cdot},h]\phi)
 \rd(\nu_{\tau})_z(\psi)|\le\BL(g_i)\|\nu_{\cdot}\|<\infty$$ is also bounded on $X$, since $\eta^i\in \mathcal{C}(X,\cM_+(X))$, using Proposition~\ref{prop-nu}(iii) again, we know
\[\lim_{|x-y|\to0}\Bigl|\int_X\int_{Y}g_i(\tau,\psi,\cS_{\tau,0}^y[\eta,\nu_{\cdot},h]\phi)
 \rd(\nu_{\tau})_z(\psi)\rd(\eta^i_x(z)-\eta^i_y(z))\Bigr|=0.\]
 Since $$\int_0^t|\int_X\int_{Y}g_i(\tau,\psi,\cS_{\tau,0}^y[\eta,\nu_{\cdot},h]\phi)
 \rd(\nu_{\tau})_z(\psi)\rd(\eta^i_x(z)|\rd\tau\le\BL(g_i)\|\nu_{\cdot}\|\|\eta\|T<\infty,$$
 By Dominated Convergence Theorem,
 $$\lim_{|x-y|\to0}\int_0^t\Bigl|\int_X\int_{Y}g_i(\tau,\psi,\cS_{\tau,0}^y[\eta,\nu_{\cdot},h]\phi)
 \rd(\nu_{\tau})_z(\psi)\rd(\eta^i_x(z)-\eta^i_y(z))\Bigr|\rd\tau=0.$$
Similarly, $$\lim_{|x-y|\to0}\int_0^t\int_Y\Bigl|\int_X\int_{Y}g_i(\tau,\psi, \cS_{\tau,0}^y[\eta,\nu_{\cdot},h]\phi)
 \rd(\nu_{\tau})_z(\psi)\rd(\eta^i_x(z)-\eta^i_y(z))\Bigr|\rd(\nu_0)_x(\phi)\rd\tau=0.$$ Moreover, by
 $(\mathbf{A8})$ as well as the Dominated Convergence Theorem,
 \[\lim_{|x-y|\to0}\int_0^t\sup_{\varphi\in Y}|h(\tau,x,\varphi)-h(\tau,y,\varphi)|\rd\tau=0.\] Since
$\nu_0\in\mathcal{C}_{\mu_X,1}(X,\cM_+(Y))$, $$\lim_{|x-y|\to0}d_{\sf BL}((\nu_0)_x,(\nu_0)_y)=0.$$ All these limits together yield
\[\lim_{|x-y|\to0}d_{\sf BL}((\nu_0)_x\circ \mathcal{S}_{t,0}^x[\eta,\nu_{\cdot},h],(\nu_0)_y\circ \mathcal{S}_{t,0}^y[\eta,\nu_{\cdot},h])=0.\]
\end{enumerate}
\item[(ii)] To show this Lipschitz continuity, we first need to show that $\cS^x_{s,t}[\nu_{\cdot}]\phi(x)$ is Lipschitz continuous in $\phi(x)$. Note that
\begin{align*}
    &|\cS_{t,0}^x\phi_1(x)-\cS_{t,0}^x\phi_2(x)|\\
    \le&|\phi_1(x)-\phi_2(x)|
    +\int_0^t|V(\tau,x,\cS_{\tau,0}^x\phi_1(x))-V(\tau,x,\cS_{\tau,0}^x\phi_2(x))|\rd\tau\\
    \le&|\phi_1(x)-\phi_2(x)|+L_1(\nu_{\cdot})\int_0^t|\cS_{\tau,0}^x\phi_1(x)
    -\cS_{0,\tau}^x\phi_2(x)|\rd\tau.
    \end{align*}
    By Gronwall's inequality,
    \begin{equation}\label{Lip-phi}|\cS_{t,0}^x\phi_1(x)-\cS_{t,0}^x\phi_2(x)|\le \textnormal{e}^{L_1t}|\phi_1(x)-\phi_2(x)|.\end{equation}
    Similarly, one can show \begin{equation*}
    |\cS_{0,t}^x\phi_1(x)-\cS_{0,t}^x\phi_2(x)|\le \textnormal{e}^{L_1t}|\phi_1(x)-\phi_2(x)|.
    \end{equation*}
 Next, we show $\cS^x_{s,t}[\nu_{\cdot}]$ is Lipschitz continuous in $\nu_{\cdot}$. Observe that
\begin{align}
\nonumber&d_{\sf BL}((\nu_0)_x^1\circ \cS_{0,t}^x[\nu^1_{\cdot}],(\nu_0)_x^2\circ \cS^x_{0,t}[\nu^2_{\cdot}])\\
\label{sum}\le&d_{\sf BL}((\nu_0)_x^1\circ \cS_{0,t}^x[\nu^1_{\cdot}],(\nu_0)_x^1\circ \cS^x_{0,t}[\nu^2_{\cdot}])+d_{\sf BL}((\nu_0)_x^1\circ \cS^x_{0,t}[\nu^2_{\cdot}],(\nu_0)_x^2\circ \cS^x_{0,t}[\nu^2_{\cdot}]).
\end{align}
We now estimate the first term.
\begin{align*}
&d_{\sf BL}((\nu_0^1)_x\circ \cS^x_{0,t}[\nu^1_{\cdot}],(\nu_0^1)_x\circ \cS^x_{0,t}[\nu^2_{\cdot}])\\
=&\sup_{f\in\mathcal{BL}_1(Y)}\int_Y f(\phi)\mathrm{d}(((\nu_0^1)_x\circ \cS^x_{0,t}[\nu^1_{\cdot}])(\phi)-((\nu_0^1)_x\circ \cS^x_{0,t}[\nu^2_{\cdot}])(\phi))\\
=&\sup_{f\in\mathcal{BL}_1(Y)}\int_Y ((f\circ \cS^x_{t,0}[\nu^1_{\cdot}])(\phi)-(f\circ \cS^x_{t,0}[\nu^2_{\cdot}])(\phi))\mathrm{d}(\nu_0^1)_x(\phi)\\
\le&\int_Y\lt|\cS^x_{t,0}[\nu^1_{\cdot}](\phi)- \cS^x_{t,0}[\nu^2_{\cdot}](\phi)\rt|\mathrm{d}(\nu_0^1)_x(\phi)\colon=\lambda_x(t)\\
=&\int_Y\lt|\int_0^t(V[\nu^1_{\cdot}](\tau,x,\cS^x_{\tau,0}[\nu^1_{\cdot}]\phi)
  -V[\nu^2_{\cdot}](\tau,x,\cS^x_{\tau,0}[\nu^2_{\cdot}]\phi))
  \mathrm{d}\tau\rt|\mathrm{d}(\nu_0^1)_x(\phi)\\
  \le&\int_Y\lt|\int_0^t(V[\nu^1_{\cdot}](\tau,x,\cS^x_{\tau,0}[\nu^1_{\cdot}]\phi)
  -V[\nu^2_{\cdot}](\tau,x,\cS^x_{\tau,0}[\nu^1_{\cdot}]\phi))
  \mathrm{d}\tau\rt|\mathrm{d}(\nu_0^1)_x(\phi)\\
  &+\int_Y\lt|\int_0^t(V[\nu^2_{\cdot}](\tau,x,\cS^x_{\tau,0}[\nu^1_{\cdot}]\phi)
  -V[\nu^2_{\cdot}](\tau,x,\cS^x_{\tau,0}[\nu^2_{\cdot}]\phi))
  \mathrm{d}\tau\rt|\mathrm{d}(\nu_0^1)_x(\phi)\\
  \le&\int_Y\int_0^t\lt|V[\nu^1_{\cdot}](\tau,x,\cS^x_{\tau,0}[\nu^1_{\cdot}]\phi)
  -V[\nu^2_{\cdot}](\tau,x,\cS^x_{\tau,0}[\nu^1_{\cdot}]\phi)\rt|
  \mathrm{d}\tau\mathrm{d}(\nu_0^1)_x(\phi)\\
  &+\int_Y\int_0^t\lt|V[\nu^2_{\cdot}](\tau,x,\cS^x_{\tau,0}[\nu^1_{\cdot}]\phi)
  -V[\nu^2_{\cdot}](\tau,x,\cS^x_{\tau,0}[\nu^2_{\cdot}]\phi)\rt|
  \mathrm{d}\tau\mathrm{d}(\nu_0^1)_x(\phi)\\
  =&\int_0^t\int_Y\lt|V[\nu^1_{\cdot}](\tau,x,\cS^x_{\tau,0}[\nu^1_{\cdot}]\phi)
  -V[\nu^2_{\cdot}](\tau,x,\cS^x_{\tau,0}[\nu^1_{\cdot}]\phi)\rt|
  \mathrm{d}(\nu_0^1)_x(\phi)\mathrm{d}\tau\\
  &+\int_0^t\int_Y\lt|V[\nu^2_{\cdot}](\tau,x,\cS^x_{\tau,0}[\nu^1_{\cdot}]\phi)
  -V[\nu^2_{\cdot}](\tau,x,\cS^x_{\tau,0}[\nu^2_{\cdot}]\phi)\rt|
  \mathrm{d}(\nu_0^1)_x(\phi)\mathrm{d}\tau\\
  \le&\int_0^t\int_Y\lt|V[\nu^1_{\cdot}](\tau,x,\psi)-V[\nu^2_{\cdot}](\tau,x,\psi)\rt|
  \mathrm{d}(\nu_{\tau}^1)_x(\psi)\mathrm{d}\tau\\
  &+\int_0^t\int_Y\lt|V[\nu^2_{\cdot}](\tau,x,\cS^x_{\tau,0}[\nu^1_{\cdot}]\phi)
  -V[\nu^2_{\cdot}](\tau,x,\cS^x_{\tau,0}[\nu^2_{\cdot}]\phi)\rt|
  \mathrm{d}(\nu_0^1)_x(\phi)\mathrm{d}\tau,
\end{align*}
which implies by Proposition~\ref{Vlasovf} that
\begin{align*}
\lambda_x(t)\le& L_2\int_0^td_{\infty}(\nu^1_{\tau},\nu^2_{\tau})(\nu_{\tau}^1)_x(Y)
\mathrm{d}\tau\\&+L_1(\nu^2_{\cdot})\int_0^t\int_Y|\cS^x_{\tau,0}[\nu^1_{\cdot}]\phi
-\cS^x_{\tau,0}[\nu^2_{\cdot}]\phi|
\mathrm{d}(\nu_{0}^1)_x(\phi)\mathrm{d}\tau\\
\le&L_2\|\nu^1_{\cdot}\|\int_0^td_{\infty}(\nu^1_{\tau},\nu^2_{\tau})\rd\tau
+L_1(\nu^2_{\cdot})\int_0^t\lambda_x(\tau)\rd\tau.
\end{align*}
Applying Gronwall's inequality (see e.g., \cite[Lemma~2.5]{KM18}), we have
\begin{align}
\label{term-1}
\lambda_x(t)\le& L_2\|\nu^1_{\cdot}\|\textnormal{e}^{L_1(\nu^2_{\cdot})t}
\int_0^td_{\infty}(\nu^1_{\tau},\nu^2_{\tau})\textnormal{e}^{-L_1(\nu^2_{\cdot})\tau}\rd\tau.
\end{align}

Next, we estimate the second term. For $f\in\mathcal{BL}_1(Y)$, from \eqref{Lip-phi} it follows that
\[\mathcal{L}(f\circ \cS^x_{t,0}[\nu^2_{\cdot}])\le \mathcal{L}(f)\mathcal{L}(\cS^x_{t,0}[\nu^2_{\cdot}])
\le\mathcal{L}(f)\textnormal{e}^{L_1(\nu^2_{\cdot})t},\quad \|f\circ \cS^x_{t,0}[\nu^2_{\cdot}]\|_{\infty}\le\|f\|_{\infty}.\]
Hence $\BL(f\circ \cS^x_{t,0})\le \textnormal{e}^{L_1(\nu^2_{\cdot})t}$. For every $x\in X$,
\begin{alignat}{2}
 \nonumber &d_{\sf BL}((\nu_0^1)_x\circ \cS^x_{0,t}[\nu^2_{\cdot}],(\nu_0)_x^2\circ \cS^x_{0,t}[\nu^2_{\cdot}])\\
  \nonumber=&\sup_{f\in\mathcal{BL}_1(Y)}\int_Y(f\circ \cS^x_{t,0}[\nu^2_{\cdot}])(\phi)\rd((\nu_0^1)_x(\phi)-(\nu_0^2)_x(\phi))\\
 \label{term-2} \le&\textnormal{e}^{L_1(\nu^2_{\cdot})t}d_{\sf BL}((\nu_0)_x^1,(\nu_0)_x^2)\le \textnormal{e}^{L_1(\nu^2_{\cdot})t}d_{\infty}(\nu_0^1,\nu_0^2).
\end{alignat}
Combining \eqref{term-1} and \eqref{term-2}, it follows from \eqref{sum} that
\begin{align}
\nonumber&d_{\infty}(\cA[\eta,h]\nu_{t}^1,\cA[\eta,h]\nu_{t}^2)\\
\nonumber =&\sup_{x\in X}d_{\sf BL}((\nu_0)_x^1\circ \cS^x_{0,t}[\nu^1_{\cdot}],(\nu_0)_x^2\circ \cS^x_{0,t}[\nu^2_{\cdot}])\\
 \nonumber\le&\textnormal{e}^{L_1(\nu^2_{\cdot})t}d_{\infty}(\nu_0^1,\nu_0^2)+ L_2\|\nu^1_{\cdot}\|\textnormal{e}^{L_1(\nu^2_{\cdot})t}\int_0^td_{\infty}(\nu^1_{\tau},\nu^2_{\tau})
  \textnormal{e}^{-L_1(\nu^2_{\cdot})\tau}\rd\tau.
\end{align}

\item[(iii)] Lipschitz continuity of $\cA[\eta,h]$ in $h$.

We first need to establish the Lipschitz continuity for $\cS^x_{s,t}[h]$. It follows from Proposition~\ref{Vlasovf} that
\begin{align*}
  &|\cS^x_{0,t}[h_1]\phi-\cS^x_{0,t}[h_2]\phi|\\
  \le&\lt|\int_0^t\lt(V[h_1](\tau,x,\cS^x_{0,\tau}[h_1]\phi)
  -V[h_2](\tau,x,\cS^x_{0,\tau}[h_2]\phi)\rt)\rd\tau\rt|\\
  \le&\int_0^t\lt|V[h_1](\tau,x,\cS^x_{0,\tau}[h_1]\phi)
  -V[h_2](\tau,x,\cS^x_{0,\tau}[h_1]\phi(x))\rt|\rd\tau\\
  &+\int_0^t\lt|V[h_2](\tau,x,\cS^x_{0,\tau}[h_1]\phi)
  -V[h_2](\tau,x,\cS^x_{0,\tau}[h_2]\phi)\rt|\rd\tau\\
  \le&\int_0^t\lt|h_1(\tau,x,\cS^x_{0,\tau}[h_1]\phi(x))
  -h_2(\tau,x,\cS^x_{0,\tau}[h_1]\phi(x))\rt|\rd\tau\\
  &+L_1(\nu_{\cdot})\int_0^t\lt|\cS^x_{0,\tau}[h_1]\phi-\cS^x_{0,\tau}[h_2]\phi\rt|\rd\tau.
\end{align*}
By Gronwall's inequality, we have
\begin{align*}
\lt|\cS^x_{0,t}[h_1]\phi-\cS^x_{0,t}[h_2]\phi\rt|
\le&\textnormal{e}^{L_1t}\int_0^t|h_1(\tau,x,\cS^x_{0,\tau}[h_1]\phi)
-h_2(\tau,x,\cS^x_{0,\tau}[h_1]\phi)|\rd\tau.
\end{align*}
Note that   \begin{align*}
  &d_{\sf BL}((\nu_0)_x\circ \cS^x_{0,t}[h_1],(\nu_0)_x\circ \cS^x_{0,t}[h_2])\\
  =&\sup_{f\in\mathcal{BL}_1(Y)}\int_Yf(\phi)\rd((\nu_0)_x\circ \cS^x_{0,t}[h_1]-(\nu_0)_x\circ \cS^x_{0,t}[h^2])\\
  =&\sup_{f\in\mathcal{BL}_1(Y)}\int_Y\lt((f\circ \cS^x_{t,0}[h_1])(\phi)
  -(f\circ \cS^x_{t,0}[h_2])(\phi)\rt)\rd(\nu_0)_x(\phi)\\
  \le&\int_Y\lt|\cS^x_{t,0}[h_1]\phi
  -\cS^x_{t,0}[h_2]\phi)\rt|\rd(\nu_0)_x(\phi)=\colon\alpha_x(t)\\
  \le& \textnormal{e}^{L_1t}\int_0^t\int_Y|h_1(\tau,x,\cS^x_{0,\tau}[h_1]\phi)
  -h_2(\tau,x,\cS^x_{0,\tau}[h_2]\phi)|\rd(\nu_0)_x(\phi)\rd\tau\\
  \le& \textnormal{e}^{L_1t}\int_0^t\int_Y|h_1(\tau,x,\cS^x_{0,\tau}[h_1]\phi)
  -h_2(\tau,x,\cS^x_{0,\tau}[h_1]\phi)|\rd(\nu_0)_x(\phi)\rd\tau\\
  &+\textnormal{e}^{L_1t}\int_0^t\int_Y|h_2(\tau,x,\cS^x_{\tau,0}[h_1]\phi)
  -h_2(\tau,x,\cS^x_{\tau,0}[h_2]\phi)|\rd(\nu_0)_x(\phi)\rd\tau\\
  \le&\textnormal{e}^{L_1t}\int_0^t\int_Y|h_1(\tau,x,\phi)
  -h_2(\tau,x,\phi)|\rd(\nu_{\tau})_x(\phi)\rd\tau\\
  &+\BL(h_2)\textnormal{e}^{L_1t}\int_0^t\int_Y|\cS^x_{\tau,0}[h_1]\phi
  -\cS^x_{\tau,0}[h_2]\phi)|\rd(\nu_0)_x(\phi)\rd\tau,
\end{align*}
which implies that
\begin{align*}
\textnormal{e}^{-L_1t}\alpha_x(t)
  \le&\int_0^t\int_Y|h_1(\tau,x,\psi)-h_2(\tau,x,\psi)|\rd(\nu_{\tau})_x(\psi)\rd\tau
  +C_2\int_0^t\textnormal{e}^{-L_1\tau}\alpha_x(\tau)\rd\tau,\end{align*}
  where $C_2\colon=\BL(h_2)\textnormal{e}^{L_1T}$.
Applying Gronwall's inequality yields that
\begin{alignat}{2}
\nonumber\alpha_x(t)\le& \textnormal{e}^{(L_1+C_2)t}
\int_0^t\textnormal{e}^{-C_2\tau}\int_Y|h_1(\tau,x,\phi)-h_2(\tau,x,\phi)|\rd(\nu_{\tau})_x(\phi)\rd\tau\\
\label{h-estimate}\le& \|\nu_{\cdot}\|\textnormal{e}^{L_3t}\|h_1-h_2\|_{\infty}\int_0^t1\rd\tau,
\end{alignat}where $L_3=L_3(\nu_{\cdot},h_2)\colon=L_1(\nu_{\cdot})+C_2$, which further implies that
\begin{align*}
&d_{\infty}(\mathcal{A}[\eta,h_1](\nu_{t}),\mathcal{A}[\eta,h_2](\nu_{t}))\\
=&\sup_{x\in X}d_{\sf BL}((\nu_0)_x\circ \cS^x_{0,t}[h_1],(\nu_0)_x\circ \cS^x_{0,t}[h_2])\le T\|\nu_{\cdot}\|\textnormal{e}^{L_3t}\|h_1-h_2\|_{\infty}.
\end{align*}
\item[(iv)] Absolute continuity.
    Assume $\nu_0\in \mathcal{B}_{\mu_X,1}(X,\cM_{+,\abs}(Y))\hookrightarrow \mathcal{B}_{\mu_X,1}(X,\cM_{+,\abs}(\R^{r_2}))$.
For every $x\in X$, let $$\rho_0(x,\phi)=\frac{\rd(\nu_0)_x(\phi)}{\rd\phi},\quad \mathfrak{m}\text{-a.e.}\ \phi\in \R^{r_2},$$ be the Radon Nikodym derivative of $(\nu_0)_x$. Since $\cS^x_{0,t}[\eta,\nu_{\cdot},h]|_{\cS^x_{t,0}[\eta,\nu_{\cdot},h]Y}$ is Lipschitz continuous, by \cite[Thm.~3.1]{EG15}, one can extend $\cS^x_{0,t}[\eta,\nu_{\cdot},h]|_{\cS^x_{t,0}[\eta,\nu_{\cdot},h]Y}$ to a Lipschitz continuous function
$\widetilde{\cS}_{0,t}^x[\eta,\nu_{\cdot},h]$ from $\R^{r_2}$ to $\R^{r_2}$ such that
\[\widetilde{\cS}_{0,t}^x[\eta,\nu_{\cdot},h]=\cS^x_{0,t}[\eta,\nu_{\cdot},h]\quad \text{on}\ \cS^x_{t,0}[\eta,\nu_{\cdot},h]Y,\]
\[\mathcal{L}(\widetilde{\cS}_{0,t}^x[\eta,\nu_{\cdot},h])
\le\sqrt{r_2}\mathcal{L}(\cS^x_{0,t}[\eta,\nu_{\cdot},h]|_{\cS^x_{t,0}[\eta,\nu_{\cdot},h]Y}).\]
By Rademacher's Differentiability Theorem \cite[Thm.~3.2]{EG15} (see also \cite[Thm.~3.4.3]{CMN19}), we have $\widetilde{\cS}^x_{0,t}[\eta,\nu_{\cdot},h]$ is differentiable $\mathfrak{m}$-a.e. in $\R^{r_2}$. Using the change of variables formula for Lipschitz continuous maps in $\R^{r_2}$ \cite[Thm.~3.9]{EG15}, for any Lebesgue integrable function $f$ on $\R^{r_2}$,
\begin{equation}\label{eq-identity-f}
\int_{\R^{r_2}}f(\phi)\det\lt(\frac{\partial}{\partial\phi}
\widetilde{\cS}^x_{0,t}[\eta,\nu_{\cdot},h]\phi\rt)\rd\phi=\int_{\R^{r_2}}\sum_{\phi\in (\widetilde{\cS}^x_{0,t}[\eta,\nu_{\cdot},h])^{-1}(\psi)}f(\phi)\rd\psi,\end{equation}
where $\det(\frac{\partial}{\partial\phi}
\widetilde{\cS}^x_{0,t}[\eta,\nu_{\cdot},h]\phi)$ is the determinant of the Jacobian of $\widetilde{\cS}^x_{0,t}[\eta,\nu_{\cdot},h]$ for $\mathfrak{m}$-a.e. $\phi\in \R^{r_2}$. Let $g\in \mathcal{C}_{\sf b}(\R^{r_2})$ such that $\supp g\subseteq \cS^x_{t,0}[\eta,\nu_{\cdot},h]Y$. Then It is easy to show that $$f(\phi)=g(\phi)\rho_0(x,\phi),\quad \phi\in\R^{r_2},$$ is integrable on $\R^{r_2}$. For all $\psi\in\R^{r_2}\setminus Y$, we have $$(\widetilde{\cS}^x_{0,t}[\eta,\nu_{\cdot},h])^{-1}(\psi)\subset\R^{r_2}\setminus \cS^x_{t,0}[\eta,\nu_{\cdot},h]Y.$$ Since $$(\cS^x_{0,t}[\eta,\nu_{\cdot},h])^{-1}=\cS^x_{t,0}[\eta,\nu_{\cdot},h]\quad \text{on}\ Y,$$
$$(\cS^x_{t,0}[\eta,\nu_{\cdot},h])^{-1}=\cS^x_{0,t}[\eta,\nu_{\cdot},h]\quad \text{on}\ \cS^x_{t,0}[\eta,\nu_{\cdot},h]Y,$$
substituting this function $f$ into \eqref{eq-identity-f} yields that
\begin{align}
\nonumber&\int_{\cS^x_{t,0}[\eta,\nu_{\cdot},h]Y}g(\phi)\rho_0(x,\phi)\det\left(\frac{\partial}{\partial\phi}
\cS^x_{0,t}[\eta,\nu_{\cdot},h]\phi\right)\rd\phi\\
\nonumber=&\int_Y\sum_{\phi\in (\widetilde{\cS}^x_{0,t}[\eta,\nu_{\cdot},h])^{-1}(\psi)\cap\cS^x_{t,0}[\eta,\nu_{\cdot},h]Y}
g(\phi)\rho_0(x,\phi)\rd\psi,\\
\nonumber=&\int_Y\sum_{\phi=\cS^x_{t,0}[\eta,\nu_{\cdot},h](\psi)}g(\phi)\rho_0(x,\phi)\rd\psi\\
\nonumber=&\int_Yg(\cS^x_{t,0}[\eta,\nu_{\cdot},h](\psi))
\rho_0(x,\cS^x_{t,0}[\eta,\nu_{\cdot},h](\psi))\rd\psi\\
\nonumber=&\int_{\cS^x_{t,0}[\eta,\nu_{\cdot},h]Y}
g(\phi)\rho_0(x,\phi)\rd\cS^x_{0,t}[\eta,\nu_{\cdot},h](\phi).
\end{align}
This shows \begin{equation}\label{weak-identity}
\rho_0(x,\phi)\rd\cS^x_{0,t}[\eta,\nu_{\cdot},h](\phi)
=\rho_0(x,\phi)\det\left(\frac{\partial}{\partial\phi}
\cS^x_{0,t}[\eta,\nu_{\cdot},h]\phi\right)\rd\phi,\quad \mathfrak{m}\text{-a.e.}\ \phi\in\cS^x_{t,0}[\eta,\nu_{\cdot},h]Y.
\end{equation}
Since $$\rd(\nu_0)_x(\phi)=\rho_0(x,\phi)\rd\phi,\quad (\cA[\eta,h]\nu_{t})_x=(\nu_0)_x\circ\cS^x_{0,t}[\eta,\nu_{\cdot},h],$$
we have \begin{align*}&\rd(\cA[\eta,h]\nu_{t})_x(\phi)=\rd(\nu_0)_x(\cS^x_{0,t}[\eta,\nu_{\cdot},h]\phi)\\
=&\rho_0(x,\cS^x_{0,t}[\eta,\nu_{\cdot},h]\phi)\rd\cS^x_{0,t}[\eta,\nu_{\cdot},h]\phi,\quad \text{for}\ \phi\in\cS^x_{t,0}[\eta,\nu_{\cdot},h]Y,\end{align*}
which implies from \eqref{weak-identity} that
$$\rd(\cA[\eta,h]\nu_{t})_x(\phi)
=\rho_0(x,\cS^x_{0,t}[\eta,\nu_{\cdot},h]\phi)\det\left(\frac{\partial}{\partial\phi}
\cS^x_{0,t}[\eta,\nu_{\cdot},h]\phi\right)\rd\phi,\quad \mathfrak{m}\text{-a.e.}\ \phi\in\cS^x_{t,0}[\eta,\nu_{\cdot},h]Y,$$
i.e.,
$(\cA[\eta,h]\nu_{t})_x\in\cM_{+,\abs}(Y),\quad x\in X$ with
$$\frac{\rd(\cA[\eta,h]\nu_{t})_x(\phi)}{\rd\phi}
=\rho_0(x,\cS^x_{0,t}[\eta,\nu_{\cdot},h]\phi)\det\left(\frac{\partial}{\partial\phi}
\cS^x_{0,t}[\eta,\nu_{\cdot},h]\phi\right),\quad \mathfrak{m}\text{-a.e.}\ \phi\in Y.$$
\end{enumerate}
\end{proof}
\section{Proof of Proposition~\ref{prop-sol-fixedpoint}}\label{appendix-prop-sol-fixedpoint}
\begin{proof}
We will again suppress the variables in $V[\eta,\nu_{\cdot},h](t,x,\psi)$ and $\mathcal{S}^x_{s,t}[\eta,\nu_{\cdot},h]$ whenever they are clear and not the emphasis from the context.

We first use the Banach contraction principle to show that the solution $$\nu_{\cdot}\in\mathcal{C}(\mathcal{T},\mathcal{B}_{\mu_X,1}(X,\cM_+(Y)))$$ to \eqref{Fixed} exists uniquely. Then by Proposition~\ref{prop-continuousdependence}(vi), we have $$\nu_t\in \mathcal{B}_{\mu_X,1}(X,\cM_{+,\abs}(Y)),\quad \forall t\in\cT,$$ provided $\nu_0\in\mathcal{B}_{\mu_X,1}(X,\cM_{+,\abs}(Y))$.

It remains to show that $\mathcal{A}[\eta,h]$ is a contraction mapping from $(\mathcal{C}(\mathcal{T},\mathcal{B}_{\mu_X,1}(X,\cM_+(Y))),$ $d_{\alpha})$ to itself.

From the proof of Proposition~\ref{prop-continuousdependence}(i), we have $$\mathcal{A}[\eta,h]\colon (\mathcal{C}(\mathcal{T},\mathcal{B}_{\mu_X,1}(X,\cM_+(Y))),d_{\alpha})\to (\mathcal{C}(\mathcal{T},\mathcal{B}_{\mu_X,1}(X,\cM_+(Y))),d_{\alpha}).$$
It then suffices to show that  $\cA[\eta,h]$ is a contraction.

By the mass conservation law given in Proposition~\ref{prop-continuousdependence}(i), let $M=\|\nu^1_0\|$ and $\mathcal{B}^M_{\mu_X,1}(X,$ $\cM_+(X))=\{\mu\in\mathcal{B}_{\mu_X,1}(X,\cM_+(X)))\colon \|\mu\|\le M\}$. For $\nu^2_{\cdot}\in \mathcal{B}^M_{\mu_X,1}(X,\cM_+(X))$, let $\alpha\ge L_1(\nu^2_{\cdot})+L_2M+1$. Setting $\nu_0^1=\nu_0^2$ and multiplying $\textnormal{e}^{-\alpha t}$ in Proposition~\ref{prop-continuousdependence}(ii) yields
\begin{align*}
&d_{\alpha}(\cA[\eta,h](\nu_{\cdot}^1),\cA[\eta,h](\nu_{\cdot}^2))\\
\le&\sup_{t\in\mathcal{T}}L_2\|\nu^1_{\cdot}\|\textnormal{e}^{L_1(\nu^2_{\cdot})t}\int_0^t
d_{\infty}(\nu^1_{\tau},\nu^2_{\tau})\textnormal{e}^{-L_1(\nu^2_{\cdot})\tau}\rd\tau\\
\le& L_2\|\nu^1_{\cdot}\|\sup_{t\in\mathcal{T}}\textnormal{e}^{(L_1(\nu^2_{\cdot})-\alpha)t}\int_0^t
d_{\alpha}(\nu_{\cdot}^1,\nu_{\cdot}^2)
\textnormal{e}^{-(L_1(\nu^2_{\cdot})-\alpha)\tau}\rd\tau\\
\le&\frac{L_2\|\nu^1_{\cdot}\|}{\alpha-L_1(\nu^2_{\cdot})}d_{\alpha}(\nu_{\cdot}^1,\nu_{\cdot}^2)
\le\frac{1}{1+(L_2\|\nu^1_{\cdot}\|)^{-1}}d_{\alpha}(\nu_{\cdot}^1,\nu_{\cdot}^2),
\end{align*}
i.e., $\cA$ is a contraction mapping from $\mathcal{C}(\mathcal{T},\mathcal{B}^M_{\mu_X,1}(X,\cM(X)))$. That the solution $\nu_{\cdot}\in\mathcal{C}(\cT,$ $\mathcal{C}_{\mu_X,1}(X,\cM_+(Y)))$ provided $\nu_0\in \mathcal{C}_{\mu_X,1}(X,\cM_+(Y))$ follows from Proposition~\ref{prop-continuousdependence}(i).

  Next, we prove continuous dependence.
  \begin{enumerate}
  \item[(i)] Continuous dependence on initial conditions.

Substituting $\nu_t^i=\cA[\eta,h](\nu_{\cdot}^i)_t$ into
Proposition~\ref{prop-continuousdependence}(ii) yields
\[d_{\infty}(\nu^1_{t},\nu^2_{t})\le \textnormal{e}^{L_1(\nu^2_{\cdot})t}d_{\infty}(\nu_0^1,\nu_0^2)+ L_2\|\nu^1_{\cdot}\|\textnormal{e}^{L_1(\nu^2_{\cdot})t}\int_0^td_{\infty}(\nu^1_{\tau},\nu^2_{\tau})
  \textnormal{e}^{-L_1(\nu^2_{\cdot})\tau}\rd\tau.\]
Applying Gronwall's inequality again to $d_{\infty}(\nu^1_{t},\nu^2_{t})\textnormal{e}^{-L_1(\nu^2_{\cdot})t}$  yields
\[d_{\infty}(\nu^1_{t},\nu^2_{t})\le \textnormal{e}^{(L_1+L_2\|\nu^1_{\cdot}\|)t}d_{\infty}(\nu_0^1,\nu_0^2),\quad t\in\cT.\]
\item[(ii)] Continuous dependence of solutions of \eqref{Fixed} on $h$.

Let $\nu^i_{\cdot}\in \mathcal{C}(\mathcal{T},\mathcal{B}_{\mu_X,1}(X,\cM_+(Y)))$ for $i=1,2$ such that $\nu^1_0=\nu^2_0$.
    \begin{align*}
      &d_{\sf BL}((\nu_0)_x^1\circ \cS^x_{0,t}[\nu^1_{\cdot},h_1],(\nu^1_0)_x\circ \cS^x_{0,t}[\nu^2_{\cdot},h_2])\\
      \le&d_{\sf BL}((\nu_0)^1\circ \cS^x_{0,t}[\nu^1_{\cdot},h_1],\nu_0^1\circ \cS^x_{0,t}[\nu^1_{\cdot},h_2])+d_{\sf BL}(\nu_0^1\circ \cS^x_{0,t}[\nu^1_{\cdot},h_2],\nu_0^1\circ \cS^x_{0,t}[\nu^2_{\cdot},h_2]).
    \end{align*}
  It follows from \eqref{h-estimate} that
  \begin{align*}
  &d_{\sf BL}((\nu_0^1)_x\circ \cS^x_{0,t}[\nu^1_{\cdot},h_1],(\nu_0^1)_x\circ \cS^x_{0,t}[\nu^1_{\cdot},h_2])\le \|\nu^1_{\cdot}\|\|h_1-h_2\|_{\infty}\textnormal{e}^{L_3t}\int_0^t1\rd\tau.\end{align*}

It suffices to estimate $d_{\sf BL}((\nu_0^1)_x\circ \cS^x_{0,t}[\nu^1_{\cdot},h_2],(\nu_0^1)_x\circ \cS^x_{0,t}[\nu^2_{\cdot},h_2])$, which follows from \eqref{term-1} that:
\[
  d_{\sf BL}((\nu_0)_x^1\circ \cS^x_{0,t}[\nu^1_{\cdot},h_2],(\nu_0)_x^1\circ \cS^x_{0,t}[\nu^2_{\cdot},h_2])\le L_2\|\nu^1_{\cdot}\|\textnormal{e}^{L_1(\nu^2_{\cdot})t}
  \int_0^td_{\infty}(\nu^1_{\tau},\nu^2_{\tau})\textnormal{e}^{-L_1(\nu^2_{\cdot})\tau}\rd\tau.\]
Hence
\begin{align*}
  &d_{\sf BL}((\nu^1_t)_x,(\nu^2_t)_x)\\
  \le& L_2\|\nu^1_{\cdot}\|\textnormal{e}^{L_1(\nu^2_{\cdot})t}\int_0^td_{\infty}(\nu^1_{\tau},\nu^2_{\tau})
  \textnormal{e}^{-L_1(\nu^2_{\cdot})\tau}\rd\tau
  +\|\nu^1_{\cdot}\|\|h_1-h_2\|_{\infty}\textnormal{e}^{L_3t}\int_0^t1\rd\tau.
\end{align*}
Since $L_3-L_1=\BL(h_2)\textnormal{e}^{L_1T}>0$, by Gronwall's inequality,
\begin{align*}
 d_{\sf BL}((\nu^1_t)_x,(\nu^2_t)_x)
 \le&\textnormal{e}^{(L_1+L_2\|\nu^1_{\cdot}\|)t}\|\nu^1_{\cdot}\|\|h_1-h_2\|_{\infty}\textnormal{e}^
 {\BL(h_2)\textnormal{e}^{L_1T}T}\int_0^t\textnormal{e}^{-L_3\tau}\rd \tau,
\end{align*} This shows
\[d_{\infty}(\nu^1_t,\nu^2_t)\le \frac{1}{L_3}\|\nu^1_{\cdot}\|\textnormal{e}^{\BL(h_2)\textnormal{e}^{L_1T}T}\textnormal{e}^{(L_1+L_2
\|\nu^1_{\cdot}\|)t}\|h_1-h_2\|_{\infty}.\]

\item[(iii)] Continuous dependence on $\eta$. Since $\nu_0\in \mathcal{C}_{\mu_X,1}(X,\cM_+(Y))$, we have $\nu_{\cdot}\in\mathcal{C}(\cT,\mathcal{C}_{\mu_X,1}$ $(X,\cM_+(Y)))$. Let $\nu^K_{\cdot}\in \mathcal{C}(\mathcal{T},\mathcal{B}_{\mu_X,1}(X,\cM_+(Y)))$ such that $\nu_0=\nu^K_0$ be the solutions to the fixed point equations
    \[\nu_t=\cA[\eta,h]\nu_t,\quad \nu^K_t=\cA[\eta^K,h]\nu^K_t,\quad t\in\cT.\]
    Assume $$\lim_{K\to\infty}d_{\infty}(\eta^i,\eta^{K,i})=0,\quad i=1,\ldots,r.$$ In the following, we show
    \[\lim_{K\to\infty}d_{\infty}(\cA[\eta,h]\nu_t,\cA[\eta^K,h]\nu^K_t)=0,\quad t\in\cT.\]
By triangle inequality,    \begin{alignat}{2}
\nonumber      d_{\sf BL}((\nu_t)_x,(\nu^K_t)_x)=&d_{\sf BL}((\nu_0)_x\circ \cS^x_{0,t}[\eta,\nu_{\cdot}],(\nu_0)_x\circ \cS^x_{0,t}[\eta^K,\nu^K_{\cdot}])\\
\label{Eq-10}      \le&d_{\sf BL}((\nu_0)_x\circ \cS^x_{0,t}[\eta^K,\nu_{\cdot}],(\nu_0)_x\circ \cS^x_{0,t}[\eta^K,\nu^K_{\cdot}])\\
\nonumber      &+d_{\sf BL}((\nu_0)_x\circ \cS_{0,t}^x[\eta,\nu_{\cdot}],(\nu_0)_x\circ \cS^x_{0,t}[\eta^K,\nu_{\cdot}]).
    \end{alignat}
    From \eqref{term-1} it follows that
   \begin{align}
   \nonumber&d_{\sf BL}((\nu_0)_x\circ \cS^x_{0,t}[\eta^K,\nu_{\cdot}],(\nu_0)_x\circ \cS^x_{0,t}[\eta^K,\nu^K_{\cdot}])\\
   \nonumber\le&\int_Y|\cS_{t,0}^x[\eta^K,\nu_{\cdot}]\phi
   -\cS^x_{t,0}[\eta^K,\nu_{\cdot}^K]\phi|\rd(\nu_0)_x(\phi)=\colon\beta_x(t),\\
   \label{Eq-11}\le& L_{2,K}\|\nu_{\cdot}\|\textnormal{e}^{L_{1,K}(\nu_{\cdot}^K)t}
   \int_0^td_{\infty}(\nu_{\tau},\nu^K_{\tau})\textnormal{e}^{-L_{1,K}(\nu_{\cdot}^K)\tau}\rd\tau,
   \end{align}
    where the index $K$ in the constants indicates the dependence on $K$.

   We now estimate the second term.
\begin{align*}
  &d_{\sf BL}((\nu_0)_x\circ \cS^x_{0,t}[\eta,\nu_{\cdot}],(\nu_0)_x\circ \cS^x_{0,t}[\eta^K,\nu_{\cdot}])\\
  =&\sup_{f\in\mathcal{BL}_1(Y)}\int_Yf\rd((\nu_0)_x\circ \cS^x_{0,t}[\eta,\nu_{\cdot}]-(\nu_0)_x\circ \cS^x_{0,t}[\eta^K,\nu_{\cdot}])\\
  =&\sup_{f\in\mathcal{BL}_1(Y)}\int_Y\lt((f\circ \cS^x_{t,0}[\eta,\nu_{\cdot}])(\phi)-(f\circ \cS^x_{t,0}[\eta^K,\nu_{\cdot}])(\phi)\rt)\rd(\nu_0)_x(\phi)\\
  \le&\int_Y|\cS^x_{t,0}[\eta,\nu_{\cdot}]\phi-\cS^x_{t,0}[\eta^K,\nu_{\cdot}]\phi)|\rd(\nu_0)_x(\phi)
  =\colon\gamma_x(t)\\
  =&\int_Y\lt|\int_0^t(V[\eta,\nu_{\cdot}](\tau,x,\cS^x_{\tau,0}[\eta,\nu_{\cdot}]\phi)
  -V[\eta^K,\nu_{\cdot}](\tau,x,\cS^x_{\tau,0}[\eta^K,\nu_{\cdot}]\phi))
  \mathrm{d}\tau\rt|
  \mathrm{d}(\nu_0)_x(\phi)\\
  \le&\int_Y\lt|\int_0^t(V[\eta,\nu_{\cdot}](\tau,x,\cS^x_{\tau,0}[\eta,\nu_{\cdot}]\phi)
  -V[\eta,\nu_{\cdot}](\tau,x,\cS^x_{\tau,0}[\eta^K,\nu_{\cdot}]\phi))
  \mathrm{d}\tau\rt|
  \mathrm{d}(\nu_0)_x(\phi)\\
  &+\int_Y\lt|\int_0^t(V[\eta,\nu_{\cdot}](\tau,x,\cS^x_{\tau,0}[\eta^K,\nu_{\cdot}]\phi)
  -V[\eta^K,\nu_{\cdot}](\tau,x,\cS^x_{\tau,0}[\eta^K,\nu_{\cdot}]\phi))
  \mathrm{d}\tau\rt|
  \mathrm{d}(\nu_0)_x(\phi)\\
  \le&\int_Y\int_0^t\lt|V[\eta,\nu_{\cdot}](\tau,x,\cS^x_{\tau,0}[\eta,\nu_{\cdot}]\phi)
  -V[\eta,\nu_{\cdot}](\tau,x,\cS^x_{\tau,0}[\eta^K,\nu_{\cdot}]\phi)\rt|
  \mathrm{d}\tau\mathrm{d}(\nu_0)_x(\phi)\\
  &+\int_Y\int_0^t\lt|V[\eta,\nu_{\cdot}](\tau,x,\cS^x_{\tau,0}[\eta^K,\nu_{\cdot}]\phi)
  -V[\eta^K,\nu_{\cdot}](\tau,x,\cS^x_{\tau,0}[\eta^K,\nu_{\cdot}]\phi)\rt|
  \mathrm{d}\tau\mathrm{d}(\nu_0)_x(\phi)\\
  \le&L_1(\nu_{\cdot})\int_0^t\int_Y\lt|\cS^x_{\tau,0}[\eta,\nu_{\cdot}]\phi
  -\cS^x_{\tau,0}[\eta^K,\nu_{\cdot}]\phi\rt|
\mathrm{d}(\nu_0)_x(\phi)\mathrm{d}\tau\\
  &+\int_0^t\int_Y\lt|V[\eta,\nu_{\cdot}](\tau,x,\cS^x_{\tau,0}[\eta^K,\nu_{\cdot}]\phi)
  -V[\eta,\nu_{\cdot}](\tau,x,\cS^x_{\tau,0}[\eta^K,\nu^K_{\cdot}]\phi)\rt|
\mathrm{d}(\nu_0)_x(\phi)\mathrm{d}\tau\\
  +&\int_0^t\int_Y\lt|V[\eta,\nu_{\cdot}](\tau,x,\cS^x_{\tau,0}[\eta^K,\nu^K_{\cdot}]\phi)
  -V[\eta^K,\nu_{\cdot}](\tau,x,\cS^x_{\tau,0}[\eta^K,\nu^K_{\cdot}]\phi)\rt|
  \mathrm{d}(\nu_0)_x(\phi)\mathrm{d}\tau\\
  +&\int_0^t\int_Y\lt|V[\eta^K,\nu_{\cdot}](\tau,x,\cS^x_{\tau,0}[\eta^K,\nu^K_{\cdot}]\phi)
  -V[\eta^K,\nu_{\cdot}](\tau,x,\cS^x_{\tau,0}[\eta^K,\nu_{\cdot}]\phi)\rt|
\mathrm{d}(\nu_0)_x(\phi)\mathrm{d}\tau\\
  \le&L_1(\nu_{\cdot})\int_0^t\gamma_x(\tau)\rd\tau
  +L_1(\nu_{\cdot})\int_0^t\int_Y\lt|\cS^x_{\tau,0}[\eta^K,\nu_{\cdot}]\phi
  -\cS^x_{\tau,0}[\eta^K,\nu^K_{\cdot}]\phi\rt|
\mathrm{d}(\nu_0)_x(\phi)\mathrm{d}\tau\\
&+\int_0^t\int_Y\lt|V[\eta,\nu_{\cdot}](\tau,x,\phi)
  -V[\eta^K,\nu_{\cdot}](\tau,x,\phi)\rt|
\mathrm{d}(\nu^K_{\tau})_x(\phi)\mathrm{d}\tau\\
&+L_{1,K}(\nu_{\cdot})\int_0^t\int_Y\lt|\cS^x_{\tau,0}[\eta^K,\nu^K_{\cdot}]\phi
-\cS^x_{\tau,0}[\eta^K,\nu_{\cdot}]\phi\rt|
\mathrm{d}(\nu_0)_x(\psi)\mathrm{d}\tau \\
=&L_1(\nu_{\cdot})\int_0^t\gamma_x(\tau)\rd\tau+(
L_1(\nu_{\cdot})+L_{1,K}(\nu_{\cdot}))\int_0^t\beta_x(\tau)\rd\tau\\
&+\int_0^t\int_Y\lt|V[\eta,\nu_{\cdot}](\tau,x,\phi)
  -V[\eta^K,\nu_{\cdot}](\tau,x,\phi)\rt|
\mathrm{d}(\nu^K_{\tau})_x(\phi)\mathrm{d}\tau.
\end{align*}
To obtain further estimate, let $$\zeta^K_x(\tau)\colon=\int_Y\lt|V[\eta,\nu_{\cdot}](\tau,x,\phi)
  -V[\eta^K,\nu_{\cdot}](\tau,x,\phi)\rt|
\mathrm{d}(\nu_{\tau})_x(\phi).$$ By triangle inequality,
\begin{align*}
&\int_0^t\int_Y\lt|V[\eta,\nu_{\cdot}](\tau,x,\phi)
  -V[\eta^K,\nu_{\cdot}](\tau,x,\phi)\rt|
\mathrm{d}(\nu^K_{\tau})_x(\phi)\mathrm{d}\tau\\
\le& \int_0^t\zeta^K_x(\tau)\mathrm{d}\tau\\
&+\int_0^t\lt|\int_Y\lt|V[\eta,\nu_{\cdot}](\tau,x,\phi)
  -V[\eta^K,\nu_{\cdot}](\tau,x,\phi)\rt|
\mathrm{d}((\nu^K_{\tau})_x(\phi)-(\nu_{\tau})_x(\phi))\rt|\mathrm{d}\tau.
\end{align*}
Recall that $$\sup_{x\in X}\eta^{K,i}_x(Y)\le \sup_{x\in X}\eta^{i}_x(Y)+d_{\infty}(\eta^i,\eta^{K,i}),\quad i=1,\ldots,r.$$ This shows that there exists some $b>0$ independent of $K$ such that $$\sup_{K}(L_1(\eta,\nu_{\cdot})+L_{1,K}(\eta^K,\nu_{\cdot})),\quad \sup_{K}(L_2(\eta)+L_2(\eta^K))\le b,$$ since $$\sum_{i=1}^r(|\|\eta^i\|-\|\eta^{K,i}\||)\le\sum_{i=1}^rd_{\infty}(\eta^i,\eta^{K,i})\to0,\quad \text{as}\quad K\to\infty.$$ Moreover, it follows from Proposition~\ref{Vlasovf}(ii) that \begin{align*}
  &\Bigl|\bigl|V[\eta,\nu_{\cdot}](\tau,x,\varphi)
  -V[\eta^K,\nu_{\cdot}](\tau,x,\varphi)\bigr|-\bigl|V[\eta,\nu_{\cdot}](\tau,x,\phi)
  -V[\eta^K,\nu_{\cdot}](\tau,x,\phi)\bigr|\Bigr|\\
  \le&\lt|V[\eta,\nu_{\cdot}](\tau,x,\varphi)
  -V[\eta,\nu_{\cdot}](\tau,x,\phi)\rt|+\lt|V[\eta^K,\nu_{\cdot}](\tau,x,\varphi)
  -V[\eta^K,\nu_{\cdot}](\tau,x,\phi)\rt|\\
  \le&(L_1+L_{1,K})|\varphi-\phi|\le b|\varphi-\phi|.
\end{align*}
Further, by Proposition\ref{Vlasovf}(ii), one can show that
$$f_K(\tau,x,\varphi)\colon=\lt|V[\eta,\nu_{\cdot}](\tau,x,\varphi)
  -V[\eta^K,\nu_{\cdot}](\tau,x,\varphi)\rt|$$ is bounded Lipschitz in $\varphi$ with some constant $\widehat{b}>0$ such that $$\sup_{K\in\N}\sup_{\tau\in\cT}\sup_{x\in X}\BL(f_K(\tau,x,\cdot))\le \widehat{b}.$$ Hence
  \begin{align*}
  &\Bigl|\int_0^t\int_Y\lt|V[\eta,\nu_{\cdot}](\tau,x,\phi)
  -V[\eta^K,\nu_{\cdot}](\tau,x,\phi)\rt|
\mathrm{d}((\nu^K_{\tau})_x(\phi)-(\nu_{\tau})_x(\phi))\mathrm{d}\tau\Bigr|\\
\le& \widehat{b}\int_0^td_{\sf BL}((\nu^K_{\tau})_x,(\nu_{\tau})_x)\rd\tau.\end{align*}

This further implies that
\begin{align*}
\gamma_x(t)\le& L_1\int_0^t\gamma_x(\tau)\rd\tau+b\int_0^t\beta_x(\tau)\rd\tau
+\widehat{b}\int_0^td_{\sf BL}((\nu^K_{\tau})_x,(\nu_{\tau})_x)\rd\tau+\int_0^t\zeta^K_x(\tau)\rd\tau.
\end{align*}
By Gronwall's inequality, we have
\begin{align*}
  \gamma_x(t)\le \textnormal{e}^{L_1t}\lt(b\int_0^t\beta_x(\tau)\rd\tau
  +\widehat{b}\int_0^td_{\sf BL}((\nu^K_{\tau})_x,(\nu_{\tau})_x)\rd\tau
  +\int_0^t\zeta^K_x(\tau)\rd\tau\rt)
  \end{align*}
Hence by \eqref{Eq-10}, \eqref{Eq-11} and monotonicity of $\int_0^t\textnormal{e}^{-L_{1,K}\tau}d_{\infty}(\nu_{\tau},\nu^K_{\tau})
(\nu_{\tau}^K)_x(Y)\rd\tau$ in $t$, we have for $t\in\cT$,
\begin{align*}
&d_{\sf BL}((\nu_t)_x,(\nu^K_t)_x)\le\beta_x(t)+\gamma_x(t)\\
\le& \beta_x(t)+\textnormal{e}^{L_1t}\lt(b\int_0^t\beta_x(\tau)\rd\tau
+\widehat{b}\int_0^td_{\sf BL}((\nu^K_{\tau})_x,(\nu_{\tau})_x)\rd\tau
+\int_0^t\zeta^K_x(\tau)\rd\tau\rt)\\
\le&
L_{2,K}\|\nu_{\cdot}\|\int_0^t\textnormal{e}^{L_{1,K}(t-\tau)}d_{\infty}(\nu_{\tau},\nu^K_{\tau})
(\nu_{\tau}^K)_x(Y)\rd\tau\\
&+\textnormal{e}^{L_1t}b\int_0^t\beta_x(\tau)\rd\tau+\widehat{b}\textnormal{e}^{L_1t}\int_0^td_{\sf BL}((\nu^K_{\tau})_x,(\nu_{\tau})_x)\rd\tau
+\textnormal{e}^{L_1t}\int_0^t\zeta^K_x(\tau)\rd\tau\\
\le&(b-L_2)\|\nu_{\cdot}\|\int_0^t\textnormal{e}^{(b-L_1)(t-\tau)}d_{\infty}(\nu_{\tau},\nu^K_{\tau})
(\nu_{\tau}^K)_x(Y)\rd\tau\\
&+\textnormal{e}^{L_1t}b\int_0^t(b-L_2)\|\nu_{\cdot}\|\int_0^{\tau}\textnormal{e}^{(b-L_1)(\tau-s)}
d_{\infty}(\nu_{s},\nu^K_{s})
(\nu_{s}^K)_x(Y)\rd s\rd\tau\\
&+\widehat{b}\textnormal{e}^{L_1t}\int_0^td_{\sf BL}((\nu^K_{\tau})_x,(\nu_{\tau})_x)\rd\tau
+\textnormal{e}^{L_1t}\int_0^t\zeta^K_x(\tau)\rd\tau\\
\le&(b-L_2)\|\nu_{\cdot}\|(1+\textnormal{e}^{L_1t}bt)\int_0^t\textnormal{e}^{(b-L_1)(t-\tau)}d_{\infty}
(\nu_{\tau},\nu^K_{\tau})
(\nu_{\tau}^K)_x(Y)\rd\tau\\
&+\widehat{b}\textnormal{e}^{L_1t}\int_0^td_{\sf BL}((\nu^K_{\tau})_x,(\nu_{\tau})_x)\rd\tau
+\textnormal{e}^{L_1t}\int_0^t\zeta^K_x(\tau)\rd\tau\\
   \le&(b-L_2)\|\nu_{\cdot}\|(1+bT)\textnormal{e}^{L_1t}\int_0^t\textnormal{e}^{(b-L_1)(t-\tau)}d_{\infty}
   (\nu_{\tau},\nu^K_{\tau})(\nu_{\tau})_x(Y)\rd\tau\\
&+(b-L_2)\|\nu_{\cdot}\|(1+bT)\textnormal{e}^{L_1t}\int_0^t\textnormal{e}^{(b-L_1)(t-\tau)}
d_{\infty}(\nu_{\tau},\nu^K_{\tau})|(\nu^K_{\tau})_x(Y)-(\nu_{\tau})_x(Y)|\rd\tau\\
&+\widehat{b}\textnormal{e}^{L_1t}\int_0^td_{\sf BL}((\nu^K_{\tau})_x,(\nu_{\tau})_x)\rd\tau
+\textnormal{e}^{L_1t}\int_0^t\zeta^K_x(\tau)\rd\tau\\
   \le& (b-L_2)\|\nu_{\cdot}\|(1+bT)\textnormal{e}^{(2b-L_1)t}\int_0^t
   \textnormal{e}^{-2(b-L_1)\tau}d_{\infty}(\nu_{\tau},\nu^K_{\tau})d_{\sf BL}((\nu_{\tau})_x,(\nu^K_{\tau})_x)\rd\tau\\
   &+(b-L_2)\|\nu_{\cdot}\|(1+bT)\textnormal{e}^{bt}\int_0^t\textnormal{e}^{-(b-L_1)\tau}
   d_{\infty}(\nu_{\tau},\nu^K_{\tau})\rd\tau
   +\textnormal{e}^{L_1t}\int_0^t\zeta^K_x(\tau)\rd\tau,
\end{align*}
This further shows
\begin{align*}
  w_K(t)\le L_4\int_0^tw_K(\tau)^2+L_4w_K(\tau)+c_K,
\end{align*}
where $L_4=(b-L_2)\|\nu_{\cdot}\|(1+bT)\textnormal{e}^{bT}$, $w_K(t)\colon=\textnormal{e}^{-(b-L_1)t}d_{\infty}(\nu_t,\nu^K_t)$ is continuous by Proposition~\ref{prop-nu}, and $c_K=\textnormal{e}^{\max\{0,2L_1-b\}T}\sup_{x\in X}\int_0^T\zeta^K_x(\tau)\rd\tau$.
Applying Lemma~\ref{le-gronwall} to $w_K(t)$ yields
\[w_K(t)\le c_K\frac{\textnormal{e}^{L_4t}}{1+c_K(1-\textnormal{e}^{L_4t})},\quad t\in\cT,\]
provided $c_K<\frac{1}{\textnormal{e}^{L_4T}-1}$.

The rest is to show $\lim_{K\to\infty}c_K=0$.
For every $t\in\cT$, define $\widehat{\nu}_t\equiv\sup_{x\in X}(\nu_t)_x$:
\[\widehat{\nu}_t(E)=\sup_{x\in X}(\nu_t)_x(E),\quad \forall E\in\mathcal{B}(Y).\] Since $\nu_t\in \mathcal{B}(X,\cM_+(Y))$, it is easy to show that $\widehat{\nu}_t\in\cM_+(Y)$.
Hence \begin{align*}
\sup_{x\in X}\zeta_x^K(\tau)=&\int_Y\lt|V[\eta,\nu_{\cdot}](\tau,x,\phi)
  -V[\eta^K,\nu_{\cdot}](\tau,x,\phi)\rt|
\mathrm{d}(\nu_{\tau})_x(\phi)\\
\le&\sup_{x\in X}\int_Y\lt|V[\eta,\nu_{\cdot}](\tau,x,\phi)
  -V[\eta^K,\nu_{\cdot}](\tau,x,\phi)\rt|
\mathrm{d}\widehat{\nu}_{\tau}(\phi)\\
\le&\int_Y\sup_{x\in X}\lt|V[\eta,\nu_{\cdot}](\tau,x,\phi)
  -V[\eta^K,\nu_{\cdot}](\tau,x,\phi)\rt|
\mathrm{d}\widehat{\nu}_{\tau}(\phi),\end{align*}
by Proposition~\ref{Vlasovf}(v) (note that $\nu_{\tau}\in\mathcal{C}_{\mu_X,1}(X,\cM_+(Y))$), we have \begin{align*}
&\sup_{x\in X}\int_0^T\zeta_x^K(\tau)\rd\tau\le\int_0^T\sup_{x\in X}\zeta_x^K(\tau)\rd\tau\\
\le&\int_0^T\int_Y\sup_{x\in X}\lt|V[\eta,\nu_{\cdot}](\tau,x,\psi)
  -V[\eta^K,\nu_{\cdot}](\tau,x,\psi)\rt|
\mathrm{d}\widehat{\nu}_{\tau}(\psi)\rd\tau\to0,\quad \text{as}\ K\to\infty,
\end{align*}
i.e., $\lim_{K\to\infty}c_K=0$. Hence $\lim_{K\to\infty}\sup_{t\in\cT}w_K(t)=0$, i.e., $$\lim_{K\to\infty}\sup_{t\in\cT}d_{\infty}(\nu_t,\nu^K_t)=0.$$
\end{enumerate}
\end{proof}

\section{Proof of Lemma~\ref{le-partition}}\label{appendix-le-partition}
\begin{proof}
By ($\mathbf{A1}$), $X\subseteq\R^{r_1}$ is compact, there exists $R>0$ such that $X\subseteq[-R,R]^{r_1}$. Let $(B_i^{m})_{1\le i\le \lfloor m^{1/r_1}\rfloor^{r_1}}$ be the equipartition of $[-R,R]^{r_1}$ into $\lfloor m^{1/r_1}\rfloor^{r_1}$ copies of small cubes with $\Diam(B_i^m)=\frac{2r_1R}{\lfloor m^{1/r_1}\rfloor}$. Since $\lfloor m^{1/r_1}\rfloor^{r_1}\le m$, one can further partition some of these $B^m_i$ so that one obtains a possibly finer partition $(\widehat{B}^m_i)_{1\le i\le m}$ of $[-R,R]^{r_1}$ with $\Diam(\widehat{B}_i^m)\le\frac{2r_1R}{\lfloor m^{1/r_1}\rfloor}$. Let $A_i^m=\widehat{B}^m_i\cap X$ for $i=1,\ldots,m$. Then $$\Diam(A_i^m)\le\frac{2r_1R}{\lfloor m^{1/r_1}\rfloor}\to0,\quad \text{as}\ m\to\infty.$$
\end{proof}
\section{Proof of Lemma~\ref{le-ini-2}}\label{appendix-le-ini-2}
\begin{proof}
Since $\nu_0\in \mathcal{C}_{\mu_X,1}(X,\cM_+(Y))$, by Proposition~\ref{prop-fibercomplete}(i) we have $(\nu_0)_x(Y)<\infty$ for all $x\in X$.

First, it is easy to verify that
\begin{align*}
  \int_X(\nu_0^{m,n})_x(Y)\rd\mu_X(x)=&\sum_{i=1}^m\int_{A^m_i}(\nu_0^{m,n})_x(Y)\mu_X(x)
  =\sum_{i=1}^ma_{m,i}\mu_X(A^m_i)\\
  =&\sum_{\mu_X(A_i^m)>0}\int_{A_i^m}(\nu_0)_x(Y)\rd\mu_X(x)
  =\int_X(\nu_0)_x(Y)\rd\mu_X(x)=1.
\end{align*}
Since $\nu_0\in\mathcal{C}_{\mu_X,1}(X,\cM_+(Y))$, we have $\nu_0^{m,n}\in\mathcal{B}_{\mu_X,1}(X,\cM_+(Y))$.

By Proposition~\ref{prop-XBC}, for every $m\in\N$ and $i=1,\ldots,m$, there exists  $\{\varphi^{m,n}_{(i-1)n+j}\colon j=1,\ldots,n\}\subseteq Y$ such that for $1\le i\le m$ with $\mu_X(A_i^m)>0$, $$\lim_{n\to\infty}d_{\sf BL}\left(\frac{(\nu_0)_{x_i^m}}{\widehat{a}_{m,i}},
\frac{1}{n}\sum_{j=1}^n\delta_{\varphi^{m,n}_{(i-1)n+j}}\right)=0,$$
where $\widehat{a}_{m,i}=\nu_{x^m_i}(Y)$.
Define $\widehat{\nu}_0^{m,n}\in\mathcal{B}(X,\cM_+(Y))$ as follows: $$(\widehat{\nu}_0^{m,n})_x\colon=\sum_{i=1}^m\mathbbm{1}_{A^m_i}(x)
\frac{\widehat{a}_{m,i}}{n}\sum_{j=1}^n\delta_{\varphi^{m,n}_{(i-1)n+j}},\quad x\in X.$$
Next, we show $$\lim_{m\to\infty}\lim_{n\to\infty}d_{\infty}(\nu_0^{m,n},\nu_0)=0.$$ It then suffices to show that for every $m\in\N$,
\begin{equation}\label{Eq-auxiliary-1}
\lim_{m\to\infty}\lim_{n\to\infty}d_{\infty}(\widehat{\nu}_0^{m,n},\nu_0)=0\end{equation}
and \begin{equation}\label{Eq-auxiliary-2}
\lim_{m\to\infty}\lim_{n\to\infty}d_{\infty}(\nu_0^{m,n},\widehat{\nu}_0^{m,n})=0.\end{equation}

We first show \eqref{Eq-auxiliary-1}.

By definition, $$(\widehat{\nu}_0^{m,n})_y\equiv(\widehat{\nu}_0^{m,n})_{x^m_i},\quad y\in A^m_i,$$
which implies that for $y\in A^m_i$,
\begin{align*}
 d_{\infty}(\nu_0,\widehat{\nu}_0^{m,n})=&\sup_{y\in X}d_{\sf BL}((\nu_0)_{y},(\widehat{\nu}_0^{m,n})_{y})\\
 \le&\max_{1\le i\le m}\left(d_{\sf BL}((\nu_0)_{y},(\nu_0)_{x_i^m})
  +\widehat{a}_{m,i}d_{\sf BL}\Bigl(\frac{(\nu_0)_{x^m_i}}{\widehat{a}_{m,i}},
  \frac{(\nu_0^{m,n})_{x^m_i})}{\widehat{a}_{m,i}}\Bigr)\right).
\end{align*}

Since $x\mapsto(\nu_0)_x$ is continuous and $X$ is compact, we have $x\mapsto(\nu_0)_x$ is uniformly continuous. Due to this uniform continuity of $x\mapsto(\nu_0)_x$ as well as $$\lim_{m\to\infty}\max\limits_{1\le i\le m}\Diam A^m_i=0,$$ we have
\[\lim_{m\to\infty}\lim_{n\to\infty}\max_{1\le i\le m}d_{\sf BL}((\nu_0)_{y},(\nu_0)_{x_i^m})=\lim_{m\to\infty}\max_{1\le i\le m}d_{\sf BL}((\nu_0)_{y},(\nu_0)_{x_i^m})=0.\]

Moreover,
\begin{align*}
  &\max_{1 i\le m}\widehat{a}_{m,i}d_{\sf BL}\Bigl(\frac{(\nu_0)_{x^m_i}}{\widehat{a}_{m,i}},
  \frac{(\nu_0^{m,n})_{x^m_i})}{\widehat{a}_{m,i}}\Bigr)\le\sup_{y\in X}\nu_{y}(Y)d_{\sf BL}\Bigl(\frac{(\nu_0)_{x^m_i}}{\widehat{a}_{m,i}},
  \frac{(\nu_0^{m,n})_{x^m_i})}{\widehat{a}_{m,i}}\Bigr),
\end{align*}
which yields that
\[\lim_{m\to\infty}\lim_{n\to\infty}\max_{1\le i\le m}\widehat{a}_{m,i}d_{\sf BL}\Bigl(\frac{(\nu_0)_{x^m_i}}{\widehat{a}_{m,i}},
  \frac{(\nu_0^{m,n})_{x^m_i})}{\widehat{a}_{m,i}}\Bigr)=\lim_{m\to\infty}0=0.\]
Hence
\[\lim_{m\to\infty}\lim_{n\to\infty}d_{\infty}(\nu_0,\widehat{\nu}_0^{m,n})=0.\]

Since $1\in\mathcal{BL}_1(Y)$,  by the uniform continuity of $x\mapsto(\nu_0)_x$,
\begin{align*}
  &|(\nu_0)_x(Y)-(\nu_0)_y(Y)|\\
  =&|\int_Y1\rd((\nu_0)_x-(\nu_0)_y|\le d_{\sf BL}((\nu_0)_x,(\nu_0)_y)\to0,\quad \text{as}\quad |x-y|\to0,\quad \forall x,y\in X.
\end{align*}
This implies that  $x\mapsto(\nu_0)_x(Y)$ is uniformly continuous in $x$.

Now we show \eqref{Eq-auxiliary-2}. By the definition of $\nu_0^{m,n}$ and $\widehat{\nu}_0^{m,n}$, for $x\in A^m_i$ with $\mu_X(A^m_i)>0$,
\begin{align*}
  d_{\sf BL}((\nu_0^{m,n})_x,(\widehat{\nu}_0^{m,n})_x)
  =&\sup_{f\in\mathcal{BL}_1(Y)}\int_Yf\rd((\nu_0^{m,n})_x-(\widehat{\nu}_0^{m,n})_x)\\
  =&\sup_{f\in\mathcal{BL}_1(Y)}\frac{1}{n}\sum_{j=1}^nf(\varphi^{m,n}_{(i-1)n+j})(a_{m,i}-\widehat{a}_{m,i})\\
  \le&|a_{m,i}-\widehat{a}_{m,i}|=\left|\frac{\int_{A_i^m}(\nu_0)_x(Y)\rd\mu_X(x)}{\mu_X(A_i^m)}
  -\nu_{x^m_i}(Y)\right|\\
  \le&{\int_{A_i^m}|(\nu_0)_x(Y)-(\nu_0)_{x^m_i}(Y)|\rd\mu_X(x)}{\mu_X(A_i^m)}\\
  \le&\sup_{x\in A_i^m}|(\nu_0)_x(Y)-(\nu_0)_{x^m_i}(Y)|,
\end{align*}

For $x\in A^m_i$ with $\mu_X(A^m_i)=0$, $$d_{\sf BL}((\nu_0^{m,n})_x,(\widehat{\nu}_0^{m,n})_x)\equiv0.$$

Since $x\mapsto(\nu_0)_x(Y)$ is uniformly continuous in $x$ and $\max\limits_{1\le i\le m}\Diam A_i^m\to0$ as $m\to\infty$, we have
\[d_{\infty}(\nu_0^{m,n},\widehat{\nu}_0^{m,n})=\sup_{x\in X}d_{\sf BL}((\nu_0^{m,n})_x,(\widehat{\nu}_0^{m,n})_x)\le\max_{1\le i\le m}\sup_{x\in A_i^m}|(\nu_0)_x(Y)-(\nu_0)_{x^m_i}(Y)|\to0,\]
i.e., \eqref{Eq-auxiliary-2} holds.
\end{proof}
\section{Proof of Lemma~\ref{le-graph}}\label{appendix-le-graph}
\begin{proof}
The proof is analogous to that of Lemma~\ref{le-ini-2} by simply replacing $Y$ in Lemma~\ref{le-ini-2} by $X$.
\end{proof}

\section{Proof of Lemma~\ref{le-h}}\label{appendix-le-h}
\begin{proof}
Since $X$ is compact, $\mathbf{(A3)}$ and $\mathbf{(A7)}$ imply that $h$ is uniformly continuous in $x$. Since $$\lim_{m\to\infty}\max_{1\le i\le m}\Diam A^m_i=0,$$ for every $\varepsilon>0$, there exists $\sigma>0$ such that
$\max_{1\le i\le m}\Diam A^m_i<\sigma$ for all $m\ge M$
for some $M\in \N$, and  for $t\in\cT$, $\phi\in Y$, $$|h(t,x,\phi)-h(t,y,\phi)|<\varepsilon,\quad x,y\in X,\quad |x-y|<\sigma.$$ Hence for every $z\in X$, we have $z\in A^m_i$ for some $1\le i\le m$, and $|z-x_i^m|\le\Diam A^m_i<\sigma$, for all $m\ge M$. Then for all $m\ge M$,
\begin{align*}
  |h^m(t,z,\psi)-h(t,z,\psi)|=&|h(t,x_i^m,\psi)-h(t,z,\psi)|<\varepsilon,
 \end{align*} which implies that
\[\sup_{x\in X}|h^m(t,z,\psi)-h(t,z,\psi)|\le\varepsilon,\quad m\ge M,\]
since $z$ is independent of $\varepsilon$.
 This shows
  \[\lim_{m\to\infty}\sup_{z\in X}|h^m(t,z,\phi)-h(t,z,\phi)|=0,\quad t\in\cT,\ \phi\in Y.\]
  Moreover, by the definition of $h^m$, $$\sup_{z\in X}|h^m(t,z,\phi)-h(t,z,\phi)|\le 2\sup_{z\in X}|h(t,z,\phi)|,\quad t\in\cT,\ \phi\in Y$$ which yields the conclusion by the Dominated Convergence Theorem as well as integrability of $h$ in $(\mathbf{A7})$.
\end{proof}
\section{Proof of Proposition~\ref{prop-well-posed-lattice}}\label{appendix-prop-well-posed-lattice}
\begin{proof}
Analogous to the proof of Proposition~\ref{Vlasovf}, one can show that the vector field of \eqref{lattice} is Lipschitz continuous, and hence unique existence of a local solution to \eqref{lattice} is obtained by Picard-Lindel\"{o}f's iteration \cite[Theorem~2.2]{T12}. The rest is the same as in the proof of Proposition~\ref{Vlasovf}, since $Y$ is positively invariant owing to ($\mathbf{A6}$).
\end{proof}

\section{A quadratic Gronwall inequality}
\begin{lemma}\label{le-gronwall}
Let $a, b, c>0$ and  $w\in \mathcal{C}(\cT,\R_+)$. Assume
\begin{equation}\label{Eq-12}
  c<\frac{b}{a(\textnormal{e}^{bT}-1)}.
\end{equation}
If $w$ satisfies
\begin{equation}
  \label{Eq-13}
  w(t)\le a\int_0^tw(s)^2\rd s+b\int_0^tw(s)\rd s+c,\quad t\in\cT,
\end{equation}
then
\[w(t)\le\frac{bc\textnormal{e}^{bt}}{b+ac(1-\textnormal{e}^{bt})},\quad t\in\cT.\]
\end{lemma}
\begin{proof}
  Let $u(t)=\int_0^tw(s)^2\rd s+b\int_0^tw(s)\rd s+c$, for $t\in\cT$. Then $u(t)\ge c$ for all $t\in\cT$ since $w(t)\ge0$. It follows from \eqref{Eq-13} that for all $t\in\cT$, $w(t)\le u(t)$ and
  \begin{equation}
    \label{Eq-14}
    u'(t)\le au(t)^2+bu(t).
  \end{equation}
  Let $v(t)=-u(t)^{-1}$. It is easy to verify that
  $$v'(t)\le a-bv(t),\quad t\in\cT.$$
  Using the auxiliary function $h(t)=\textnormal{e}^{bt}v(t)$, one can show that
   \[v(t)\le \textnormal{e}^{-bt}(v(0)+\frac{a}{b}(\textnormal{e}^{bt}-1)),\]
   i.e.,
   \[-u(t)^{-1}\le \textnormal{e}^{-bt}(-c^{-1}+\frac{a}{b}(\textnormal{e}^{bt}-1)).\]
   It then follows from \eqref{Eq-12} that
   \[w(t)\le u(t)\le\frac{bc\textnormal{e}^{bt}}{b+ac(1-\textnormal{e}^{bt})}.\]
\end{proof}
\end{document}